\documentclass[11pt,reqno]{amsart}

\usepackage{amsmath}
\usepackage{amssymb}
\usepackage{amsthm}
\usepackage{amsfonts}
\usepackage{upref}
\usepackage{xspace}

\usepackage[paperwidth=20.5cm,paperheight=29.7cm,width=16cm,height=25.5cm,top=24mm,left=19mm,headsep=11mm]{geometry}

\theoremstyle{plain}
\newtheorem{theorem}{Theorem}[section]
\newtheorem{lemma}[theorem]{Lemma}
\newtheorem{proposition}[theorem]{Proposition}
\newtheorem{corollary}[theorem]{Corollary}
\newtheorem{condition}[theorem]{Condition}

\theoremstyle{definition}
\newtheorem{remark}[theorem]{Remark}

\theoremstyle{plain}
\newtoks\thehProclaim

\theoremstyle{definition}
\newtoks\thehDefinition
\newtheorem*{Definition}{\the\thehDefinition}

\theoremstyle{definition}
\newtoks{\thehRemark}
\newtheorem*{Remark}{\the\thehRemark}

\numberwithin{equation}{section}

\theoremstyle{plain}
\newtoks\thehProclaim
\newtheorem*{Proclaim}{\the\thehProclaim}

\numberwithin{equation}{section}

\def\C{\mathbb C}

\def\x{\mathbf x}

\def\D{\mathbf D}

\def\1{\mathbf 1}

\def\1{\bold 1}

\def\div{\mathrm{div}\,}

\def\eps{\varepsilon}

\begin{document}

\title[Homogenization of the first initial boundary-value problem]{Homogenization of the first initial boundary-value problem for parabolic systems: operator error estimates}

\author{ Yu. M. Meshkova}

\address{Chebyshev Laboratory\\
St. Petersburg State University\\
14 Liniya V.O.,  29b\\
199178, St.~Petersburg,
Russia}

\email{y.meshkova@spbu.ru}

\author{ T. A. Suslina}

\address{Department of Physics\\
St. Petersburg State University\\
Ul'yanovskaya 3, Petrodvorets \\198504, St.~Petersburg,
Russia}

\email{t.suslina@spbu.ru}


\thanks{Supported by Russian Foundation for Basic Research (grant no.~16-01-00087).
The first author was supported by ``{Native Towns}'', a social investment program of PJSC ``{Gazprom Neft}'', by the ``{Dynasty}'' \ foundation, and by the Rokhlin grant.}

\date{\today}


\begin{abstract}
Let $\mathcal{O}\subset\mathbb{R}^d$ be a bounded domain of class $C^{1,1}$. In $L_2(\mathcal{O};\mathbb{C}^n)$, we consider a selfadjoint matrix second order elliptic differential operator $B_{D,\varepsilon}$, $0<\varepsilon\leqslant1$, with the Dirichlet boundary condition.
The principal part of the operator is given in a factorized form. The operator involves first and zero order terms.
The operator $B_{D,\varepsilon}$ is positive definite; its coefficients are periodic and depend on $\mathbf{x}/\varepsilon$.
We study the behavior of the operator exponential $e^{-B_{D,\varepsilon}t}$, $t>0$, as $\varepsilon\rightarrow 0$.
We obtain approximations for the exponential $e^{-B_{D,\varepsilon}t}$ in the operator norm on $L_2(\mathcal{O};\mathbb{C}^n)$ and in the norm of
operators acting from $L_2(\mathcal{O};\mathbb{C}^n)$ to the Sobolev space $H^1(\mathcal{O};\mathbb{C}^n)$.
The results are applied to homogenization of solutions of the first initial boundary-value problem for parabolic systems.
\end{abstract}

\maketitle

\section*{Introduction}

The paper concerns homogenization theory of periodic differential operators (DO's).
We mention the books on homogenization \cite{BaPa,BeLP,ZhKO,Sa}.

\subsection{Statement of the problem} Let $\Gamma\subset\mathbb{R}^d$ be a lattice and let $\Omega$ be the elementary cell of the lattice $\Gamma$. For a $\Gamma$-periodic function $\psi$ in $\mathbb{R}^d$, we denote $\psi ^\varepsilon (\mathbf{x}):=\psi (\mathbf{x}/\varepsilon)$, where $\varepsilon >0$, and $\overline{\psi}:=\vert \Omega\vert ^{-1}\int_\Omega \psi (\mathbf{x})\,d\mathbf{x}$.

Let $\mathcal{O}\subset\mathbb{R}^d$ be a bounded domain of class $C^{1,1}$. In $L_2(\mathcal{O};\mathbb{C}^n)$, we study a selfadjoint matrix strongly elliptic second order DO ${B}_{D,\varepsilon}$, $0<\varepsilon\leqslant 1$,  with the Dirichlet boundary condition.
The principal part of the operator ${B}_{D,\varepsilon}$ is given in a factorized form
$A_{\varepsilon}=b(\mathbf{D})^*g^\varepsilon (\mathbf{x})b(\mathbf{D})$,
where $b(\mathbf{D})$ is a matrix homogeneous first order DO, and $g(\mathbf{x})$ is a $\Gamma$-periodic bounded and positive definite matrix-valued function in  $\mathbb{R}^d$. (The precise assumptions on $b(\mathbf{D})$ and $g(\mathbf{x})$ are given below in Subsection~\ref{Subsection operatoer A_D,eps}.) The operator ${B}_{D,\varepsilon}$ is given by the differential expression
\begin{equation}
\label{B_D,eps in introduction}
{B}_{\varepsilon}=b(\mathbf{D})^* g^\varepsilon (\mathbf{x}) b(\mathbf{D})
+\sum_{j=1}^d\bigl(a_j^\varepsilon (\mathbf{x})D_j+D_ja_j^\varepsilon(\mathbf{x})^*\bigr)
+Q^\varepsilon (\mathbf{x}) +\lambda Q_0^\varepsilon (\mathbf{x})
\end{equation}
with the Dirichlet condition on $\partial\mathcal{O}$.
Here $a_j(\mathbf{x})$, $j=1,\dots,d$, and $Q(\mathbf{x})$ are $\Gamma$-periodic matrix-valued functions, in general, unbounded;
a $\Gamma$-periodic matrix-valued function $Q_0(\mathbf{x})$ is such that $Q_0(\mathbf{x})>0$ and $Q_0, Q_0^{-1}\in L_\infty$.
The constant $\lambda$ is chosen so that the operator $B_{D,\varepsilon}$ is positive definite.
(The precise assumptions on the coefficients are given below in Subsection~\ref{Subsection lower order terms}.)

The coefficients of the operator \eqref{B_D,eps in introduction} oscillate rapidly for small $\varepsilon$.
Let $\mathbf{u}_\varepsilon (\mathbf{x},t)$ be the solution of the first initial boundary-value problem:
\begin{equation}
\label{intr problem}
\begin{cases}
Q_0^\varepsilon (\mathbf{x})\partial _t \mathbf{u}_\varepsilon (\mathbf{x},t)=-B_{\varepsilon}\mathbf{u}_\varepsilon (\mathbf{x},t),\quad\mathbf{x}\in\mathcal{O},\;t>0;
\\
\mathbf{u}_\varepsilon (\mathbf{x},t)=0,\quad\mathbf{x}\in\partial\mathcal{O},\;t>0;\quad Q_0^\varepsilon (\mathbf{x})\mathbf{u}_\varepsilon (\mathbf{x},0)=\boldsymbol{\varphi}(\mathbf{x}),\quad\mathbf{x}\in\mathcal{O},
\end{cases}
\end{equation}
where $\boldsymbol{\varphi}\in L_2(\mathcal{O};\mathbb{C}^n)$. We are interested in the behavior of the solution in the small period limit.

\subsection{Main results} It turns out that, as $\varepsilon\rightarrow 0$, the solution $\mathbf{u}_\varepsilon (\,\cdot\, ,t)$ converges in  $L_2(\mathcal{O};\mathbb{C}^n)$ to the solution $\mathbf{u}_0(\,\cdot\, ,t)$ of the effective problem with constant coefficients:
\begin{equation}
\label{intr eff problem}
\begin{cases}
\overline{Q_0}\partial _t \mathbf{u}_0 (\mathbf{x},t)=-B^0\mathbf{u}_0 (\mathbf{x},t),\quad\mathbf{x}\in\mathcal{O},\;t>0;
\\
\mathbf{u}_0 (\mathbf{x},t)=0,\quad\mathbf{x}\in\partial\mathcal{O},\;t>0;\quad \overline{Q_0}
\mathbf{u}_0 (\mathbf{x},0)=\boldsymbol{\varphi}(\mathbf{x}),\quad\mathbf{x}\in\mathcal{O}.
\end{cases}
\end{equation}
Here $B^0$ is the differential expression for the effective operator $B_D^0$.
Our first main result is the estimate
\begin{equation}
\label{itr solutions L2}
\Vert \mathbf{u}_\varepsilon (\,\cdot\, ,t)-\mathbf{u}_0(\,\cdot\, ,t)\Vert _{L_2(\mathcal{O})}
\leqslant C\varepsilon (t+\varepsilon ^2)^{-1/2}e^{-c t}\Vert \boldsymbol{\varphi}\Vert _{L_2(\mathcal{O})},\quad t\geqslant 0,
\end{equation}
for sufficiently small $\varepsilon$. For fixed time $t>0$, this estimate is of sharp order $O(\varepsilon)$.
Our second main result is approximation of the solution $\mathbf{u}_\varepsilon (\,\cdot\, ,t)$ in the energy norm:
\begin{equation}
\label{intr solutions H1}
\Vert \mathbf{u}_\varepsilon (\,\cdot\, ,t)-\mathbf{v}_\varepsilon(\,\cdot\, ,t)\Vert _{H^1(\mathcal{O})}
\leqslant C(\varepsilon ^{1/2}t^{-3/4}+\varepsilon t^{-1})
e^{-c t}\Vert \boldsymbol{\varphi}\Vert _{L_2(\mathcal{O})},\quad t>0.
\end{equation}
Here $\mathbf{v}_\varepsilon (\,\cdot\, ,t)=\mathbf{u}_0(\,\cdot\, ,t)+\varepsilon \mathcal{K}_D(t;\varepsilon )\boldsymbol{\varphi}(\,\cdot\,)$
is the first order approximation of the solution $\mathbf{u}_\varepsilon (\,\cdot\, ,t)$. The operator $\mathcal{K}_D(t;\varepsilon )$ is a corrector. It involves rapidly oscillating factors, and so depends on $\varepsilon$. We have ${\Vert \varepsilon \mathcal{K}_D(t;\varepsilon )\Vert_{L_2\rightarrow H^1}\!\!=\!\!O(1)}$. For fixed $t$, estimate \eqref{intr solutions H1} is of order $O(\varepsilon ^{1/2})$ due to the influence of the boundary layer.
The presence of the boundary layer is confirmed by the fact that, in a strictly interior subdomain $\mathcal{O}'\subset \mathcal{O}$, the order of the $H^1$-estimate can be improved:
\begin{equation*}
\Vert \mathbf{u}_\varepsilon (\,\cdot\, ,t)-\mathbf{v}_\varepsilon(\,\cdot\, ,t)\Vert _{H^1(\mathcal{O}')}
\leqslant C\varepsilon (t^{-1/2}\delta ^{-1}+t^{-1})e^{-c t}\Vert \boldsymbol{\varphi}\Vert _{L_2(\mathcal{O})},\quad  t>0.
\end{equation*}
Here $\delta =\mathrm{dist}\,\lbrace\mathcal{O}';\partial\mathcal{O}\rbrace$.

In the general case, the corrector involves a smoothing operator. We distinguish conditions under which it is possible to use a simpler corrector
which does not include the smoothing operator. Along with estimate \eqref{intr solutions H1}, we obtain approximation of the flux $g^\varepsilon b(\mathbf{D})\mathbf{u}_\varepsilon (\,\cdot\, ,t)$ in the $L_2$-norm.

The constants in estimates \eqref{itr solutions L2} and \eqref{intr solutions H1} are controlled in terms of the problem data; they do not depend on  $\boldsymbol{\varphi}$. Therefore, estimates \eqref{itr solutions L2} and \eqref{intr solutions H1} can be rewritten in the uniform operator topology.
In a simpler case where $Q_0(\mathbf{x})=\mathbf{1}_n$, we have
\begin{align*}
&\Vert e^{-B_{D,\varepsilon}t}-e^{-B_D^0t}\Vert _{L_2(\mathcal{O})\rightarrow L_2(\mathcal{O})}
\leqslant C\varepsilon (t+\varepsilon ^2)^{-1/2}e^{-c t},\quad t\geqslant 0,
\\
&\Vert e^{-B_{D,\varepsilon}t}-e^{-B_D^0t}-\varepsilon\mathcal{K}_D(t;\varepsilon )\Vert _{L_2(\mathcal{O})\rightarrow H^1(\mathcal{O})}
\leqslant C (\varepsilon ^{1/2}t^{-3/4}+\varepsilon t^{-1}) e^{-c t},\ t>0.
\end{align*}
The resuts of such type are called \textit{operator error estimates} in homogenization theory.

\subsection{Operator error estimates. Survey}

Currently, the study of operator error estimates is an actively developing area of homogenization theory. The interest in this subject arose in connection with the papers \cite{BSu0,BSu} by M.~Sh.~Birman and T.~A.~Suslina, where the operator $A_\varepsilon$ of the form $b(\mathbf{D})^*g^\varepsilon (\mathbf{x})b(\mathbf{D})$ acting in $L_2(\mathbb{R}^d;\mathbb{C}^n)$
was studied. By the \textit{spectral approach}, it was proved that
\begin{equation}
\label{A_eps res L2 intr}
\Vert (A_\varepsilon +I)^{-1}-(A^0+I)^{-1}\Vert _{L_2(\mathbb{R}^d)\rightarrow L_2(\mathbb{R}^d)}
\leqslant C\varepsilon .
\end{equation}
Here $A^0=b(\mathbf{D})^* g ^0 b(\mathbf{D})$ is an effective operator and $g^0$ is a constant effective matrix.
Approximation for the operator $(A_\varepsilon +I)^{-1}$ in the $(L_2\rightarrow H^1)$-norm was obtained in \cite{BSu06}:
\begin{equation}
\label{A_eps res H1 intr}
\Vert (A_\varepsilon +I)^{-1}-(A^0+I)^{-1}-\varepsilon K(\varepsilon)\Vert _{L_2(\mathbb{R}^d)\rightarrow H^1(\mathbb{R}^d)}
\leqslant C\varepsilon.
\end{equation}
Later T.~A.~Suslina carried over estimates \eqref{A_eps res L2 intr} and \eqref{A_eps res H1 intr} to more general operator  $B_\varepsilon$ of the form \eqref{B_D,eps in introduction} acting in  $L_2(\mathbb{R}^d;\mathbb{C}^n)$.
We also mention the paper \cite{Bo} by D.~I.~Borisov, where the expression for the effective operator ${B}^0$ was found and
approximations \eqref{A_eps res L2 intr}, \eqref{A_eps res H1 intr} for the resolvent were obtained.
In \cite{Bo}, it was assumed that the coefficients of the operator depend not only on the rapid variable, but also on the slow variable;
however, the coefficients of ${B}_\varepsilon$ were assumed to be sufficiently smooth.

To parabolic systems, the spectral approach was applied in the papers \cite{Su04,Su07} by T.~A.~Sus\-li\-na, where
the principal term of approximation was found, and in \cite{Su10}, where estimate with the corrector was proved:
\begin{align}
\label{A_eps exp L2 intr}
&\Vert e^{-A_\varepsilon t}-e^{-A^0 t}\Vert _{L_2(\mathbb{R}^d)\rightarrow L_2(\mathbb{R}^d)}
\leqslant C\varepsilon (t+\varepsilon ^2)^{-1/2},\quad t\geqslant 0,\\
\label{A_eps exp H1 intr}
&\Vert e^{-A_\varepsilon t}-e^{-A^0 t}-\varepsilon \mathcal{K}(t;\varepsilon )\Vert _{L_2(\mathbb{R}^d)\rightarrow H^1(\mathbb{R}^d)}
\leqslant C\varepsilon (t^{-1/2}+t^{-1}),\quad t \geqslant \varepsilon ^2.
\end{align}
In these estimates, the exponentially decreasing function of $t$ is absent, because the bottom of the spectra of $A_\varepsilon$ and $A^0$ is
zero. The exponential of the operator $B_\varepsilon$ of the form \eqref{B_D,eps in introduction} was studied in the paper \cite{M} by Yu.~M.~Meshkova, where analogs of inequalities \eqref{A_eps exp L2 intr} and \eqref{A_eps exp H1 intr} were obtained.

A \textit{different approach} to operator error estimates in homogenization theory was suggested by V.~V.~Zhikov \cite{Zh1}.
In \cite{Zh1,ZhPas}, estimates of the form \eqref{A_eps res L2 intr} and \eqref{A_eps res H1 intr} for the acoustics and elasticity operators were obtained.
The \textit{``modified method of the first order approximation''\,} or the \textit{``shift method''}, in the terminology of the authors,
was based on analysis of the first order approximation to the solution and introduction of the additional parameter.
Along with problems in $\mathbb{R}^d$, in \cite{Zh1,ZhPas}, homogenization problems in a bounded domain $\mathcal{O}\subset \mathbb{R}^d$ with the Dirichlet or Neumann
boundary conditions were studied. To parabolic equations, the shift method was applied in \cite{ZhPAs_parabol}, where analogs of estimates \eqref{A_eps exp L2 intr} and \eqref{A_eps exp H1 intr} were proved. Further results of V.~V.~Zhikov, S.~E.~Pastukhova, and their students are discussed in the recent survey \cite{ZhPasUMN}.

Operator error estimates for the Dirichlet and Neumann problems for second order elliptic equations
in a bounded domain were studied by many authors. Apparently, the first result is due to Sh.~Moskow and M.~Vogelius who proved an estimate
\begin{equation}
\label{A_D,eps L2 ots in introduction}
\Vert A_{D,\varepsilon}^{-1}-(A_D^0)^{-1}\Vert _{L_2(\mathcal{O})\rightarrow L_2(\mathcal{O})}\leqslant C\varepsilon;
\end{equation}
see \cite[Corollary~2.2]{MoV}. Here the operator $A_{D,\varepsilon}$ acts in $L_2(\mathcal{O})$, where $\mathcal{O}\subset \mathbb{R}^2$,
and is given by $-\div g^\varepsilon (\mathbf{x})\nabla$ with the Dirichlet condition on  $\partial\mathcal{O}$. The matrix-valued
function  $g(\mathbf{x})$ is assumed to be infinitely smooth.

For arbitrary dimension, homogenization problems in a bounded domain were studied in  \cite{Zh1} and \cite{ZhPas}.
The acoustics and elasticity operators with the Dirichlet or Neumann boundary conditions and without any smoothness assumptions on coefficients
were considered. The authors obtained approximation with corrector for the inverse operator
 in the $(L_2\rightarrow H^1)$-norm with error estimate of order $O(\sqrt{\varepsilon})$.
The order deteriorates as compared with a similar result in $\mathbb{R}^d$; this is explained by the boundary influence.
As a rough consequence, approximation of the form \eqref{A_D,eps L2 ots in introduction} with error estimate of order $O(\sqrt{\varepsilon})$ was deduced.
Similar results for the operator $-\div g^\varepsilon (\mathbf{x})\nabla$
in a bounded domain $\mathcal{O}\subset\mathbb{R}^d$ with the Dirichlet or Neumann boundary conditions were
obtained by G.~Griso \cite{Gr1,Gr2} with the help of the  ``unfolding''  method.
In \cite{Gr2}, for the same operator a sharp-order estimate \eqref{A_D,eps L2 ots in introduction} was proved.
For elliptic systems similar results were independently obtained in \cite{KeLiS} and in \cite{PSu,Su13}.
Further results and a detailed survey can be found in \cite{Su_SIAM,Su15}.

For the matrix operator of the form \eqref{B_D,eps in introduction} with the Dirichlet condition, a homogenization problem was studied by Q.~Xu \cite{Xu,Xu3}.
The case of the Neumann boundary condition was studied in \cite{Xu2}.
However, in the papers by Q.~Xu, the operator is subject to a rather restrictive condition of uniform ellipticity.
Approximations of the generalized resolvent of the operator \eqref{B_D,eps in introduction} with two-parametric error estimates were obtained in the recent paper \cite{MSuPOMI} by the authors (see also a brief communication \cite{MSuFAA2017}).
We focus on these results in more detail, since they are basic for us.
For $\zeta\in\mathbb{C}\setminus\mathbb{R}_+$, $\vert\zeta\vert\geqslant 1$, and sufficiently small $\varepsilon$, we have
\begin{align}
\label{main result 1}
&\Vert (B_{D,\varepsilon}-\zeta Q_0^\varepsilon )^{-1}-(B_D^0-\zeta \overline{Q_0})^{-1}\Vert _{L_2(\mathcal{O})\rightarrow L_2(\mathcal{O})}
\leqslant C(\phi)\varepsilon\vert\zeta\vert ^{-1/2},
\\
\label{main result 2}
\begin{split}
&\Vert (B_{D,\varepsilon}-\zeta Q_0^\varepsilon )^{-1}-(B_D^0-\zeta \overline{Q_0})^{-1}-\varepsilon K_D(\varepsilon;\zeta)\Vert _{L_2(\mathcal{O})\rightarrow H^1(\mathcal{O})}
\leqslant C(\phi)\bigl(\varepsilon ^{1/2}\vert \zeta\vert ^{-1/4}+\varepsilon\bigr).
\end{split}
\end{align}
Note that the values $C(\phi)$ are controlled explicitly in terms of the problem data and the angle $\phi=\mathrm{arg}\,\zeta$.
 Estimates \eqref{main result 1} and \eqref{main result 2} are uniform with respect to $\phi$ in any domain of the form $\lbrace\zeta=\vert\zeta\vert e^{i\phi}\in\mathbb{C}: \vert\zeta\vert\geqslant 1, \phi _0\leqslant \phi\leqslant 2\pi-\phi _0\rbrace$ with arbitrarily small $\phi _0 >0$.
 Moreover, in \cite{MSuPOMI}, analogs of estimates \eqref{main result 1} and \eqref{main result 2} in a wider domain of spectral parameter $\zeta$ were proved.

We proceed to discussion of the parabolic problems in a bounded domain.
In the two-dimensional case, some estimates of operator type for elliptic and parabolic equations were obtained in \cite{ShKoLe}.
However, in \cite{ShKoLe}, the matrix $g$ was assumed to be $C^\infty$-smooth, and the initial data for a parabolic equation belonged to $H^2(\mathcal{O})$.
In the case of arbitrary dimension and without smoothness assumptions on coefficients, approximation for the exponential of the operator
$b(\mathbf{D})^*g^\varepsilon (\mathbf{x})b(\mathbf{D})$ (with the Dirichlet or Neumann conditions) was found in the paper \cite{MSu2} by the authors:
\begin{align*}
&\Vert e^{-A_{D,\varepsilon}t}-e^{-A_D^0 t}\Vert _{L_2(\mathcal{O})\rightarrow L_2(\mathcal{O})}
\leqslant C\varepsilon (t+\varepsilon ^2)^{-1/2}e^{-c t},\quad t\geqslant 0,\\
&\Vert e^{-A_{D,\varepsilon}t}-e^{-A_D^0 t}-\varepsilon\mathcal{K}_D(t;\varepsilon )\Vert _{L_2(\mathcal{O})\rightarrow H^1(\mathcal{O})}
\leqslant C\varepsilon ^{1/2}t^{-3/4}e^{-c t},\quad t\geqslant \varepsilon ^2.
\end{align*}
The method of \cite{MSu2} was based on employing the identity
\begin{equation*}
e^{-A_{D,\varepsilon}t}=-\frac{1}{2\pi i}\int\limits_\gamma e^{-\zeta t}(A_{D,\varepsilon}-\zeta I)^{-1}\,d\zeta,
\end{equation*}
where $\gamma \subset \mathbb{C}$ is a contour enclosing the spectrum of $A_{D,\varepsilon}$ in positive direction.
This identity allows us to deduce approximations for the operator exponential $e^{-A_{D,\varepsilon}t}$ from the corresponding approximations of the resolvent $(A_{D,\varepsilon}-\zeta I)^{-1}$ with two-parametric error estimates (with respect to $\varepsilon $ and $\zeta $).
The required approximations for the resolvent were found in \cite{Su15}.

The operator with coefficients periodic in the space and time variables was studied by J.~Geng and Z.~Shen \cite{GeSh}. In \cite{GeSh},
operator error estimates for the equation
$$
\partial _t\mathbf{u}_\varepsilon(\mathbf{x},t)=-\mathrm{div}\,g (\varepsilon^{-1}\mathbf{x} ,\varepsilon ^{-2}t)\nabla\mathbf{u}_\varepsilon(\mathbf{x},t)
$$
in a bounded domain of class $C^{1,1}$ were obtained. The results of \cite{GeSh}
were generalized to the case of Lipschitz domains by Q.~Xu and Sh.~Zhou \cite{XuZ}.

\subsection{Method}
We develop the method of the paper \cite{MSu2}. It is based upon the following representation for the solution
$\mathbf{u}_\varepsilon$ of the first initial boundary-value problem \eqref{intr problem}:
$\mathbf{u}_\varepsilon (\,\cdot\, ,t)=-\frac{1}{2\pi i}\int_\gamma e^{-\zeta t}(B_{D,\varepsilon}-\zeta Q_0^\varepsilon )^{-1}\boldsymbol{\varphi}\,d\zeta $,
where $\gamma \subset\mathbb{C}$ is a suitable contour. The solution of the effective problem \eqref{intr eff problem} admits a similar representation. Hence,
\begin{equation}
\label{intr tozd solutions}
\begin{split}
\mathbf{u}_\varepsilon (\,\cdot\, ,t)-\mathbf{u}_0 (\,\cdot\, ,t)
=-\frac{1}{2\pi i}\int\limits_\gamma e^{-\zeta t}\left((B_{D,\varepsilon}-\zeta Q_0^\varepsilon )^{-1}
-(B_D^0-\zeta\overline{Q_0})^{-1}\right)
\boldsymbol{\varphi}\,d\zeta .
\end{split}
\end{equation}
Using the results of \cite{MSuPOMI} (estimate \eqref{main result 1}), we obtain approximation of the resolvent for $\zeta \in\gamma$
and employ representation \eqref{intr tozd solutions}. This leads to \eqref{itr solutions L2}.
Note that the dependence of the right-hand side of \eqref{main result 1} on $\zeta$ for large $\vert\zeta\vert$
is important for us. Approximation with the corrector taken into account is obtained in a similar way.

\subsection{Plan of the paper} The paper consists of five sections and Appendix (\S\S\ref{Section 6}--\ref{Section removing S-eps in strictly interior subdomain}).
In \S\ref{Section 1}, we describe the class of operators $B_{D,\varepsilon}$, introduce the effective operator $B_D^0$, and formulate the needed results about approximation of the operator  $(B_{D,\varepsilon}-\zeta Q_0^\varepsilon )^{-1}$. The main results of the paper are obtained in \S\ref{Section 2}. In \S\ref{Section 3}, these results are applied to homogenization of the solutions of the first initial boundary-value problem for nonhomogeneous parabolic equation. \S\S\ref{Section 4}, \ref{Section 5} are devoted to applications of the general results.
In \S\ref{Section 4}, a scalar elliptic operator with a singular potential of order $O(\varepsilon ^{-1})$ is considered. In \S\ref{Section 5}, we study
an operator with a singular potential of order $O(\varepsilon ^{-2})$.
In Appendix (\S\S\ref{Section 6}--\ref{Section removing S-eps in strictly interior subdomain}), we prove some statements
   concerning removal of the smoothing operator in the corrector. The case of additional smoothness of the boundary is considered in \S\ref{Section removing steklov operator}; the case
    of a strictly interior subdomain is discussed in \S\ref{Section removing S-eps in strictly interior subdomain}.
    The needed properties of the oscillating factors in the corrector are obtained in \S\ref{Section 6}.

\subsection{Notation} Let $\mathfrak{H}$ and $\mathfrak{H}_*$ be complex separable Hilbert spaces. The symbols $(\,\cdot\, ,\,\cdot\,)_\mathfrak{H}$ and $\Vert \,\cdot\,\Vert _\mathfrak{H}$ stand for the inner product and the norm in $\mathfrak{H}$; the symbol $\Vert \,\cdot\,\Vert _{\mathfrak{H}\rightarrow\mathfrak{H}_*}$ denotes the norm of a linear
continuous operator acting from $\mathfrak{H}$ to $\mathfrak{H}_*$.

The set of natural numbers and the set of nonnegative integers are denoted by $\mathbb{N}$ and $\mathbb{Z}_+$, respectively.
We denote  $\mathbb{R}_+: = [0,\infty)$.
The symbols $\langle \,\cdot\, ,\,\cdot\,\rangle$ and $\vert \,\cdot\,\vert$ denote the inner product and the norm in $\mathbb{C}^n$; $\mathbf{1}_n$ is the identity
$(n\times n)$-matrix. If $a$ is an $(m\times n)$-matrix, then the symbol $\vert a\vert$ denotes the norm of $a$ viewed as operator from $\mathbb{C}^n$ to $\mathbb{C}^m$.
If $\alpha =(\alpha _1,\dots, \alpha_d)\in\mathbb{Z}_+^d$ is a multiindex, $\vert\alpha\vert$ denotes its length: $\vert\alpha\vert =\sum _{j=1}^d\alpha _j$.
For $z\in\mathbb{C}$, the complex conjugate number is denoted by $z^*$.
(We use such nonstandard notation, because the upper line denotes the mean value of a periodic function over the periodicity cell.)
We denote $\mathbf{x}=(x_1,\dots , x_d)\in\mathbb{R}^d$, $iD_j=\partial _j =\partial /\partial x_j$, $j=1,\dots,d$, $\mathbf{D}=-i\nabla=(D_1,\dots ,D_d)$.
The $L_p$-classes of $\mathbb{C}^n$-valued functions in a domain $\mathcal{O}\subset\mathbb{R}^d$ are denoted by $L_p(\mathcal{O};\mathbb{C}^n)$, $1\leqslant p\leqslant \infty$.
The Sobolev classes of $\mathbb{C}^n$-valued functions in a domain $\mathcal{O}\subset\mathbb{R}^d$ are denoted by $H^s(\mathcal{O};\mathbb{C}^n)$.
By $H^1_0(\mathcal{O};\mathbb{C}^n)$ we denote the closure of $C_0^\infty (\mathcal{O};\mathbb{C}^n)$ in $H^1(\mathcal{O};\mathbb{C}^n)$. If $n=1$, we write simply $L_p(\mathcal{O})$, $H^s(\mathcal{O})$, etc., but sometimes, if this does not lead to confusion, we use such simple notation for the spaces of vector-valued or matrix-valued functions.
The symbol $L_p((0,T);\mathfrak{H})$, $1\leqslant p\leqslant\infty$, denotes the $L_p$-space of $\mathfrak{H}$-valued functions on the interval $(0,T)$.

Various constants in estimates are denoted by ${c, C, \mathrm{C}, \mathcal{C}, \mathfrak{C}}$
(probably, with indices and marks).

The main results of the present paper were announced in \cite{MSuFAA2017}.

\section{The results on homogenization of the Dirichlet problem\\ for elliptic systems}

\label{Section 1}

\subsection{Lattices in $\mathbb{R}^d$} Let $\Gamma \subset \mathbb{R}^d$ be a lattice generated by the basis $\mathbf{a}_1,\dots ,\mathbf{a}_d \in \mathbb{R}^d$:
$$
\Gamma =\Bigl\lbrace
\mathbf{a}\in \mathbb{R}^d : \mathbf{a}=\sum _{j=1}^d \nu _j \mathbf{a}_j, \nu _j\in \mathbb{Z}
\Bigr\rbrace ,
$$
and let $\Omega$ be the elementary cell of the lattice $\Gamma$:
$$
\Omega =
\Bigl \lbrace
\mathbf{x}\in \mathbb{R}^d :\mathbf{x}=\sum _{j=1}^d \tau _j \mathbf{a}_j , -\frac{1}{2}<\tau _j<\frac{1}{2}
\Bigr\rbrace .
$$
By $\vert \Omega \vert $ we denote the Lebesgue measure of the cell $\Omega$: $\vert \Omega \vert =\mathrm{meas}\,\Omega$.
We put $2r_1:=\mathrm{diam}\,\Omega.$

Let $\widetilde{H}^1(\Omega)$ denote the subspace of functions in $H^1(\Omega)$, whose $\Gamma$-periodic extension to $\mathbb{R}^d$ belongs to $H^1_{\mathrm{loc}}(\mathbb{R}^d)$.
If $\Phi (\mathbf{x})$~is a~$\Gamma$-periodic matrix-valued function in $\mathbb{R}^d$, we put
$\Phi ^\varepsilon (\mathbf{x}):=\Phi (\mathbf{x}/\varepsilon)$, $\varepsilon >0$; $\overline{\Phi}:=\vert \Omega\vert ^{-1}\int_\Omega \Phi(\mathbf{x})\,d\mathbf{x}$, $\underline{\Phi}:=\big(\vert \Omega\vert ^{-1}\int_\Omega \Phi(\mathbf{x})^{-1}\,d\mathbf{x}\big)^{-1}$.
 Here, in the definition of  $\overline{\Phi}$ it is assumed that $\Phi\in L_{1,\mathrm{loc}}(\mathbb{R}^d)$;
 in the definition of $\underline{\Phi}$ it is assumed that the matrix $\Phi$ is square and nondegenerate, and  $\Phi^{-1}\in L_{1,\mathrm{loc}}(\mathbb{R}^d)$.
 By $[\Phi^\varepsilon ]$ we denote the operator of multiplication by the matrix-valued function $\Phi^\varepsilon (\mathbf{x})$.

\subsection{The Steklov smoothing} The Steklov smoothing operator $S_\varepsilon^{(k)}$, $\varepsilon >0$, acts in $L_2(\mathbb{R}^d;\mathbb{C}^k)$ (where $k\in\mathbb{N}$) and is given by
\begin{equation}
\label{S_eps}
\begin{split}
(S_\varepsilon^{(k)} \mathbf{u})(\mathbf{x})=\vert \Omega \vert ^{-1}\int\limits_\Omega \mathbf{u}(\mathbf{x}-\varepsilon \mathbf{z})\,d\mathbf{z},\quad \mathbf{u}\in L_2(\mathbb{R}^d;\mathbb{C}^k).
\end{split}
\end{equation}
We shall omit the index $k$ in the notation and write simply $S_\varepsilon$. Obviously,
$S_\varepsilon \mathbf{D}^\alpha \mathbf{u}=\mathbf{D}^\alpha S_\varepsilon \mathbf{u}$ for $\mathbf{u}\in H^\sigma(\mathbb{R}^d;\mathbb{C}^k)$
and any multiindex $\alpha$ such that $\vert \alpha\vert \leqslant \sigma$.
Note that
\begin{equation}
\label{S_eps <= 1}
\Vert S_\varepsilon \Vert _{H^\sigma(\mathbb{R}^d)\rightarrow H^\sigma(\mathbb{R}^d)}\leqslant 1,\quad \sigma\geqslant 0.
\end{equation}
We need the following properties of the operator $S_\varepsilon$
(see \cite[Lemmas~1.1 and~1.2]{ZhPas} or \cite[Propositions 3.1 and 3.2]{PSu}).

\begin{proposition}
\label{Proposition S__eps - I}
For any function $\mathbf{u}\in H^1(\mathbb{R}^d;\mathbb{C}^k)$, we have
\begin{equation*}
\Vert S_\varepsilon \mathbf{u}-\mathbf{u}\Vert _{L_2(\mathbb{R}^d)}\leqslant \varepsilon r_1\Vert \mathbf{D}\mathbf{u}\Vert _{L_2(\mathbb{R}^d)},
\end{equation*}
where $2r_1=\mathrm{diam}\,\Omega$.
\end{proposition}

\begin{proposition}
\label{Proposition f^eps S_eps}
Let $\Phi$ be a $\Gamma$-periodic function in $\mathbb{R}^d$ such that $\Phi\in L_2(\Omega)$.
Then the operator $[\Phi ^\varepsilon ]S_\varepsilon $ is continuous in $L_2(\mathbb{R}^d)$ and
\begin{equation*}
\Vert [\Phi^\varepsilon]S_\varepsilon \Vert _{L_2(\mathbb{R}^d)\rightarrow L_2(\mathbb{R}^d)}\leqslant \vert \Omega \vert ^{-1/2}\Vert \Phi\Vert _{L_2(\Omega)}.
\end{equation*}
\end{proposition}

\subsection{The operator $A_{D,\varepsilon}$}
\label{Subsection operatoer A_D,eps}
Let $\mathcal{O}\subset \mathbb{R}^d$ be a bounded domain of class $C^{1,1}$.
In $L_2(\mathcal{O};\mathbb{C}^n)$, we consider the operator $A_{D,\varepsilon}$ given formally by the differential expression $A_\varepsilon = b(\mathbf{D})^*g^\varepsilon (\mathbf{x})b(\mathbf{D})$ with the Dirichlet condition on $\partial\mathcal{O}$. Here $g(\mathbf{x})$ is a $\Gamma$-periodic Hermitian $(m\times m)$-matrix-valued function (in general, with complex entries). It is assumed that $g(\mathbf{x})>0$ and $g,g^{-1}\in L_\infty (\mathbb{R}^d)$. The differential operator $b(\mathbf{D})$ is given by $b(\mathbf{D})=\sum _{j=1}^d b_jD_j$, where $b_j$, $j=1,\dots ,d$, are constant matrices of size $m\times n$ (in general, with complex entries). Assume that $m\geqslant n$ and that the symbol $b(\boldsymbol{\xi})=\sum _{j=1}^d b_j\xi_j$ of the operator $b(\mathbf{D})$ has maximal rank:
$\mathrm{rank}\,b(\boldsymbol{\xi})=n$ for $0\neq \boldsymbol{\xi}\in\mathbb{R}^d$.
This condition is equivalent to the estimates
\begin{equation}
\label{<b^*b<}
\alpha _0\mathbf{1}_n \leqslant b(\boldsymbol{\theta})^*b(\boldsymbol{\theta})
\leqslant \alpha _1\mathbf{1}_n,\quad
\boldsymbol{\theta}\in \mathbb{S}^{d-1};\quad
0<\alpha _0\leqslant \alpha _1<\infty,
\end{equation}
with some positive constants $\alpha _0$ and $\alpha _1$.
From \eqref{<b^*b<} it follows that
\begin{equation}
\label{b_l <=}
\vert b_j\vert \leqslant \alpha _1^{1/2},\quad j=1,\dots ,d.
\end{equation}

The precise definition of the operator $A_{D,\varepsilon}$ is given in terms of the quadratic form
\begin{equation}
\label{a_D,eps}
\mathfrak{a}_{D,\varepsilon} [\mathbf{u},\mathbf{u}]=\int\limits_{\mathcal{O}}\langle g^\varepsilon (\mathbf{x})b(\mathbf{D})\mathbf{u},b(\mathbf{D})\mathbf{u}\rangle \,d\mathbf{x},\quad \mathbf{u}\in H^1_0(\mathcal{O};\mathbb{C}^n).
\end{equation}
Extending $\mathbf{u}\in H^1_0(\mathcal{O};\mathbb{C}^n)$ by zero to $\mathbb{R}^d\setminus\mathcal{O}$ and taking \eqref{<b^*b<} into account, we find
\begin{equation}
\label{a_D,eps estimates}
\begin{split}
\alpha _0\| g^{-1}\| ^{-1}_{L_\infty}\|\mathbf{D}\mathbf{u}\| ^2 _{L_2(\mathcal{O})}
\leqslant \!
\mathfrak{a}_{D,\varepsilon}[\mathbf{u},\mathbf{u}]
\leqslant \alpha _1\Vert g\Vert _{L_\infty}\!\| \mathbf{D}\mathbf{u}\| ^2 _{L_2(\mathcal{O})},
\quad\mathbf{u}\in H^1_0(\mathcal{O};\mathbb{C}^n).
\end{split}
\end{equation}

\subsection{Lower order terms. The operator $B_{D,\varepsilon}$}
\label{Subsection lower order terms}
We study the selfadjoint operator $B_{D,\varepsilon}$ whose principal part coincides with $A_{\varepsilon}$.
To define the lower order terms, we introduce $\Gamma$-periodic $(n\times n)$-matrix-valued functions (in general, with complex entries) $a_j$, $j=1,\dots ,d$, such that
\begin{align}
\label{a_j cond}
a_j \in L_\rho (\Omega ), \quad\rho =2 \;\mbox{for}\;d=1,\quad\rho >d\;\mbox{for}\;d\geqslant 2,\quad j=1,\dots ,d.
\end{align}
Next, let $Q$ and $Q_0$ be $\Gamma$-periodic Hermitian $(n\times n)$-matrix-valued functions (with complex entries) such that
\begin{align}
\label{Q condition}
&Q\in L_s(\Omega ),\quad s=1 \;\mbox{for}\;d=1,\quad s >d/2\;\mbox{for}\;d\geqslant 2;\\
& Q_0(\mathbf{x})>0;\quad Q_0,Q_0^{-1}\in L_\infty (\mathbb{R}^d).\nonumber
\end{align}
For convenience of further references, the following set of variables is called the ``problem data'':
\begin{equation}
\label{problem data}
\begin{split}
&d,\,m,\,n,\,\rho ,\,s ;\,\alpha _0,\, \alpha _1 ,\,\Vert g\Vert _{L_\infty},\, \Vert g^{-1}\Vert _{L_\infty},\, \Vert a_j\Vert _{L_\rho (\Omega)},\, j=1,\dots ,d;\\
&\Vert Q\Vert _{L_s(\Omega)};\, \Vert Q_0\Vert _{L_\infty},\,\Vert Q_0^{-1}\Vert _{L_\infty};\;\text{the parameters of the lattice} \ \Gamma ;\;\text{the domain}\  \mathcal{O}.
\end{split}
\end{equation}

In $L_2(\mathcal{O};\mathbb{C}^n)$, we consider the operator $B_{D,\varepsilon}$, $0<\varepsilon\leqslant 1$, formally given by the differential expression
\begin{equation}
\label{B_D,eps}
B_{\varepsilon}=b(\mathbf{D})^* g^\varepsilon (\mathbf{x})b(\mathbf{D})+\sum _{j=1}^d \left(
a_j^\varepsilon (\mathbf{x})D_j +D_j a_j^\varepsilon (\mathbf{x})^*
\right)
+Q^\varepsilon (\mathbf{x})+\lambda Q_0^\varepsilon (\mathbf{x})
\end{equation}
with the Dirichlet boundary condition. Here the constant $\lambda$ is chosen so that the operator $B_{D,\varepsilon}$ is positive definite
(see \eqref{lambda =} below). The precise definition of the operator $B_{D,\varepsilon}$ is given in terms of the quadratic form
\begin{equation}
\label{b_D,eps}
\begin{split}
\mathfrak{b}_{D,\varepsilon }[\mathbf{u},\mathbf{u}]&=(g^\varepsilon b(\mathbf{D})\mathbf{u},b(\mathbf{D})\mathbf{u})_{L_2(\mathcal{O})}+2\mathrm{Re}\,\sum _{j=1}^d (a_j^\varepsilon D_j \mathbf{u},\mathbf{u})_{L_2(\mathcal{O})}
\\
&\quad+(Q^\varepsilon \mathbf{u},\mathbf{u})_{L_2(\mathcal{O})}+\lambda (Q_0^\varepsilon \mathbf{u},\mathbf{u})_{L_2(\mathcal{O})},\quad \mathbf{u}\in H^1_0(\mathcal{O};\mathbb{C}^n).
\end{split}
\end{equation}
Let us check that the form $\mathfrak{b}_{D,\varepsilon}$ is closed.
By the H\"older inequality and the Sobolev embedding theorem, it can be shown that for any $\nu>0$ there exist constants $C_j(\nu)>0$ such that
\begin{equation*}
\begin{split}
\Vert a_j^*\mathbf{u}\Vert ^2 _{L_2(\mathbb{R}^d)}\leqslant \nu \Vert \mathbf{D}\mathbf{u}\Vert ^2_{L_2(\mathbb{R}^d)}+C_j(\nu )\Vert \mathbf{u}\Vert ^2 _{L_2(\mathbb{R}^d)},\quad \mathbf{u}\in H^1(\mathbb{R}^d;\mathbb{C}^n),
\end{split}
\end{equation*}
$j=1,\dots ,d$; see  \cite[(5.11)--(5.14)]{SuAA}.
By the change of variables $\mathbf{y}:=\varepsilon ^{-1}\mathbf{x}$ and $\mathbf{u}(\mathbf{x})=:\mathbf{v}(\mathbf{y})$, we deduce
\begin{align*}
\|   (a_j^\varepsilon )^*\mathbf{u}\| ^2_{L_2(\mathbb{R}^d)}&=\int\limits_{\mathbb{R}^d}\vert a_j(\varepsilon ^{-1}\mathbf{x})^*\mathbf{u}(\mathbf{x})\vert ^2\,d\mathbf{x}
=\varepsilon ^d\int\limits_{\mathbb{R}^d}\vert a_j(\mathbf{y})^*\mathbf{v}(\mathbf{y})\vert ^2\,d\mathbf{y}
\\
&\quad\leqslant \varepsilon ^d\nu \int\limits_{\mathbb{R}^d}\vert \mathbf{D}_{\mathbf{y}}\mathbf{v}(\mathbf{y})\vert ^2\,d\mathbf{y}
+\varepsilon ^d C_j(\nu)\int\limits_{\mathbb{R}^d}\vert \mathbf{v}(\mathbf{y})\vert ^2\,d\mathbf{y}
\\
&\quad\leqslant \nu \Vert \mathbf{D}\mathbf{u}\Vert ^2_{L_2(\mathbb{R}^d)}+C_j(\nu)\Vert \mathbf{u}\Vert ^2_{L_2(\mathbb{R}^d)},
\quad \mathbf{u}\in H^1(\mathbb{R}^d;\mathbb{C}^n),\quad 0<\varepsilon\leqslant 1.
\end{align*}
Then, by \eqref{<b^*b<}, for any $\nu >0$ there exists a constant $C(\nu)>0$ such that
\begin{equation}
\label{sum a-j u}
\begin{split}
\sum _{j=1}^d \Vert (a_j^\varepsilon)^*\mathbf{u}\Vert ^2 _{L_2(\mathbb{R}^d)}
\leqslant \nu
\Vert (g^\varepsilon)^{1/2}b(\mathbf{D})\mathbf{u}\Vert ^2_{L_2(\mathbb{R}^d)}
+C(\nu)\Vert \mathbf{u}\Vert ^2_{L_2(\mathbb{R}^d)},\\\mathbf{u}\in H^1(\mathbb{R}^d;\mathbb{C}^n),\quad 0<\varepsilon\leqslant 1.
\end{split}
\end{equation}
If $\nu$ is fixed, then $C(\nu)$ depends only on $d$, $\rho$, $\alpha _0$, the norms $\Vert g^{-1}\Vert _{L_\infty}$, $\Vert a_j\Vert _{L_\rho (\Omega)}$, $j=1,\dots ,d$,
and the parameters of the lattice $\Gamma$.

By \eqref{<b^*b<}, for $\mathbf{u}\in H^1(\mathbb{R}^d;\mathbb{C}^n)$ we have
\begin{equation}
\label{Du <= c_1^2a}
\Vert \mathbf{D}\mathbf{u}\Vert ^2_{L_2(\mathbb{R}^d)}
\leqslant c_1^2\Vert (g^\varepsilon )^{1/2}b(\mathbf{D})\mathbf{u}\Vert ^2_{L_2(\mathbb{R}^d)},
\end{equation}
where $c_1:=\alpha _0 ^{-1/2}\Vert g^{-1}\Vert^{1/2}_{L_\infty}$.
Combining this with \eqref{sum a-j u}, we obtain
\begin{equation}
\label{2Re sum j <=}
\begin{split}
2\bigg\vert \mathrm{Re}\,\sum _{j=1}^d (D_j\mathbf{u},(a_j^\varepsilon)^*\mathbf{u})_{L_2(\mathbb{R}^d)}\bigg\vert
\leqslant\frac{1}{4}\Vert (g^\varepsilon)^{1/2}b(\mathbf{D})\mathbf{u}\Vert ^2_{L_2(\mathbb{R}^d)} +c_2\Vert \mathbf{u}\Vert ^2_{L_2(\mathbb{R}^d)},\\
\mathbf{u}\in H^1(\mathbb{R}^d;\mathbb{C}^n),\quad 0<\varepsilon\leqslant 1,
\end{split}
\end{equation}
where $c_2:=8c_1^2C(\nu _0)$ with $\nu _0:=2^{-6}\alpha _0\Vert g^{-1}\Vert ^{-1}_{L_\infty}$.

Next, by condition \eqref{Q condition} on $Q$, for any $\nu >0$ there exists a constant $C_Q(\nu)>0$ such that
\begin{equation}
\label{(Qu,u)<=}
\begin{split}
\vert (Q^\varepsilon \mathbf{u},\mathbf{u})_{L_2(\mathbb{R}^d)}\vert \leqslant\nu \Vert \mathbf{D}\mathbf{u}\Vert ^2_{L_2(\mathbb{R}^d)}+C_Q(\nu)\Vert \mathbf{u}\Vert ^2_{L_2(\mathbb{R}^d)},\\
\mathbf{u}\in H^1(\mathbb{R}^d;\mathbb{C}^n),\quad 0<\varepsilon\leqslant 1 .
\end{split}
\end{equation}
For fixed $\nu$, the constant $C_Q(\nu)$ is controlled in terms of $d$, $s$, $\Vert Q\Vert _{L_s(\Omega)}$, and the parameters of the lattice $\Gamma$.

We fix a constant $\lambda$ in \eqref{B_D,eps} as in \cite[Subsection 2.8]{MSu15}:
\begin{equation}
\label{lambda =}
\lambda := (C_Q(\nu _*)+c_2)\Vert Q_0^{-1}\Vert _{L_\infty}\quad\text{for}\;\nu _* :=2^{-1}\alpha _0\Vert g^{-1}\Vert ^{-1}_{L_\infty}.
\end{equation}

We return to the form \eqref{b_D,eps}. Extending the function $\mathbf{u}\in H^1_0(\mathcal{O};\mathbb{C}^n)$ by zero to $\mathbb{R}^d\setminus\mathcal{O}$ and using
 \eqref{a_D,eps}, \eqref{Du <= c_1^2a}, \eqref{2Re sum j <=}, and \eqref{(Qu,u)<=} with $\nu=\nu_*$, we obtain the lower estimate for the form \eqref{b_D,eps}:
\begin{align}
\label{b_eps >=}
\mathfrak{b}_{D,\varepsilon}[\mathbf{u},\mathbf{u}]\geqslant \frac{1}{4}\mathfrak{a}_{D,\varepsilon} [\mathbf{u},\mathbf{u}]\geqslant c_*\Vert \mathbf{D}\mathbf{u}\Vert ^2_{L_2(\mathcal{O})},\quad\mathbf{u}\in H^1(\mathcal{O};\mathbb{C}^n);\\
\label{c_*=}
c_*:=\frac{1}{4}\alpha _0\Vert g^{-1}\Vert ^{-1}_{L_\infty}.
\end{align}
Next, by \eqref{a_D,eps estimates}, \eqref{2Re sum j <=}, and \eqref{(Qu,u)<=} with $\nu=1$, we have
\begin{equation*}
\begin{split}
&\mathfrak{b}_{D,\varepsilon} [\mathbf{u},\mathbf{u}]\leqslant C_*\Vert \mathbf{u}\Vert ^2_{H^1(\mathbb{R}^d)},\quad
\mathbf{u}\in H^1(\mathcal{O};\mathbb{C}^n),
\end{split}
\end{equation*}
where $C_*:=\max\lbrace\frac{5}{4}\alpha _1\Vert g\Vert _{L_\infty}+1;C_Q(1)+\lambda\Vert Q_0\Vert _{L_\infty}+c_2\rbrace$.
Thus, the form $\mathfrak{b}_{D,\varepsilon}$ is closed. The corresponding selfadjoint operator in $L_2(\mathcal{O};\mathbb{C}^n)$ is denoted by $B_{D,\varepsilon}$.

By the Friedrichs inequality, \eqref{b_eps >=} implies that
\begin{equation}
\label{b D,eps >= H1-norm}
\mathfrak{b}_{D,\varepsilon}[\mathbf{u},\mathbf{u}]\geqslant c_*(\mathrm{diam}\,\mathcal{O})^{-2}\Vert \mathbf{u}\Vert ^2_{L_2(\mathcal{O})},\quad \mathbf{u}\in H^1_0(\mathcal{O};\mathbb{C}^n).
\end{equation}
Hence, the operator $B_{D,\varepsilon}$ is positive definite. By \eqref{b_eps >=} and \eqref{b D,eps >= H1-norm},
\begin{align}
\label{H^1-norm <= BDeps^1/2}
\Vert \mathbf{u}\Vert _{H^1(\mathcal{O})}\leqslant c_3\Vert B_{D,\varepsilon}^{1/2}\mathbf{u}\Vert _{L_2(\mathcal{O})},\quad \mathbf{u}\in H^1_0(\mathcal{O};\mathbb{C}^n);\\
\label{c_3}
c_3:=c_*^{-1/2}\left(1+(\mathrm{diam}\,\mathcal{O})^{2}\right)^{1/2}.
\end{align}

We also need an auxiliary operator $\widetilde{B}_{D,\varepsilon}$.
We factorize the matrix $Q_0(\mathbf{x})$: there exists a $\Gamma$-periodic matrix-valued function $f(\mathbf{x})$ such that $f$, $f^{-1}\in L_\infty (\mathbb{R}^d)$ and
\begin{equation}
\label{Q_0=}
Q_0(\mathbf{x})=(f(\mathbf{x})^*)^{-1}f(\mathbf{x})^{-1}.
\end{equation}
(For instance, one can choose $f(\mathbf{x})=Q_0(\mathbf{x})^{-1/2}$.)
Let $\widetilde{B}_{D,\varepsilon}$ be a selfadjoint operator in $L_2(\mathcal{O};\mathbb{C}^n)$ generated by the quadratic form
\begin{equation}
\label{tilde b D,eps=}
\widetilde{\mathfrak{b}}_{D,\varepsilon }[\mathbf{u},\mathbf{u}]:=\mathfrak{b}_{D,\varepsilon}[f^\varepsilon \mathbf{u},f^\varepsilon \mathbf{u}]
\end{equation}
on the domain
$
\mathrm{Dom}\,\widetilde{\mathfrak{b}}_{D,\varepsilon}:=\lbrace \mathbf{u}\in L_2(\mathcal{O};\mathbb{C}^n) : f^\varepsilon \mathbf{u}\in H^1_0(\mathcal{O};\mathbb{C}^n)\rbrace $.
In other words, $\widetilde{B}_{D,\varepsilon}=(f^\varepsilon )^*B_{D,\varepsilon}f^\varepsilon $. Let $\widetilde{B}_{\varepsilon}$ denote the differential expression $(f^\varepsilon )^*B_{\varepsilon}f^\varepsilon $. Note that
\begin{align}
\label{B deps and tilde B D,eps tozd resolvent}
(B_{D,\varepsilon}-\zeta Q_0^\varepsilon)^{-1}=f^\varepsilon (\widetilde{B}_{D,\varepsilon}-\zeta I)^{-1}(f^\varepsilon)^* .
\end{align}

\subsection{The effective matrix and its properties}
The effective operator for $A_{D,\varepsilon} $ is given by the differential expression $A^0=b(\mathbf{D})^*g^0b(\mathbf{D})$
with the Dirichlet condition on $\partial\mathcal{O}$. Here $g^0$ is a constant \textit{effective} matrix of size $m\times m$.
The matrix $g^0$ is expressed in terms of the solution of an auxiliary problem on the cell.
Let an $(n\times m)$-matrix-valued function $\Lambda (\mathbf{x})$ be the (weak) $\Gamma$-periodic solution of the problem
\begin{equation}
\label{Lambda problem}
b(\mathbf{D})^*g(\mathbf{x})(b(\mathbf{D})\Lambda (\mathbf{x})+\mathbf{1}_m)=0,\quad \int\limits_{\Omega }\Lambda (\mathbf{x})\,d\mathbf{x}=0.
\end{equation}
Then the effective matrix is given by
\begin{align}
\label{g^0}
&g^0:=\vert \Omega \vert ^{-1}\int\limits_{\Omega} \widetilde{g}(\mathbf{x})\,d\mathbf{x},
\\
\label{tilde g}
&\widetilde{g}(\mathbf{x}):=g(\mathbf{x})(b(\mathbf{D})\Lambda (\mathbf{x})+\mathbf{1}_m).
\end{align}
It can be checked that the matrix $g^0$ is positive definite.

According to \cite[(6.28) and Subsection~7.3]{BSu05}, the solution of problem \eqref{Lambda problem} satisfies
\begin{equation}
\label{Lamba in H1<=}
\Vert \Lambda\Vert _{H^1(\Omega)}
\leqslant M.
\end{equation}
Here the constant $M$  depends only on $m$, $\alpha _0$, $\Vert g\Vert _{L_\infty}$, $\Vert g^{-1}\Vert _{L_\infty}$, and the parameters of the lattice $\Gamma$.

 The effective matrix satisfies the estimates known as the Voigt--Reuss bracketing
(see, e.~g., \cite[Chapter~3, Theorem~1.5]{BSu}).

\begin{proposition}
Let $g^0$ be the effective matrix \eqref{g^0}. Then
\begin{equation}
\label{Foigt-Reiss}
\underline{g}\leqslant g^0\leqslant \overline{g}.
\end{equation}
If $m=n$, then $g^0=\underline{g}$.
\end{proposition}

From \eqref{Foigt-Reiss} it follows that
\begin{equation}
\label{|g^0|<=}
\vert g^0\vert \leqslant \Vert g\Vert _{L_\infty},\quad \vert (g^0)^{-1}\vert \leqslant \Vert g^{-1}\Vert _{L_\infty}.
\end{equation}

Now we distinguish the cases where one of the inequalities in \eqref{Foigt-Reiss}
becomes an identity, see \cite[Chapter~3, Propositions~1.6 and~1.7]{BSu}.

\begin{proposition}
The identity $g^0=\overline{g}$ is equivalent to the relations
\begin{equation}
\label{overline-g}
b(\D)^* {\mathbf g}_k(\x) =0,\ \ k=1,\dots,m,
\end{equation}
where ${\mathbf g}_k(\x)$, $k=1,\dots,m,$ are the columns of the matrix $g(\x)$.
\end{proposition}

\begin{proposition} The identity $g^0 =\underline{g}$ is equivalent to the representations
\begin{equation}
\label{underline-g}
{\mathbf l}_k(\x) = {\mathbf l}_k^0 + b(\D) {\mathbf w}_k,\ \ {\mathbf l}_k^0\in \C^m,\ \
{\mathbf w}_k \in \widetilde{H}^1(\Omega;\C^m),\ \ k=1,\dots,m,
\end{equation}
where ${\mathbf l}_k(\x)$, $k=1,\dots,m,$ are the columns of the matrix $g(\x)^{-1}$.
\end{proposition}

\subsection{The effective operator}
\label{Subsection Effective operator}
To describe how the lower order terms of the operator $B_{D,\varepsilon}$ are homogenized, we consider a
 $\Gamma$-periodic $(n\times n)$-matrix-valued function $\widetilde{\Lambda}(\mathbf{x})$ which is the (weak) solution of the problem
\begin{equation}
\label{tildeLambda_problem}
b(\mathbf{D})^*g(\mathbf{x})b(\mathbf{D})\widetilde{\Lambda }(\mathbf{x})+\sum \limits _{j=1}^dD_ja_j(\mathbf{x})^*=0,\quad \int\limits_{\Omega }\widetilde{\Lambda }(\mathbf{x})\,d\mathbf{x}=0.
\end{equation}
According to \cite[(7.51) and (7.52)]{SuAA}, we have
\begin{equation}
\label{tilde Lambda in H1 <=}
\Vert \widetilde{\Lambda}\Vert _{H^1(\Omega)}
\leqslant \widetilde{M},
\end{equation}
where the constant $\widetilde{M}$ depends only on $n$, $\rho$, $\alpha _0$, $\Vert g^{-1}\Vert _{L_\infty}$, $\Vert a_j\Vert _{L_\rho(\Omega)}$, $j=1,\dots,d$,
and the parameters of the lattice $\Gamma$.

Next, we define constant matrices $V$ and $W$ as follows:
\begin{align}
\label{V}
&V:=\vert \Omega \vert ^{-1}\int\limits_{\Omega}(b(\mathbf{D})\Lambda (\mathbf{x}))^*g(\mathbf{x})(b(\mathbf{D})\widetilde{\Lambda}(\mathbf{x}))\,d\mathbf{x},\\
\label{W}
&W:=\vert \Omega \vert ^{-1}\int\limits_{\Omega} (b(\mathbf{D})\widetilde{\Lambda}(\mathbf{x}))^*g(\mathbf{x})(b(\mathbf{D})\widetilde{\Lambda}(\mathbf{x}))\,d\mathbf{x}.
\end{align}

In $L_2(\mathcal{O};\mathbb{C}^n)$, consider the quadratic form
\begin{equation*}
\begin{split}
\mathfrak{b}_D^0[\mathbf{u},\mathbf{u}]&=(g^0b(\mathbf{D})\mathbf{u},b(\mathbf{D})\mathbf{u})_{L_2(\mathcal{O})}
+2\mathrm{Re}\,\sum _{j=1}^d (\overline{a_j}D_j\mathbf{u},\mathbf{u})_{L_2(\mathcal{O})}
\\
&\quad-2\mathrm{Re}\,(V\mathbf{u},b(\mathbf{D})\mathbf{u})_{L_2(\mathcal{O})}
-(W\mathbf{u},\mathbf{u})_{L_2(\mathcal{O})}
+(\overline{Q}\mathbf{u},\mathbf{u})_{L_2(\mathcal{O})}
\\
&\qquad+\lambda (\overline{Q_0}\mathbf{u},\mathbf{u})_{L_2(\mathcal{O})},
\quad \mathbf{u}\in H^1_0(\mathcal{O};\mathbb{C}^n).
\end{split}
\end{equation*}
The following estimates were proved in \cite[(2.22) and (2.23)]{MSuPOMI}:
\begin{align}
\label{b_D^0 ots }
&c_*\Vert \mathbf{D}\mathbf{u}\Vert ^2_{L_2(\mathcal{O})}\leqslant \mathfrak{b}_D^0[\mathbf{u},\mathbf{u}]\leqslant c_4\Vert \mathbf{u}\Vert ^2_{H^1(\mathcal{O})},\quad \mathbf{u}\in H^1_0(\mathcal{O};\mathbb{C}^n),
\\
\label{b_D^0 ots 2}
&\mathfrak{b}_D^0[\mathbf{u},\mathbf{u}]\geqslant c_*(\mathrm{diam}\,\mathcal{O})^{-2}\Vert \mathbf{u}\Vert ^2_{L_2(\mathcal{O})},\quad \mathbf{u}\in H^1_0(\mathcal{O};\mathbb{C}^n).
\end{align}
Here the constant $c_4$ depends only on the problem data \eqref{problem data}.
A selfadjoint operator in $L_2(\mathcal{O};\mathbb{C}^n)$ corresponding to the form $\mathfrak{b}_D^0$ is denoted by $B_D^0$.
By \eqref{b_D^0 ots } and \eqref{b_D^0 ots 2},
\begin{equation}
\label{H^1-norm <= BD0^1/2}
\begin{split}
&\Vert \mathbf{u}\Vert _{H^1(\mathcal{O})}\leqslant c_3\Vert (B_D^0)^{1/2}\mathbf{u}\Vert _{L_2(\mathcal{O})},\quad \mathbf{u}\in H^1_0(\mathcal{O};\mathbb{C}^n),
\end{split}
\end{equation}
where $c_3$ is given by \eqref{c_3}.

Due to condition $\partial\mathcal{O}\in C^{1,1}$, the operator $B_D^0$ is given by
\begin{equation}
\label{B_D^0}
B^0=b(\mathbf{D})^*g^0b(\mathbf{D})-b(\mathbf{D})^*V-V^*b(\mathbf{D})
+\sum _{j=1}^d(\overline{a_j+a_j^*})D_j-W+\overline{Q}+\lambda\overline{Q_0}
\end{equation}
on the domain $H^2(\mathcal{O};\mathbb{C}^n)\cap H^1_0(\mathcal{O};\mathbb{C}^n)$, and we have
\begin{equation}
\label{BD0 ^-1 L2->H2}
\Vert (B_D^0)^{-1}\Vert _{L_2(\mathcal{O})\rightarrow H^2(\mathcal{O})}\leqslant \widehat{c}.
\end{equation}
Here the constant $\widehat{c}$ depends only on the problem data \eqref{problem data}.
To justify this fact, we refer to the theorems about regularity of solutions of the strongly elliptic systems (see \cite[Chapter~4]{McL}).

\begin{remark}
Instead of condition $\partial\mathcal{O}\in C^{1,1}$, one could impose the following implicit condition\textnormal{:}
a bounded Lipschitz domain $\mathcal{O}\subset \mathbb{R}^d$ is such that estimate \eqref{BD0 ^-1 L2->H2} holds.
For such domain the results of the paper remain true. In the case of scalar elliptic operators, wide conditions on $\partial \mathcal{O}$ ensuring estimate \eqref{BD0 ^-1 L2->H2}
can be found in \textnormal{\cite{KoE}} and \textnormal{\cite[Chapter~7]{MaSh} (}in particular, it suffices to assume that $\partial\mathcal{O}\in C^\alpha$, $\alpha >3/2${\rm)}.
\end{remark}

Denote
\begin{equation}
\label{f_0=}
f_0:=\left(\overline{Q_0}\right)^{-1/2}.
\end{equation}
By \eqref{Q_0=},
\begin{equation}
\label{f_0<=}
\vert f_0\vert \leqslant \Vert f\Vert _{L_\infty}=\Vert Q_0^{-1}\Vert^{1/2}_{L_\infty},\quad \vert f_0^{-1}\vert \leqslant \Vert f^{-1}\Vert _{L_\infty}=\Vert Q_0\Vert ^{1/2}_{L_\infty}.
\end{equation}
In what follows, we will need the operator $\widetilde{B}_D^0:=f_0B_D^0 f_0$ corresponding to the quadratic form
\begin{equation}
\label{tilde b_D^0 =}
\widetilde{\mathfrak{b}}_{D}^0[\mathbf{u},\mathbf{u}]:=\mathfrak{b}_D^0[f_0\mathbf{u},f_0\mathbf{u}],\quad \mathbf{u}\in H^1_0(\mathcal{O};\mathbb{C}^n).
\end{equation}
Note that
$ (B_D^0-\zeta\overline{Q_0})^{-1}=f_0(\widetilde{B}_D^0-\zeta I)^{-1}f_0$.

\subsection{Approximation of the generalized resolvent $(B_{D,\varepsilon}-\zeta Q_0^\varepsilon)^{-1}$} 
Now we formulate the results of the paper \cite{MSuPOMI}, where the behavior of the generalized resolvent $(B_{D,\varepsilon}-\zeta Q_0^\varepsilon )^{-1}$ was studied.
Suppose that $\zeta\in\mathbb{C}\setminus\mathbb{R}_+$ and $\vert\zeta\vert\geqslant 1$. The principal term of approximation of the generalized resolvent $(B_{D,\varepsilon}-\zeta Q_0^\varepsilon)^{-1}$ was found in~\cite[Theorem~2.5]{MSuPOMI}; approximation of this resolvent in the~$(L_2\rightarrow H^1)$-norm with the corrector taken into account was found in~\cite[Theorem~2.6]{MSuPOMI}; an appropriate approximation of the operator $g^\varepsilon b(\mathbf{D})(B_{D,\varepsilon}-\zeta Q_0^\varepsilon )^{-1}$ (corresponding to the ``flux'') was obtained in~\cite[Proposition~10.7]{MSuPOMI}.

We choose the numbers $\varepsilon _0, \varepsilon _1\in (0,1]$ according to the following condition.

\begin{condition}
\label{condition varepsilon}
Let $\mathcal{O}\subset \mathbb{R}^d$ be a bounded domain. Denote
$$
(\partial\mathcal{O})_{\varepsilon}: =\left\lbrace \mathbf{x}\in \mathbb{R}^d : \mathrm{dist}\,\lbrace \mathbf{x};\partial\mathcal{O}\rbrace <\varepsilon \right\rbrace.
$$
Suppose that there exists a number $\varepsilon _0\in (0,1]$ such that the strip $(\partial\mathcal{O})_{\varepsilon _0}$ can be covered by a finite number of
  open sets admitting diffeomorphisms of class $C^{0,1}$ rectifying the boundary $\partial\mathcal{O}$. Denote $\varepsilon _1:=\varepsilon _0 (1+r_1)^{-1}$,
where $2r_1=\mathrm{diam}\,\Omega$.
\end{condition}
Obviously, the number $\varepsilon _1$ depends only on the domain $\mathcal{O}$ and the lattice $\Gamma$.
Note that Condition \ref{condition varepsilon} is ensured only by the assumption that $\partial\mathcal{O}$ is Lipschitz;
we imposed a more restrictive condition $\partial\mathcal{O}\in C^{1,1}$ in order to ensure estimate~\eqref{BD0 ^-1 L2->H2}.

\begin{theorem}[\cite{MSuPOMI}]
\label{Theorem Dirichlet L2}
Let $\mathcal{O}\subset \mathbb{R}^d$ be a bounded domain of class $C^{1,1}$.
Suppose that the assumptions of Subsections \textnormal{\ref{Subsection operatoer A_D,eps}--\ref{Subsection Effective operator}} are satisfied. Let
$$
\zeta =\vert \zeta \vert e^{i\phi}\in \mathbb{C}\setminus \mathbb{R}_+, \quad\vert \zeta \vert \geqslant 1.
$$
 Denote
\begin{equation*}
c(\phi):=\begin{cases}
\vert \sin \phi \vert ^{-1}, &\phi\in (0,\pi /2)\cup (3\pi /2 ,2\pi),\\
1, &\phi\in [\pi /2,3\pi /2].
\end{cases}
\end{equation*}
Suppose that $\varepsilon _1$ is subject to Condition~\textnormal{\ref{condition varepsilon}}. Then for $0<\varepsilon\leqslant \varepsilon _1$ we have
\begin{equation*}
\Vert (B_{D,\varepsilon}-\zeta Q_0^\varepsilon )^{-1}-(B_D^0-\zeta \overline{Q_0})^{-1}\Vert _{L_2(\mathcal{O})\rightarrow L_2(\mathcal{O})}\leqslant C_1 c(\phi)^5 \varepsilon \vert \zeta \vert ^{-1/2}.
\end{equation*}
The constant $C_1$ depends only on the problem data \eqref{problem data}.
\end{theorem}

We fix a linear continuous extension operator
\begin{equation}
\label{P_O H^1, H^2}
\begin{split}
P_\mathcal{O}: H^\sigma(\mathcal{O};\mathbb{C}^n)\rightarrow H^\sigma(\mathbb{R}^d;\mathbb{C}^n),\quad \sigma\geqslant 0.
\end{split}
\end{equation}
Such a ``universal'' extension operator exists for any Lipschitz bounded domain (see \cite{R}).
We have
\begin{equation}
\label{PO}
\| P_{\mathcal O} \|_{H^\sigma({\mathcal O}) \to H^\sigma({\mathbb R}^d)} \leqslant C_{\mathcal O}^{(\sigma)},\quad \sigma\geqslant 0,
\end{equation}
where the constant $C_{\mathcal O}^{(\sigma)}$ depends only on $\sigma$ and the domain ${\mathcal O}$.
By $R_\mathcal{O}$ we denote the operator of restriction of functions in $\mathbb{R}^d$ to the domain $\mathcal{O}$. We put
\begin{equation}
\label{K_D(eps,zeta)}
K_D(\varepsilon ;\zeta ):=R_{\mathcal{O}}\bigl([\Lambda ^\varepsilon ] b(\mathbf{D})+[\widetilde{\Lambda}^\varepsilon ]\bigr)S_\varepsilon P_{\mathcal{O}}(B_D^0-\zeta\overline{Q_0})^{-1}.
\end{equation}
The corrector \eqref{K_D(eps,zeta)} is a continuous mapping of $L_2(\mathcal{O};\mathbb{C}^n)$ to $H^1(\mathcal{O};\mathbb{C}^n)$. This can be easily checked with the help of Proposition~\ref{Proposition f^eps S_eps} and relations $\Lambda$, $\widetilde{\Lambda}\in\widetilde{H}^1(\Omega)$. Note that\break
$\Vert \varepsilon K_D(\varepsilon ;\zeta)\Vert _{L_2(\mathcal{O})\rightarrow H^1(\mathcal{O})}=O(1)$ for small $\varepsilon$ and fixed $\zeta$.

\begin{theorem}[\cite{MSuPOMI}]
\label{Theorem Dirichlet H1}
Suppose that the assumptions of Theorem~\textnormal{\ref{Theorem Dirichlet L2}} are satisfied. Let $K_D(\varepsilon;\zeta)$ be given by \eqref{K_D(eps,zeta)}.
Then for $\zeta\in\mathbb{C}\setminus\mathbb{R}_+$, $\vert \zeta\vert\geqslant 1$, and $0<\varepsilon \leqslant \varepsilon _1$ we have
\begin{equation}
\label{Th L2->H1}
\begin{split}
\Vert (B_{D,\varepsilon}-\zeta Q_0^\varepsilon )^{-1}-(B_D^0-\zeta \overline{Q_0})^{-1}-\varepsilon K_D(\varepsilon ;\zeta )\Vert _{L_2(\mathcal{O})\rightarrow H^1(\mathcal{O})}
\leqslant C_2 c(\phi)^2\varepsilon ^{1/2}\vert \zeta \vert ^{-1/4}+C_3 c(\phi)^4\varepsilon .
\end{split}
\end{equation}
Let $\widetilde{g}(\mathbf{x})$ be the matrix-valued function~\eqref{tilde g}. We put
\begin{equation}
\label{G(eps,zeta)}
G_D(\varepsilon ;\zeta)\! :=\! \widetilde{g}^\varepsilon S_\varepsilon b(\mathbf{D})P_\mathcal{O}(B_D^0\!-\!\zeta\overline{Q_0})^{-1}\!\!+\!g^\varepsilon\bigl( b(\mathbf{D})\widetilde{\Lambda}\bigr)^\varepsilon S_\varepsilon P_\mathcal{O}(B_D^0\!-\!\zeta \overline{Q_0})^{-1}.
\end{equation}
Then for $\zeta\in\mathbb{C}\setminus\mathbb{R}_+$, $\vert \zeta\vert\geqslant 1$, and $0<\varepsilon \leqslant \varepsilon _1$
the operator $g^\varepsilon b(\mathbf{D})(B_{D,\varepsilon}-\zeta Q_0^\varepsilon )^{-1}$ corresponding to the ``flux'' satisfies
\begin{equation}
\label{Th L_2->H^1 fluxes ell case}
\begin{split}
\Vert  g^\varepsilon b(\mathbf{D})(B_{D,\varepsilon}\!-\!\zeta Q_0^\varepsilon )^{-1}\!-\!G_D(\varepsilon ;\zeta)\Vert _{L_2(\mathcal{O})\rightarrow L_2(\mathcal{O})}
\!\leqslant\!\widetilde{C}_2c(\phi)^{5/2}\!\varepsilon ^{1/2}\vert \zeta \vert ^{-1/4}\!.
\end{split}
\end{equation}
The constants $C_2$, $C_3$, and $\widetilde{C}_2$ depend only on the problem data \eqref{problem data}.
\end{theorem}

 In \cite[Theorem~9.2]{MSuPOMI}, estimates in a wider domain of the spectral parameter were obtained. It was assumed that $\zeta\in\mathbb{C}\setminus[c_\flat,\infty)$, where $c_\flat$ is a common lower bound of the operators $\widetilde{B}_{D,\varepsilon}$ and $\widetilde{B}_D^0$. We put
\begin{equation}
\label{c_flat}
c_\flat:=4^{-1}\alpha _0\Vert g^{-1}\Vert ^{-1}_{L_\infty} \Vert Q_0\Vert _{L_\infty}^{-1} (\mathrm{diam}\,\mathcal{O})^{-2},
\end{equation}
using relations \eqref{c_*=}, \eqref{b D,eps >= H1-norm}, \eqref{Q_0=}, \eqref{tilde b D,eps=}, \eqref{b_D^0 ots 2}, \eqref{f_0<=}, and \eqref{tilde b_D^0 =}.

\begin{theorem}[\cite{MSuPOMI}]
\label{Theorem another approximation}
\label{Theorem Dr appr}
Let $\mathcal{O}\subset \mathbb{R}^d$ be a bounded domain of class $C^{1,1}$.
Suppose that the assumptions of Subsections~\textnormal{\ref{Subsection operatoer A_D,eps}--\ref{Subsection Effective operator}} are satisfied.
Let $K_D(\varepsilon;\zeta)$ be the corrector \eqref{K_D(eps,zeta)} and let $G_D(\varepsilon ;\zeta)$ be the operator~\eqref{G(eps,zeta)}.
Suppose that $\zeta\in\mathbb{C}\setminus [c_\flat,\infty)$, where $c_\flat$ is given by~\eqref{c_flat}. Denote $\psi :=\mathrm{arg}\,(\zeta -c_\flat)$, $0<\psi <2\pi$, and
\begin{equation}
\label{rho(zeta)}
\varrho _\flat(\zeta):=\begin{cases}
c(\psi)^2\vert \zeta -c_\flat\vert ^{-2}, &\vert \zeta -c_\flat\vert <1,\\
c(\psi)^2, &\vert \zeta -c_\flat\vert \geqslant 1.
\end{cases}
\end{equation}
Suppose that the number $\varepsilon_1$ is subject to Condition~\textnormal{\ref{condition varepsilon}}.
For $0<\varepsilon \leqslant \varepsilon _1$ we have
\begin{align}
\label{Th dr appr 1}
\Vert & (B_{D,\varepsilon}-\zeta Q_0^\varepsilon )^{-1}-(B_D^0-\zeta \overline{Q_0})^{-1}\Vert _{L_2(\mathcal{O})\rightarrow L_2(\mathcal{O})}\leqslant C_4\varepsilon\varrho _\flat(\zeta ) ,\\
\label{Th dr appr 2}
\begin{split}
\Vert &(B_{D,\varepsilon}-\zeta Q_0^\varepsilon )^{-1}-(B_D^0-\zeta \overline{Q_0})^{-1}
-\varepsilon K_D(\varepsilon ;\zeta )\Vert _{L_2(\mathcal{O})\rightarrow H^1(\mathcal{O})}
\\
&\leqslant C_5\bigl(\varepsilon ^{1/2}\varrho _\flat(\zeta )^{1/2}+\varepsilon \vert 1+\zeta \vert ^{1/2}\varrho _\flat(\zeta )\bigr),
\end{split}
\\
\label{Th dr appr 3}
\begin{split}
\Vert&  g^\varepsilon b(\mathbf{D})(B_{D,\varepsilon}-\zeta Q_0^\varepsilon )^{-1}-G_D(\varepsilon ;\zeta)\Vert _{L_2(\mathcal{O})\rightarrow L_2(\mathcal{O})}
\leqslant \widetilde{C}_5\bigl(\varepsilon ^{1/2}\varrho _\flat(\zeta )^{1/2}+\varepsilon \vert 1+\zeta \vert ^{1/2}\varrho _\flat(\zeta )\bigr).
\end{split}
\end{align}
The constants $C_4$, $C_5$, and $\widetilde{C}_5$ depend only on the problem data~\eqref{problem data}.
\end{theorem}

\begin{remark}
\label{Remark elliptic another appr}
\textnormal{1)}~In \eqref{rho(zeta)}, expression $c(\psi)^2\vert \zeta - c_\flat \vert ^{-2}$ is inverse to the square of the distance from $\zeta$ to $[c_\flat ,\infty)$. \textnormal{2)}~The number \eqref{c_flat} in Theorem~\textnormal{\ref{Theorem Dr appr}} can be replaced by any common lower bound of the operators $\widetilde{B}_{D,\varepsilon}$ and $\widetilde{B}_D^0$.
Let $\kappa >0$ be an arbitrarily small number. According to \eqref{Th dr appr 1} \textnormal{(}with $\zeta =0$\textnormal{)},
 $B_{D,\varepsilon}$ converges to $B_D^0$ in the norm-resolvent sense. Therefore, for sufficiently small $\varepsilon$ one can take $c_\flat =\lambda _1^0\Vert Q_0\Vert ^{-1}_{L_\infty} -\kappa$, where $\lambda _1^0$ is the first eigenvalue of the operator $B_D^0$.
Under such choice of $c_\flat$, the constants in estimates become dependent on $\kappa$.
\textnormal{3)}~It makes sense to use estimates \eqref{Th dr appr 1}--\eqref{Th dr appr 3} for bounded values of $\vert\zeta\vert$ and small $\varepsilon\varrho _\flat (\zeta)$. In this case, the value $\varepsilon ^{1/2}\varrho _\flat (\zeta )^{1/2}+\varepsilon \vert 1+\zeta \vert ^{1/2}\varrho _\flat (\zeta)$ is controlled in terms of $C\varepsilon ^{1/2}\varrho _\flat (\zeta )^{1/2}$. For large $\vert\zeta\vert$ and for $\phi$ separated from the points $0$ and $2\pi$, it is preferable to use Theorems \textnormal{\ref{Theorem Dirichlet L2}} and \textnormal{\ref{Theorem Dirichlet H1}}.
\end{remark}

\subsection{Removal of the smoothing operator in the corrector}
It turns out that the smoothing operator in the corrector can be removed under some additional assumptions on the matrix-valued functions $\Lambda (\mathbf{x})$ and  $\widetilde{\Lambda}(\mathbf{x})$.

\begin{condition}
\label{Condition Lambda in L infty}
Suppose that the $\Gamma$-periodic solution $\Lambda (\mathbf{x})$ of problem~\eqref{Lambda problem} is bounded, i.~e., $\Lambda\in L_\infty (\mathbb{R}^d)$.
\end{condition}

Some cases ensuring that Condition~\ref{Condition Lambda in L infty} is satisfied were distinguished in \cite[Lemma~8.7]{BSu06}.

\begin{proposition}[\cite{BSu06}]
\label{Proposition Lambda in L infty <=}
Suppose that at least one of the following assumptions is satisfied\textnormal{:}

\noindent
$1^\circ )$ $d\leqslant 2${\rm ;}

\noindent
$2^\circ )$ dimension $d\geqslant 1$ is arbitrary, and the differential expression $A_\varepsilon$ is given by $A_\varepsilon =\mathbf{D}^* g^\varepsilon (\mathbf{x})\mathbf{D}$, where $g(\mathbf{x})$ is a symmetric matrix with real entries{\rm ;}

\noindent
$3^\circ )$ dimension $d$ is arbitrary, and $g^0=\underline{g}$, i.~e., relations \eqref{underline-g} are satisfied.

\noindent Then Condition~\textnormal{\ref{Condition Lambda in L infty}} holds.
\end{proposition}

In order to remove $S_\varepsilon$ in the term of the corrector involving $\widetilde{\Lambda}^\varepsilon$, it suffices to impose the following condition.

\begin{condition}
\label{Condition tilde Lambda in Lp}
Suppose that the $\Gamma$-periodic solution $\widetilde{\Lambda}(\mathbf{x})$ of problem~\eqref{tildeLambda_problem} is such that
\begin{equation*}
\widetilde{\Lambda}\in L_p(\Omega),\quad p=2 \;\mbox{for}\;d=1,\quad p>2\;\mbox{for}\;d=2, \quad p=d \;\mbox{for}\;d\geqslant 3.
\end{equation*}
\end{condition}

The following result was checked in \cite[Proposition~8.11]{SuAA}.

\begin{proposition}[\cite{SuAA}]
\label{Proposition tilde Lambda in Lp if}
Suppose that at least one of the following assumptions is satisfied\textnormal{:}

\noindent
$1^\circ )$ $d\leqslant 4${\rm ;}

\noindent
$2^\circ )$ dimension $d$ is arbitrary, and $A_{\varepsilon}$ is given by $A_{\varepsilon} =\mathbf{D}^* g^\varepsilon (\mathbf{x})\mathbf{D}$,
where $g(\mathbf{x})$ is a symmetric matrix with real entries.

\noindent
Then Condition~\textnormal{\ref{Condition tilde Lambda in Lp}} is satisfied.
\end{proposition}

\begin{remark}
\label{Remark scalar problem}
If $A_{\varepsilon} =\mathbf{D}^* g^\varepsilon (\mathbf{x})\mathbf{D}$, where $g(\mathbf{x})$ is a symmetric matrix with real entries, then from  \textnormal{\cite[Chapter~{\rm III}, Theorem~{\rm 13.1}]{LaU}} it follows that $\Lambda,\widetilde{\Lambda}\in L_\infty$ and the norm $\Vert\Lambda\Vert _{L_\infty}$ does not exceed a constant depending on $d$, $\Vert g\Vert _{L_\infty}$, $\Vert g^{-1}\Vert _{L_\infty}$, and $\Omega$, while the norm $\Vert \widetilde{\Lambda}\Vert _{L_\infty }$ is controlled in terms of $d$, $\rho$, $\Vert g\Vert _{L_\infty}$, $\Vert g^{-1}\Vert _{L_\infty}$, $\Vert a_j\Vert _{L_\rho (\Omega)}$, $j=1,\dots ,d$, and $\Omega$. In this case, Conditions~\textnormal{\ref{Condition Lambda in L infty}}
and~\textnormal{\ref{Condition tilde Lambda in Lp}} hold.
\end{remark}

In \cite[Theorem~7.6]{MSuPOMI} the following result was obtained.

\begin{theorem}[\cite{MSuPOMI}]
\label{Theorem no S-eps}
Suppose that the assumptions of Theorem~\textnormal{\ref{Theorem Dirichlet H1}} are satisfied. Suppose that $\Lambda(\mathbf{x})$ is subject to Condition~\textnormal{\ref{Condition Lambda in L infty},}
and $\widetilde{\Lambda}(\mathbf{x})$ satisfies Condition~\textnormal{\ref{Condition tilde Lambda in Lp}}. We put
\begin{align}
\label{K_D^0n elliptic}
&K_D^0(\varepsilon ;\zeta ):=(
 \Lambda ^\varepsilon  b(\mathbf{D})+
 \widetilde{\Lambda}^\varepsilon  )(B_D^0-\zeta \overline{Q_0})^{-1},
\\
\label{G3(eps;zeta)}
&G_D^0(\varepsilon ;\zeta ):= \widetilde{g}^\varepsilon  b(\mathbf{D})(B_D^0-\zeta\overline{Q_0})^{-1}
+g^\varepsilon \bigl(b(\mathbf{D})\widetilde{\Lambda}\bigr)^\varepsilon (B_D^0-\zeta\overline{Q_0})^{-1}.
\end{align}
Then for $\zeta \in \mathbb{C}\setminus \mathbb{R}_+$, $\vert \zeta\vert\geqslant 1$, and $0<\varepsilon\leqslant\varepsilon_1$ we have
\begin{align*}
\begin{split}
\Vert & (B_{D,\varepsilon}-\zeta Q_0^\varepsilon )^{-1}-(B_D^0-\zeta \overline{Q_0})^{-1}-\varepsilon K_D^0(\varepsilon;\zeta)
\Vert _{L_2(\mathcal{O})\rightarrow H^1(\mathcal{O})}
\leqslant C_2 c(\phi)^2\varepsilon ^{1/2}\vert \zeta \vert ^{-1/4}+C_{6}c(\phi)^4\varepsilon ,
\end{split}
\\
\begin{split}
\Vert & g^\varepsilon b(\mathbf{D})(B_{D,\varepsilon}-\zeta Q_0^\varepsilon)^{-1}- G_D^0(\varepsilon ;\zeta)\Vert _{L_2(\mathcal{O})\rightarrow L_2(\mathcal{O})}
\leqslant \widetilde{C}_2 c(\phi)^2\varepsilon ^{1/2}\vert \zeta\vert ^{-1/4}+\widetilde{C}_{6}c(\phi)^4\varepsilon .
\end{split}
\end{align*}
Here the constants $C_2$, $\widetilde{C}_2$ are as in \eqref{Th L2->H1} and \eqref{Th L_2->H^1 fluxes ell case}.
The constants $C_{6}$ and $\widetilde{C}_{6}$ depend only on the problem data \eqref{problem data}, $p$, and the normsì $\Vert \Lambda\Vert _{L_\infty}$, $\Vert \widetilde{\Lambda}\Vert _{L_p(\Omega)}$.
\end{theorem}

Approximations in a wider domain of the spectral parameter were found in~\cite[Theorem~9.8]{MSuPOMI}.

\begin{theorem}[\cite{MSuPOMI}]
\label{Theorem Dr appr no S_eps}
Suppose that the assumptions of Theorem~\textnormal{\ref{Theorem Dr appr}} and Conditions~\textnormal{\ref{Condition Lambda in L infty}}, \textnormal{\ref{Condition tilde Lambda in Lp}}
are satisfied. Let $K_D^0(\varepsilon ;\zeta)$ be the corrector \eqref{K_D^0n elliptic}. Let $G_D^0(\varepsilon ;\zeta)$ be given by~\eqref{G3(eps;zeta)}.
Then for $0<\varepsilon\leqslant \varepsilon _1$ and $\zeta\in\mathbb{C}\setminus[c_\flat ,\infty)$ we have
\begin{align*}
\begin{split}
\Vert &(B_{D,\varepsilon}-\zeta Q_0^\varepsilon )^{-1}-(B_D^0-\zeta\overline{Q_0})^{-1}-\varepsilon K_D^0(\varepsilon ;\zeta)
\Vert _{L_2(\mathcal{O})\rightarrow H^1(\mathcal{O})}
\leqslant
C_7 \bigl(\varepsilon ^{1/2}\varrho _\flat (\zeta)^{1/2}
+\varepsilon \vert 1+\zeta \vert ^{1/2}\varrho _\flat (\zeta)\bigr),
\end{split}
\\
\begin{split}
\Vert &g^\varepsilon b(\mathbf{D})(B_{D,\varepsilon}-\zeta Q_0^\varepsilon )^{-1}-G_D^0(\varepsilon;\zeta)\Vert _{L_2(\mathcal{O})\rightarrow L_2(\mathcal{O})}
\leqslant \widetilde{C}_7
\bigl(\varepsilon ^{1/2}\varrho _\flat (\zeta)^{1/2}
+\varepsilon \vert 1+\zeta \vert ^{1/2}\varrho _\flat (\zeta)\bigr).
\end{split}
\end{align*}
Here the constants $C_{7}$ and $\widetilde{C}_{7}$ depend only on the problem data \eqref{problem data}, $p$, and the norms $\Vert \Lambda\Vert _{L_\infty}$, $\Vert \widetilde{\Lambda}\Vert _{L_p(\Omega)}$.
\end{theorem}

According to \cite[Remarks 7.9 and 9.9]{MSuPOMI}, we observe the following.

\begin{remark}
\label{Remark elliptic no S-eps}
If only Condition~\textnormal{\ref{Condition Lambda in L infty}} \textnormal{(}respectively, Condition~\textnormal{\ref{Condition tilde Lambda in Lp}}\textnormal{)}
is satisfied, then the smoothing operator $S_\varepsilon$ can be removed in the term of the corrector involving $\Lambda^\varepsilon$ \textnormal{(}respectively, in the term containing $\widetilde{\Lambda}^\varepsilon$\textnormal{)}.
\end{remark}

\subsection{The case where the corrector is equal to zero}

Suppose that $g^0=\overline{g}$, i.~e., relations \eqref{overline-g} hold.
Then the $\Gamma$-periodic solution of problem~\eqref{Lambda problem} is equal to zero: $\Lambda (\mathbf{x})=0$.
Suppose in addition that
\begin{equation}
\label{sum Dj aj =0}
\sum _{j=1}^d D_j a_j(\mathbf{x})^* =0.
\end{equation}
Then the $\Gamma$-periodic solution of problem~\eqref{tildeLambda_problem} is also equal to zero: $\widetilde{\Lambda}(\mathbf{x})=0$.
According to \cite[Propositions 7.10 and 9.12]{MSuPOMI}, in this case the $(L_2\rightarrow H^1)$-estimate of sharp order $O(\varepsilon)$ holds.

\begin{proposition}[\cite{MSuPOMI}]
\label{Proposition K=0}
Suppose that relations~\eqref{overline-g} and~\eqref{sum Dj aj =0} are satisfied.

\noindent
$1^\circ$. Under the assumptions of Theorem~\textnormal{\ref{Theorem Dirichlet L2}}, for $\zeta\in\mathbb{C}\setminus\mathbb{R}_+$, $\vert\zeta\vert\geqslant 1$, and $0<\varepsilon\leqslant 1$ we have
\begin{equation}
\label{Pr K=0}
\Vert (B_{D,\varepsilon}-\zeta Q_0^\varepsilon )^{-1}-(B_D^0-\zeta \overline{Q_0})^{-1}\Vert _{L_2(\mathcal{O})\rightarrow H^1(\mathcal{O})}
\leqslant C_8 c(\phi)^4\varepsilon.
\end{equation}

\noindent
$2^\circ$. Under the assumptions of Theorem~\textnormal{\ref{Theorem Dr appr}},
for $\zeta\in\mathbb{C}\setminus[c_\flat,\infty)$ and $0<\varepsilon\leqslant 1$ we have
\begin{equation}
\label{Pf K=0 dr. appr.}
\begin{split}
\Vert (B_{D,\varepsilon}-\zeta Q_0^\varepsilon )^{-1}-(B_D^0&-\zeta\overline{Q_0})^{-1}\Vert _{L_2(\mathcal{O})\rightarrow H^1(\mathcal{O})}
\leqslant (C_9 +C_{10} \vert 1+\zeta\vert ^{1/2}) \varepsilon\varrho _\flat (\zeta).
\end{split}
\end{equation}
The constants $C_8$, $C_9$, and $C_{10}$ depend only on the problem data \eqref{problem data}.
\end{proposition}

\subsection{Estimates in a strictly interior subdomain}
It is possible to improve the $H^1$-estimates in a strictly interior subdomain $\mathcal{O}'$ of the domain $\mathcal{O}$.
In Theorems~8.1 and~9.14 of \cite{MSuPOMI}, the following result was obtained.

\begin{theorem}[\cite{MSuPOMI}]
\label{Theorem O'}
Let $\mathcal{O}'$ be a strictly interior subdomain of the domain $\mathcal{O}$. Denote
\begin{equation}
\label{delta= 1.62a}
\delta :=\min \left\lbrace 1 ;\mathrm{dist}\,\lbrace \mathcal{O}';\partial \mathcal{O}\rbrace \right\rbrace .
\end{equation}

\noindent
$1^\circ$.
Under the assumptions of Theorem~\textnormal{\ref{Theorem Dirichlet H1}},
for $\zeta\in\mathbb{C}\setminus\mathbb{R}_+$, $\vert\zeta\vert\geqslant 1$, and \hbox{$0<\varepsilon\leqslant\varepsilon _1$} we have
\begin{align*}
\begin{split}
\Vert  &(B_{D,\varepsilon}-\zeta Q_0^\varepsilon )^{-1}-(B_D^0-\zeta \overline{Q_0})^{-1}-\varepsilon K_D(\varepsilon ;\zeta )\Vert _{L_2(\mathcal{O})\rightarrow H^1(\mathcal{O'})}
\leqslant c(\phi)^6\varepsilon (C_{11}'\vert\zeta\vert ^{-1/2}\delta ^{-1}+C_{11}'') ,
\end{split}
\\
\begin{split}
\Vert &g^\varepsilon b(\mathbf{D})(B_{D,\varepsilon}-\zeta Q_0^\varepsilon )^{-1}-G_D(\varepsilon;\zeta)\Vert _{L_2(\mathcal{O})\rightarrow L_2(\mathcal{O}')}
\leqslant c(\phi)^6\varepsilon  \bigl(\widetilde{C}_{11}'\vert\zeta\vert ^{-1/2}\delta ^{-1}+\widetilde{C}_{11}''\bigr) .
\end{split}
\end{align*}
The constants $C_{11}', C_{11}'', \widetilde{C}_{11}'$, and $\widetilde{C}_{11}''$ depend only on the problem data~\eqref{problem data}.

\smallskip

\noindent
$2^\circ$. Under the assumptions of Theorem~\textnormal{\ref{Theorem Dr appr}},
for $\zeta\in\mathbb{C}\setminus[c_\flat,\infty)$ and $0<\varepsilon\leqslant\varepsilon _1$ we have
\begin{align}
\label{Th Dr appr strictly interior subdomain solutions}
\begin{split}
\Vert &(B_{D,\varepsilon}-\zeta Q_0^\varepsilon )^{-1}-(B_D^0-\zeta \overline{Q_0})^{-1}-\varepsilon K_D(\varepsilon ;\zeta )\Vert _{L_2(\mathcal{O})\rightarrow H^1(\mathcal{O}')}
\\
&\leqslant \varepsilon
\bigl(
C_{12}'\delta ^{-1}\varrho _\flat (\zeta )^{1/2}+C_{12}''\vert 1+\zeta\vert ^{1/2}\varrho _\flat (\zeta)
\bigr)
 ,
\end{split}
\\
\label{Th Dr appr strictly interior subdomain fluxes}
\begin{split}
\Vert &g^\varepsilon b(\mathbf{D})(B_{D,\varepsilon}-\zeta Q_0^\varepsilon )^{-1}-G_D(\varepsilon;\zeta)\Vert _{L_2(\mathcal{O})\rightarrow L_2(\mathcal{O}')}
\leqslant \varepsilon  \bigl(
\widetilde{C}_{12}'\delta ^{-1}\varrho _\flat (\zeta )^{1/2}+\widetilde{C}_{12}''\vert 1+\zeta\vert ^{1/2}\varrho _\flat (\zeta)
\bigr)
.
\end{split}
\end{align}
The constants $C_{12}',C_{12}''$, and $\widetilde{C}_{12}',\widetilde{C}_{12}''$\! depend only on the problem data~\eqref{problem data}.
\end{theorem}

If the matrix-valued functions $\Lambda (\mathbf{x})$ and $\widetilde{\Lambda}(\mathbf{x})$ satisfy some additional assumptions, this result remains true
with a simpler corrector. Approximations for $\zeta\in\mathbb{C}\setminus\mathbb{R}_+$, $\vert\zeta\vert\geqslant 1$, were found in~\cite[Theorem~8.2]{MSuPOMI}.

\begin{theorem}[\cite{MSuPOMI}]
\label{Theorem O' no S_eps}
Suppose that the assumptions of Theorem~\textnormal{\ref{Theorem O'}($1^\circ$)} are satisfied. Suppose that the matrix-valued functions
  $\Lambda (\mathbf{x})$ and $\widetilde{\Lambda}(\mathbf{x})$ satisfy Conditions~\textnormal{\ref{Condition Lambda in L infty}} and~\textnormal{\ref{Condition tilde Lambda in Lp}}, respectively. Let $K_D^0(\varepsilon ;\zeta)$ and
$G_D^0(\varepsilon ;\zeta)$ be the operators defined by \eqref{K_D^0n elliptic} and \eqref{G3(eps;zeta)}. Then for $0<\varepsilon\leqslant\varepsilon _1$ and $\zeta\in\mathbb{C}\setminus\mathbb{R}_+$, $\vert \zeta\vert \geqslant 1$, we have
\begin{align*}
\begin{split}
\Vert &(B_{D,\varepsilon}-\zeta Q_0^\varepsilon )^{-1}-(B_D^0-\zeta \overline{Q_0})^{-1}
-\varepsilon K_D^0(\varepsilon ;\zeta)
\Vert _{L_2(\mathcal{O})\rightarrow H^1(\mathcal{O}')}
\leqslant c(\phi )^6\varepsilon (C_{11}'\vert \zeta\vert ^{-1/2}\delta ^{-1}+C_{13}) ,
\end{split}
\\
\begin{split}
\Vert &g^\varepsilon b(\mathbf{D})(B_{D,\varepsilon}-\zeta Q_0^\varepsilon )^{-1}-G_D^0(\varepsilon;\zeta)\Vert _{L_2(\mathcal{O})\rightarrow L_2(\mathcal{O}')}
\leqslant c(\phi )^6\varepsilon  (\widetilde{C}_{11}'\vert \zeta\vert ^{-1/2}\delta ^{-1}+\widetilde{C}_{13}).
\end{split}
\end{align*}
The constants $C_{11}'$ and $\widetilde{C}_{11}'$ are as in Theorem~\textnormal{\ref{Theorem O'}}.
The constants $C_{13}$ and $\widetilde{C}_{13}$ depend on the problem data \eqref{problem data}, $p$, and the norms $\Vert \Lambda\Vert _{L_\infty}$, $\Vert \widetilde{\Lambda}\Vert _{L_p(\Omega)}$.
\end{theorem}

Approximations in a wider domain of the parameter~$\zeta$ are obtained in~\cite[Theorem~9.15]{MSuPOMI}.

\begin{theorem}[\cite{MSuPOMI}]
\label{Theorem Dr appr strictly interior subdomain no S_eps}
Suppose that the assumptions of Theorem~\textnormal{\ref{Theorem O'}($2^\circ$)} are satisfied.
Suppose that the matrix-valued functions $\Lambda (\mathbf{x})$ and $\widetilde{\Lambda}(\mathbf{x})$ are subject to Conditions~\textnormal{\ref{Condition Lambda in L infty}} and \textnormal{\ref{Condition tilde Lambda in Lp}}, respectively. Let $K_D^0(\varepsilon ;\zeta )$ be the corrector \eqref{K_D^0n elliptic}, and let
$G_D^0(\varepsilon;\zeta)$ be the operator \eqref{G3(eps;zeta)}. Then for $\zeta\in\mathbb{C}\setminus [c_\flat ,\infty)$ and $0<\varepsilon\leqslant\varepsilon _1$ we have
\begin{align*}
\begin{split}
\Vert &(B_{D,\varepsilon}-\zeta Q_0^\varepsilon )^{-1}-(B_D^0-\zeta \overline{Q_0})^{-1}-\varepsilon K_D^0(\varepsilon;\zeta)
\Vert _{L_2(\mathcal{O})\rightarrow H^1(\mathcal{O}')}
\\
&\leqslant \varepsilon \bigl( C_{12}'\delta ^{-1} \varrho _\flat (\zeta )^{1/2}+C_{14}\vert 1+\zeta\vert ^{1/2}\varrho _\flat (\zeta )\bigr)
,
\end{split}
\\
\begin{split}
\Vert &g^\varepsilon b(\mathbf{D})(B_{D,\varepsilon}-\zeta Q_0^\varepsilon )^{-1}-G_D^0(\varepsilon;\zeta)\Vert _{L_2(\mathcal{O})\rightarrow L_2(\mathcal{O}')}
\leqslant \varepsilon \bigl( \widetilde{C} _{12}'\delta ^{-1} \varrho _\flat (\zeta )^{1/2}+\widetilde{C}_{14}\vert 1+\zeta\vert ^{1/2}\varrho _\flat (\zeta )\bigr)
.
\end{split}
\end{align*}
Here the constants $C_{12}'$ and $\widetilde{C}_{12}'$ are as in \eqref{Th Dr appr strictly interior subdomain solutions} and \eqref{Th Dr appr strictly interior subdomain fluxes}.
The constants $C_{14}$ and $\widetilde{C}_{14}$ depend on the problem data \eqref{problem data}, $p$, and the norms $\Vert \Lambda\Vert _{L_\infty}$, $\Vert \widetilde{\Lambda}\Vert _{L_p(\Omega)}$.
\end{theorem}

\section{Statement of the problem. Main results}

\label{Section 2}

\subsection{Statement of the problem}
We study the behavior of the solution of the first initial boundary-value problem
\begin{equation}
\label{first initial-boundary value problem}
\begin{cases}
Q_0^\varepsilon (\mathbf{x}) \frac{\partial \mathbf{u}_\varepsilon}{\partial t}(\mathbf{x},t)&=-B_{\varepsilon}\mathbf{u}_\varepsilon (\mathbf{x},t)
,
\quad\mathbf{x}\in\mathcal{O},\quad t>0;
\\
\mathbf{u}_\varepsilon (\,\cdot\, ,t) \vert _{\partial \mathcal{O}}&=0,\quad t>0;\\
Q_0^\varepsilon (\mathbf{x}) \mathbf{u}_\varepsilon (\mathbf{x},0)&=\boldsymbol{\varphi}(\mathbf{x}),\quad \mathbf{x}\in \mathcal{O}.
\end{cases}
\end{equation}
Here $\boldsymbol{\varphi}\in L_2(\mathcal{O};\mathbb{C}^n)$. (The solution is understood in the weak sense.)
Let us find relation between $\mathbf{u}_\varepsilon (\,\cdot\, ,t)$ and $\boldsymbol{\varphi}$.
According to \eqref{Q_0=}, the function $\mathbf{s}_\varepsilon (\mathbf{x} ,t):=\left(f^\varepsilon (\mathbf{x})\right)^{-1}\mathbf{u}_\varepsilon (\mathbf{x} ,t)$
is the solution of the problem
\begin{equation*}
\begin{cases}
\frac{\partial \mathbf{s}_\varepsilon}{\partial t}(\mathbf{x} ,t)&=-\widetilde{B}_{\varepsilon}\mathbf{s}_\varepsilon (\mathbf{x},t)\quad\mathbf{x}\in\mathcal{O},\quad t>0;\\
\mathbf{s}_\varepsilon (\,\cdot\, ,t)\vert _{\partial \mathcal{O}}&=0,\quad t>0;\\
\mathbf{s}_\varepsilon (\mathbf{x},0)&= (f^\varepsilon (\mathbf{x}))^*\boldsymbol{\varphi}(\mathbf{x}),\quad \mathbf{x}\in \mathcal{O}.
\end{cases}
\end{equation*}
Then $\mathbf{s}_\varepsilon (\,\cdot\, ,t)=e^{-\widetilde{B}_{D,\varepsilon}t} (f^\varepsilon )^*\boldsymbol{\varphi}$ and
$\mathbf{u}_\varepsilon (\,\cdot\, ,t)= f^\varepsilon \mathbf{s}_\varepsilon (\,\cdot\, ,t)=f^\varepsilon e^{-\widetilde{B}_{D,\varepsilon} t}(f^\varepsilon)^*\boldsymbol{\varphi}$.

\textit{Our goal} is to study the behavior of the generalized solution $\mathbf{u}_\varepsilon$ of the first initial boundary-value problem \eqref{first initial-boundary value problem} in the small period limit. In other words, we are interested in approximations of the sandwiched operator exponential
$f^\varepsilon e^{-\widetilde{B}_{D,\varepsilon}t} (f^\varepsilon)^*$ for small $\varepsilon$.

The corresponding effective problem is given by
\begin{equation}
\label{effective first problem}
\begin{cases}
\overline{Q_0}\frac{\partial \mathbf{u}_0}{\partial t}(\mathbf{x},t)&=-B^0\mathbf{u}_0 (\mathbf{x},t),
\quad\mathbf{x}\in\mathcal{O},\quad t>0;
\\
\mathbf{u}_0(\,\cdot\, ,t) \vert _{\partial \mathcal{O}}&=0,\quad t>0;\\
\overline{Q_0}\mathbf{u}_0 (\mathbf{x},0)&=\boldsymbol{\varphi}(\mathbf{x}),\quad \mathbf{x}\in \mathcal{O}.
\end{cases}
\end{equation}
By \eqref{f_0=}, the solution of the effective problem is given by
\begin{equation}
\label{u_0=}
\mathbf{u}_0 (\,\cdot\, ,t)=f_0 e^{-\widetilde{B}_{D}^0 t}f_0\boldsymbol{\varphi}(\,\cdot\, ).
\end{equation}

\subsection{The properties of the operator exponential}
We prove the following simple statement about estimates for the operator exponentials $e^{ -\widetilde{B}_{D,\varepsilon}t}$ and $e^{-\widetilde{B}_D^0 t}$.

\begin{lemma}
\label{Lemma properties of operator exponential}
For $0<\varepsilon\leqslant 1$ we have
\begin{align}
\label{exp tilde B_D,eps L2-L2}
\Vert e^{-\widetilde{B}_{D,\varepsilon}t}\Vert _{L_2(\mathcal{O})\rightarrow L_2(\mathcal{O})}
&\leqslant e^{-c_\flat t},\quad t\geqslant 0,
\\
\label{exp tilde B_D,eps L2-H1}
\Vert f^\varepsilon e^{-\widetilde{B}_{D,\varepsilon}t}\Vert _{L_2(\mathcal{O})\rightarrow H^1(\mathcal{O})}
&\leqslant c_3t^{-1/2}e^{-c_\flat t/2},\quad t>0,
\\
\label{exp tilde B_D^0 L2-L2}
\Vert e^{-\widetilde{B}_{D}^0t}\Vert _{L_2(\mathcal{O})\rightarrow L_2(\mathcal{O})}
&\leqslant e^{-c_\flat t},\quad t\geqslant 0,
\\
\label{exp tilde B_D^0 L2-H1}
\Vert f_0 e^{-\widetilde{B}_{D}^0t}\Vert _{L_2(\mathcal{O})\rightarrow H^1(\mathcal{O})}
&\leqslant c_3t^{-1/2}e^{-c_\flat t/2},\quad t>0,
\\
\label{exp tilde B_D^0 L2-H2}
\Vert f_0 e^{-\widetilde{B}_{D}^0t}\Vert _{L_2(\mathcal{O})\rightarrow H^2(\mathcal{O})}
&\leqslant \widetilde{c}t^{-1}e^{-c_\flat t/2},\quad t>0.
\end{align}
Here the constants $c_3$ and $c_\flat$ are given by \eqref{c_3} and \eqref{c_flat}. The constant $\widetilde{c}$ depends only on the problem data~\eqref{problem data}.
\end{lemma}

\begin{proof}
Since the number $c_\flat$ defined by \eqref{c_flat} is a common lower bound of the operators $\widetilde{B}_{D,\varepsilon}$ and $\widetilde{B}_{D}^0$,
estimates~\eqref{exp tilde B_D,eps L2-L2} and~\eqref{exp tilde B_D^0 L2-L2} are obvious.

By \eqref{H^1-norm <= BDeps^1/2} and \eqref{tilde b D,eps=},
\begin{equation}
\label{exp B_d,eps L-2-H1 start of est}
\begin{split}
\Vert f^\varepsilon e^{-\widetilde{B}_{D,\varepsilon}t}\Vert _{L_2(\mathcal{O})\rightarrow H^1(\mathcal{O})}
\leqslant c_3 \Vert B_{D,\varepsilon}^{1/2}f^\varepsilon e^{-\widetilde{B}_{D,\varepsilon}t}\Vert _{L_2(\mathcal{O})\rightarrow L_2(\mathcal{O})}
=c_3\Vert \widetilde{B}_{D,\varepsilon}^{1/2}e^{-\widetilde{B}_{D,\varepsilon}t}\Vert _{L_2(\mathcal{O})\rightarrow L_2(\mathcal{O})}.
\end{split}
\end{equation}
Since $\widetilde{B}_{D,\varepsilon}\geqslant c_\flat I$, then
\begin{equation}
\label{tilde B_D,eps exp}
\begin{split}
\Vert \widetilde{B}_{D,\varepsilon}^{1/2}e^{-\widetilde{B}_{D,\varepsilon}t}\Vert _{L_2(\mathcal{O})\rightarrow L_2(\mathcal{O})}
\leqslant \sup _{x\geqslant c_\flat}x^{1/2}e^{-xt}
\leqslant e^{-c_\flat t/2}\sup _{x\geqslant c_\flat}x^{1/2}e^{-xt/2}
\leqslant t^{-1/2}e^{-c_\flat t/2}.
\end{split}
\end{equation}
Combining this with~\eqref{exp B_d,eps L-2-H1 start of est}, we obtain inequality~\eqref{exp tilde B_D,eps L2-H1}.
Similarly, \eqref{H^1-norm <= BD0^1/2} and~\eqref{tilde b_D^0 =} imply estimate~\eqref{exp tilde B_D^0 L2-H1}.

From \eqref{BD0 ^-1 L2->H2}, \eqref{f_0<=}, and the identity $\widetilde{B}_D^0=f_0B_D^0f_0$ it follows that
\begin{equation*}
\begin{split}
\Vert f_0 e^{-\widetilde{B}_{D}^0 t}\Vert _{L_2(\mathcal{O})\rightarrow H^2 (\mathcal{O})}
\leqslant
\widehat{c}\Vert B_{D}^0f_0 e^{-\widetilde{B}_D^0 t}\Vert _{L_2(\mathcal{O})\rightarrow L_2(\mathcal{O})}
\leqslant \widehat{c}\Vert f^{-1}\Vert _{L_\infty}\Vert \widetilde{B}_D^0 e^{-\widetilde{B}_D^0t}\Vert _{L_2(\mathcal{O})\rightarrow L_2(\mathcal{O})}.
\end{split}
\end{equation*}
Hence,
\begin{equation*}
\begin{split}
\Vert & f_0 e^{-\widetilde{B}_{D}^0 t}\Vert _{L_2(\mathcal{O})\rightarrow H^2 (\mathcal{O})}
\leqslant
\widehat{c}\Vert f^{-1}\Vert _{L_\infty}\sup _{x\geqslant c_\flat}x e^{-xt}
\leqslant \widehat{c}\Vert f ^{-1}\Vert _{L_\infty}t^{-1}e^{-c_\flat t/2}.
\end{split}
\end{equation*}
This proves estimate~\eqref{exp tilde B_D^0 L2-H2} with the constant $\widetilde{c}=\widehat{c}\Vert f^{-1}\Vert _{L_\infty}$.
\end{proof}

\subsection{Approximation of the solution in $L_2(\mathcal{O};\mathbb{C}^n)$}

\begin{theorem}
\label{Theorem L2 exp}
Let $\mathcal{O}\subset \mathbb{R}^d$ be a bounded domain of class~$C^{1,1}$. Suppose that the assumptions of Subsections~\textnormal{\ref{Subsection operatoer A_D,eps}--\ref{Subsection Effective operator}} are satisfied.
Let $B_{D,\varepsilon}$ be the operator in $L_2(\mathcal{O};\mathbb{C}^n)$ corresponding to the quadratic form~\eqref{b_D,eps}.
Let $B_D^0$ be the operator in~$L_2(\mathcal{O};\mathbb{C}^n)$ given by expression~\eqref{B_D^0} on~$H^2(\mathcal{O};\mathbb{C}^n)\cap H^1_0(\mathcal{O};\mathbb{C}^n)$. We put~$\widetilde{B}_{D,\varepsilon}=(f^\varepsilon)^*B_{D,\varepsilon}f^\varepsilon$ and $\widetilde{B}_D^0=f_0B_D^0f_0$, where the matrix-valued function~$f$ is defined by~\eqref{Q_0=},
and the matrix~$f_0$ is given by~\eqref{f_0=}.
  Let~$\mathbf{u}_\varepsilon$ be the solution of problem~\eqref{first initial-boundary value problem}, and let~$\mathbf{u}_0$ be the solution of the corresponding effective problem~\eqref{effective first problem}. Suppose that the number~$\varepsilon _1$ is subject to Condition~\textnormal{\ref{condition varepsilon}}.
  Then for~\hbox{$0<\varepsilon \leqslant \varepsilon_1$} we have
\begin{equation*}
\Vert \mathbf{u}_\varepsilon (\,\cdot\, ,t)-\mathbf{u}_0(\,\cdot\, ,t)\Vert _{L_2(\mathcal{O})}\leqslant C_{15}\varepsilon  (t+\varepsilon ^2)^{-1/2}e^{-c_\flat t/2}\Vert \boldsymbol{\varphi}\Vert _{L_2(\mathcal{O})} ,\quad t\geqslant 0.
\end{equation*}
In the operator terms,
\begin{equation}
\label{Th_exp_L_2}
\begin{split}
\Vert f^\varepsilon e^{-\widetilde{B}_{D,\varepsilon}t}(f^\varepsilon)^*-f_0 e^{-\widetilde{B}_D^0t}f_0\Vert _{L_2(\mathcal{O})\rightarrow L_2(\mathcal{O})}\leqslant C_{15}\varepsilon  (t+\varepsilon ^2)^{-1/2}e^{-c_\flat t/2},
\quad t\geqslant 0.
\end{split}
\end{equation}
Here the constant $c_\flat$ is given by~\eqref{c_flat}.
The constant $C_{15}$ depends only on the problem data~\eqref{problem data}.
\end{theorem}

\begin{proof}
The proof is based on the results of Theorems~\ref{Theorem Dirichlet L2}, \ref{Theorem another approximation}, and representations for the exponentials of
the operators $\widetilde{B}_{D,\varepsilon}$, $\widetilde{B}_D^0$ in terms of the contour integrals of the corresponding resolvents.

We have (see, e.~g., \cite[Chapter~IX, Section~1.6]{Ka})
\begin{equation}
\label{2.1a}
e^{-\widetilde{B}_{D,\varepsilon}t}=-\frac{1}{2\pi i}\int\limits_\gamma e^{-\zeta t}(\widetilde{B}_{D,\varepsilon}-\zeta I)^{-1}\,d\zeta ,\quad t>0.
\end{equation}
Here $\gamma\subset\mathbb{C}$ is a contour enclosing the spectrum of the operator $\widetilde{B}_{D,\varepsilon}$ in positive direction.
The exponential of the operator $\widetilde{B}_D^0$ satisfies a similar representation. Since the constant \eqref{c_flat} is a common lower bound of the operators $\widetilde{B}_{D,\varepsilon}$ and $\widetilde{B}_D^0$, it is convenient to choose the contour of integration as follows:
\begin{equation*}
\begin{split}
\gamma =\lbrace \zeta \in\mathbb{C} : \textnormal{Im}\,\zeta \geqslant 0,\, \textnormal{Re}\,\zeta =\textnormal{Im}\,\zeta +c_\flat/2\rbrace 
\cup \lbrace \zeta \in \mathbb{C} : \textnormal{Im}\,\zeta \leqslant 0,\, \textnormal{Re}\,\zeta =-\textnormal{Im}\,\zeta +c_\flat /2\rbrace .
\end{split}
\end{equation*}

Multiplying \eqref{2.1a} by $f^\varepsilon$ from the left and by $(f^\varepsilon)^*$ from the right and using identity~\eqref{B deps and tilde B D,eps tozd resolvent},
we obtain
\begin{equation*}
f^\varepsilon e^{-\widetilde{B}_{D,\varepsilon}t}(f^\varepsilon)^*=-\frac{1}{2\pi i}\int\limits_\gamma e^{-\zeta t}(B_{D,\varepsilon}-\zeta Q_0^\varepsilon)^{-1}\,d\zeta ,\quad t>0.
\end{equation*}
Similarly,
\begin{equation*}
f_0 e^{-\widetilde{B}_D^0t}f_0=-\frac{1}{2\pi i}\int\limits_\gamma e^{-\zeta t}(B_D^0-\zeta \overline{Q_0})^{-1}\,d\zeta ,\quad t>0.
\end{equation*}
Hence,
\begin{equation}
\label{raznost exp=int}
\begin{split}
f^\varepsilon e^{-\widetilde{B}_{D,\varepsilon}t}(f^\varepsilon)^*-f_0 e^{-\widetilde{B}_D^0t}f_0
=-\frac{1}{2\pi i}\int\limits_\gamma e^{-\zeta t}\left((B_{D,\varepsilon}-\zeta Q_0^\varepsilon)^{-1}-(B_D^0-\zeta \overline{Q_0})^{-1}\right)\,d\zeta .
\end{split}
\end{equation}

By Theorems~\ref{Theorem Dirichlet L2} and~\ref{Theorem another approximation}, we estimate the difference of the generalized resolvents for $\zeta \in \gamma$
uniformly in $ \textnormal{arg}\,\zeta$. Recall the notation $\psi =\textnormal{arg}\,(\zeta -c_\flat)$.
Note that for $\zeta \in \gamma$ and $\psi =\pi/2$ or $\psi =3\pi/2$ we have $\vert \zeta \vert =\sqrt{5} c_\flat/2$.
We apply Theorem~\ref{Theorem another approximation} for $\zeta \in \gamma$ such that $\vert \zeta \vert \leqslant \check{c}$, where
\begin{equation}
\label{check c}
\check{c}:= \max \lbrace 1; \sqrt{5}c_\flat /2\rbrace .
\end{equation}
Obviously, $\psi \in (\pi/4,7\pi/4)$ on the contour $\gamma$ and
\begin{equation}
\label{rho <= on gamma}
\rho _\flat(\zeta)\leqslant 2 \max \lbrace 1; 8c_\flat^{-2}\rbrace =:\mathfrak{C} ,\quad\zeta\in\gamma .
\end{equation}
Therefore, \eqref{Th dr appr 1} implies that
\begin{equation}
\label{ots_zeta_levee_c2}
\begin{split}
\Vert (B_{D,\varepsilon}-\zeta Q_0^\varepsilon)^{-1}-(B_D^0-\zeta \overline{Q_0})^{-1}\Vert _{L_2 (\mathcal{O})\rightarrow L_2(\mathcal{O})}\leqslant C_4\mathfrak{C} \varepsilon
\leqslant C_{15}' \vert \zeta \vert ^{-1/2}\varepsilon,
\\
 \zeta \in \gamma ,\quad \vert \zeta \vert \leqslant \check{c},\quad 0<\varepsilon \leqslant \varepsilon _1 ;\quad C_{15}' :=  C_4\mathfrak{C} \check{c}^{1/2}.
\end{split}
\end{equation}

For other~$\zeta \in \gamma $, we have
\begin{equation}
\label{sin >= on gamma}
\vert \sin \phi \vert \geqslant 5^{-1/2},\quad\zeta\in\gamma,\quad\vert\zeta\vert>\check{c},
\end{equation}
and, by Theorem~\ref{Theorem Dirichlet L2},
\begin{equation}
\label{ots_zeta_pravee_c2}
\begin{split}
\Vert (B_{D,\varepsilon}-\zeta Q_0^\varepsilon)^{-1}-(B_D^0-\zeta \overline{Q_0})^{-1}\Vert _{L_2(\mathcal{O})\rightarrow L_2(\mathcal{O})}\leqslant
C_{15}'' \vert \zeta \vert ^{-1/2}\varepsilon ,\\
\zeta \in \gamma ,\quad \vert \zeta \vert > \check{c} ,\quad 0<\varepsilon \leqslant \varepsilon _1,
\end{split}
\end{equation}
where~$C_{15}'' :=  5^{5/2}C_1$.
As a result, combining (\ref{ots_zeta_levee_c2}) and (\ref{ots_zeta_pravee_c2}), for $0<\varepsilon \leqslant \varepsilon _1$ we have
\begin{equation}
\label{Raznost res na gamma}
\Vert (B_{D,\varepsilon}-\zeta Q_0^\varepsilon)^{-1}-(B_D^0-\zeta \overline{Q_0})^{-1}\Vert _{L_2(\mathcal{O})\rightarrow L_2(\mathcal{O})}\leqslant\widehat{C}_{15}\vert \zeta \vert ^{-1/2}\varepsilon ,\quad \zeta \in \gamma ,
\end{equation}
where $\widehat{C}_{15}:=\max\lbrace C_{15}'; C_{15}'' \rbrace$.

From (\ref{raznost exp=int}) and (\ref{Raznost res na gamma}) it follows that
\begin{equation*}
\Vert f^\varepsilon e^{-\widetilde{B}_{D,\varepsilon}t}(f^\varepsilon)^*-f_0 e^{-\widetilde{B}_D^0t}f_0\Vert _{L_2(\mathcal{O})\rightarrow L_2(\mathcal{O})}\leqslant 2\pi ^{-1}\widehat{C}_{15} \varepsilon t^{-1/2}\Gamma (1/2)  e^{-c_\flat t/2}.
\end{equation*}
Taking into account that $\Gamma (1/2)=\pi ^{1/2}$, we find
\begin{equation}
\label{Ots.3 L_2}
\begin{split}
\Vert &f^\varepsilon e^{-\widetilde{B}_{D,\varepsilon}t}(f^\varepsilon)^*-f_0 e^{-\widetilde{B}_D^0t}f_0 \Vert _{L_2(\mathcal{O})\rightarrow L_2(\mathcal{O})}\leqslant
2\pi ^{-1/2}\widehat{C}_{15}\varepsilon t^{-1/2}e^{-c_\flat t/2}
\\
&\leqslant \check{C}_{15}\varepsilon (t+\varepsilon ^2)^{-1/2} e^{-c_\flat t/2}
,\quad t\geqslant \varepsilon ^2,
\end{split}
\end{equation}
where $\check{C}_{15}:=2\sqrt{2}\pi ^{-1/2}\widehat{C}_{15}$.
For $t\leqslant \varepsilon ^2$ we use a rough estimate
\begin{equation}
\label{Ots.4 L_2}
\begin{split}
\Vert & f^\varepsilon e^{-\widetilde{B}_{D,\varepsilon}t}(f^\varepsilon)^*-f_0 e^{-\widetilde{B}_D^0t}f_0\Vert _{L_2(\mathcal{O})\rightarrow L_2(\mathcal{O})}\leqslant
2\Vert f\Vert _{L_\infty}^2e^{-c_\flat t}
\\
&\leqslant 2\sqrt{2} \Vert f\Vert _{L_\infty}^2\varepsilon (t+\varepsilon^2)^{-1/2} e^{-c_\flat t/2},\quad t\leqslant \varepsilon ^2.
\end{split}
\end{equation}
Relations (\ref{Ots.3 L_2}) and (\ref{Ots.4 L_2}) imply the required inequality
(\ref{Th_exp_L_2}) with the constant $C_{15} := \max \lbrace \check{C}_{15}; 2\sqrt{2}\Vert f\Vert _{L_\infty}^2 \rbrace $.
\end{proof}

\subsection{Approximation of the solution in $ H^1(\mathcal{O};\mathbb{C}^n)$}
We introduce a \textit{corrector}
\begin{equation}
\label{K_D(t,e)}
\mathcal{K}_D(t;\varepsilon):=R_\mathcal{O}\left([\Lambda ^\varepsilon]S_\varepsilon b(\mathbf{D})+[\widetilde{\Lambda}^\varepsilon ]S_\varepsilon\right) P_\mathcal{O}f_0 e^{-\widetilde{B}_D^0t}f_0.
\end{equation}
For $t>0$ the operator \eqref{K_D(t,e)} is a continuous mapping of $L_2(\mathcal{O};\mathbb{C}^n)$ to $H^1(\mathcal{O};\mathbb{C}^n)$.
Indeed, according to \eqref{exp tilde B_D^0 L2-H2}, for $t>0$ the operator $f_0e^{-\widetilde{B}_D^0t}f_0$ is continuous from $L_2(\mathcal{O};\mathbb{C}^n)$ to $H^2(\mathcal{O};\mathbb{C}^n)$. Hence, the operator $b(\mathbf{D})P_\mathcal{O}f_0e^{-\widetilde{B}_D^0 t}f_0$ is continuous from $L_2(\mathcal{O};\mathbb{C}^n)$ to  $H^1(\mathbb{R}^d;\mathbb{C}^m)$. Obviously, the operator $P_\mathcal{O}f_0 e^{-\widetilde{B}_D^0 t}f_0$ is also continuous from $L_2(\mathcal{O};\mathbb{C}^n)$ to $H^1(\mathbb{R}^d;\mathbb{C}^n)$.
It remains to use the continuity of the operators $[\Lambda^\varepsilon ]S_\varepsilon : H^1(\mathbb{R}^d;\mathbb{C}^m)\rightarrow H^1(\mathbb{R}^d;\mathbb{C}^n)$ and $[\widetilde{\Lambda}^\varepsilon ]S_\varepsilon : H^1(\mathbb{R}^d;\mathbb{C}^n) \rightarrow H^1(\mathbb{R}^d;\mathbb{C}^n)$ which follows from Proposition~\ref{Proposition f^eps S_eps}
and relations $\Lambda,\widetilde{\Lambda}\in \widetilde{H}^1(\Omega)$.

We put $\widetilde{\mathbf{u}}_0(\,\cdot\, ,t):=P_\mathcal{O}\mathbf{u}_0(\,\cdot\, ,t)$. By $\mathbf{v}_\varepsilon$ we denote the first order approximation of the solution $\mathbf{u}_\varepsilon$ of problem~\eqref{first initial-boundary value problem}:
\begin{equation}
\label{v_eps}
\begin{split}
\widetilde{\mathbf{v}}_\varepsilon (\,\cdot\, ,t) =\widetilde{\mathbf{u}}_0(\,\cdot\, ,t)+\varepsilon \Lambda ^\varepsilon S_\varepsilon b(\mathbf{D})\widetilde{\mathbf{u}}_0(\,\cdot\, ,t)+\varepsilon\widetilde{\Lambda}^\varepsilon S_\varepsilon \widetilde{\mathbf{u}}_0(\,\cdot\, ,t),
\\
 \mathbf{v}_\varepsilon(\,\cdot\, ,t) :=\widetilde{\mathbf{v}}_\varepsilon (\,\cdot\, ,t) \vert _{\mathcal{O}}.
\end{split}
\end{equation}
So, $\mathbf{v}_\varepsilon (\,\cdot\, ,t) = f_0 e^{-\widetilde{B}_D^0t}f_0\boldsymbol{\varphi}(\,\cdot\,)+\varepsilon \mathcal{K}_D(t ;\varepsilon )\boldsymbol{\varphi}(\,\cdot\,)$.

\begin{theorem}
\label{Theorem H1 exp}
Suppose that the assumptions of Theorem~\textnormal{\ref{Theorem L2 exp}} are satisfied. Suppose that the matrix-valued functions $\Lambda (\mathbf{x})$ and $\widetilde{\Lambda}(\mathbf{x})$ are $\Gamma$-periodic solutions of the problems~\eqref{Lambda problem} and~\eqref{tildeLambda_problem}, respectively. Let $S_\varepsilon $ be the Steklov smoothing operator~\eqref{S_eps}, and let  $P_\mathcal{O}$ be the extension operator~\eqref{P_O H^1, H^2}. We put $\widetilde{\mathbf{u}}_0(\,\cdot\, ,t)=P_\mathcal{O}\mathbf{u}_0(\,\cdot\, ,t)$.
Suppose that $\mathbf{v}_\varepsilon$ is defined by~\eqref{v_eps}.
 Then for \hbox{$0<\varepsilon\leqslant\varepsilon _1$} and $t>0$ we have
 \begin{equation*}
 \Vert \mathbf{u}_\varepsilon (\,\cdot\, ,t)-\mathbf{v}_\varepsilon (\,\cdot\, ,t)\Vert _{H^1(\mathcal{O})}
 \leqslant C_{16}( \varepsilon ^{1/2}t^{-3/4}+\varepsilon t^{-1})e^{-c_\flat t/2}\Vert \boldsymbol{\varphi}\Vert _{L_2(\mathcal{O})}.
 \end{equation*}
 In the operator terms,
\begin{align}
\label{Th_exp_korrector}
\begin{split}
\Vert   f^\varepsilon e^{-\widetilde{B}_{D,\varepsilon}t}(f^\varepsilon)^*-f_0 e^{-\widetilde{B}_D^0t}f_0-\varepsilon \mathcal{K}_D(t;\varepsilon)\Vert _{L_2(\mathcal{O})\rightarrow H^1(\mathcal{O})}
\leqslant C_{16}( \varepsilon ^{1/2}t^{-3/4}+\varepsilon t^{-1})e^{-c_\flat t/2},
\end{split}
\end{align}
where $\mathcal{K}_D(t;\varepsilon)$ is the corrector~\textnormal{(\ref{K_D(t,e)})}.
Suppose that the matrix-valued function $\widetilde{g}(\mathbf{x})$ is defined by~\eqref{tilde g}.
For $0<\varepsilon\leqslant\varepsilon_1$ and $t>0$
the flux $\mathbf{p}_\varepsilon :=g^\varepsilon b(\mathbf{D})\mathbf{u}_\varepsilon $ satisfies
\begin{equation*}
\begin{split}
\bigl\Vert  \mathbf{p}_\varepsilon (\,\cdot\, ,t) -\widetilde{g}^\varepsilon S_\varepsilon b(\mathbf{D})\widetilde{\mathbf{u}}_0 (\,\cdot\, ,t)-g^\varepsilon \bigl(b(\mathbf{D})\widetilde{\Lambda}\bigr)^\varepsilon S_\varepsilon\widetilde{\mathbf{u}}_0(\,\cdot\, ,t)\bigr\Vert _{L_2(\mathcal{O})}
\leqslant \widetilde{C}_{16}\varepsilon ^{1/2}t^{-3/4}e^{-c_\flat t/2}\Vert \boldsymbol{\varphi}\Vert _{L_2(\mathcal{O})}.
\end{split}
\end{equation*}
In the operator terms,
\begin{equation}
\label{exp apppr fluxes}
\begin{split}
\Vert g^\varepsilon b(\mathbf{D})f^\varepsilon e^{-\widetilde{B}_{D,\varepsilon}t}(f^\varepsilon)^*-\mathcal{G}_D(t;\varepsilon)\Vert _{L_2(\mathcal{O})\rightarrow L_2(\mathcal{O})}
\leqslant \widetilde{C}_{16}\varepsilon ^{1/2}t^{-3/4}e^{-c_\flat t/2}.
\end{split}
\end{equation}
Here
\begin{equation*}
\mathcal{G}_D(t;\varepsilon):=\widetilde{g}^\varepsilon S_\varepsilon b(\mathbf{D})P_\mathcal{O}f_0e^{-\widetilde{B}_D^0 t}f_0
+g^\varepsilon \bigl(b(\mathbf{D})\widetilde{\Lambda}\bigr)^\varepsilon S_\varepsilon P_\mathcal{O}f_0 e^{-\widetilde{B}_D^0t}f_0.
\end{equation*}
The constants $C_{16}$ and $\widetilde{C}_{16}$ depend only on the problem data \eqref{problem data}.
\end{theorem}

\begin{proof}
As in the proof of Theorem~\ref{Theorem L2 exp}, we use representations for the sandwiched operator exponentials in terms of the contour integrals of the corresponding generalized resolvents. We have
\begin{equation}
\label{exp -exp-K=int}
\begin{split}
&f^\varepsilon e^{-\widetilde{B}_{D,\varepsilon}t}(f^\varepsilon)^*-f_0 e^{-\widetilde{B}_D^0t}f_0-\varepsilon \mathcal{K}_D(t;\varepsilon)\\
&=-\frac{1}{2 \pi i}\int\limits_\gamma e^{-\zeta t}\left( (B_{D,\varepsilon}-\zeta Q_0^\varepsilon)^{-1}-(B_D^0-\zeta \overline{Q_0})^{-1}-\varepsilon K_D(\varepsilon;\zeta )\right) \,d\zeta .
\end{split}
\end{equation}
Here $K_D(\varepsilon;\zeta )$ is the operator~\eqref{K_D(eps,zeta)}.

Similarly to \eqref{ots_zeta_levee_c2}--\eqref{Raznost res na gamma}, by Theorems~\ref{Theorem Dirichlet H1} and~\ref{Theorem another approximation},
\begin{equation}
\label{Th Su 14 zeta in gamma}
\begin{split}
 \Vert &(B_{D,\varepsilon}-\zeta Q_0^\varepsilon)^{-1}-(B_D^0-\zeta \overline{Q_0})^{-1}-\varepsilon K_D(\varepsilon ;\zeta )\Vert _{L_2(\mathcal{O})\rightarrow H^1(\mathcal{O})}\\
 &\leqslant
 \widehat{C}_{16}\left(\varepsilon ^{1/2}\vert \zeta \vert ^{-1/4}+\varepsilon\right) ,\quad
 \zeta \in \gamma ,\quad 0<\varepsilon \leqslant\varepsilon _1,
 \end{split}
\end{equation}
with the constant $\widehat{C}_{16}:=\max \lbrace C_{16}'; C_{16}'' \rbrace$, where
$C_{16}':=(1+\check{c})^{1/2}C_5\mathfrak{C}$ and $C_{16}'':= \max \lbrace 5C_2;25C_3\rbrace$.
Relations (\ref{exp -exp-K=int}) and (\ref{Th Su 14 zeta in gamma}) imply the required estimate~(\ref{Th_exp_korrector}) with the constant \break$C_{16}:=2\pi ^{-1}\Gamma (3/4) \widehat{C}_{16}$.

Similarly, the identity
\begin{equation}
\label{fluxes = int}
\begin{split}
g^\varepsilon b(\mathbf{D})f^\varepsilon e^{-\widetilde{B}_{D,\varepsilon}t}(f^\varepsilon)^*-\mathcal{G}_D(t;\varepsilon)
=
-\frac{1}{2\pi i}\int\limits_\gamma e^{-\zeta t}\bigl(g^\varepsilon b(\mathbf{D})(B_{D,\varepsilon}-\zeta Q_0^\varepsilon)^{-1}
-G_D(\varepsilon;\zeta)\bigr)\,d\zeta
\end{split}
\end{equation}
and estimates~\eqref{Th L_2->H^1 fluxes ell case}, \eqref{Th dr appr 3} yield the inequality \eqref{exp apppr fluxes} with the constant
\begin{equation*}
\widetilde{C}_{16}:=2\pi ^{-1}\Gamma (3/4)\max\left\lbrace
5^{5/4}\widetilde{C}_2;2\check{c}^{1/4}(1+\check{c})^{1/2}\widetilde{C}_5\mathfrak{C}
\right\rbrace.\qedhere
\end{equation*}
\end{proof}

By Remark~\ref{Remark elliptic another appr}(2), we observe the following.

\begin{remark}
Let $\lambda _1^0$ be the first eigenvalue of the operator~$B_D^0$, and let \hbox{$\kappa >0$} be an arbitrarily small number.
Due to the norm-resolvent convergence, for sufficiently small $\varepsilon_\circ$ the number~$\lambda _1^0\Vert Q_0\Vert ^{-1}_{L_\infty}-\kappa /2$ is a common lower bound of the operators $\widetilde{B}_{D,\varepsilon}$ for all $0<\varepsilon\leqslant\varepsilon _\circ$. Therefore, we can shift the integration contour so that it will
 intersect the real axis at the point $\mathfrak{c}:=\lambda _1^0\Vert Q_0\Vert ^{-1}_{L_\infty}-\kappa $ instead of $c_\flat /2$.
 By this way, we obtain estimates \eqref{Th_exp_L_2}, \eqref{Th_exp_korrector}, and \eqref{exp apppr fluxes} with $e^{-c_\flat t/2}$ replaced by $e^{-\mathfrak{c}t}$ in the right-hand sides. The constants in estimates become dependent on $\kappa$.
\end{remark}

\subsection{Estimates for small time} Note that for $0<t<\varepsilon ^2$ it makes no sense to apply estimates \eqref{Th_exp_korrector} and \eqref{exp apppr fluxes},
since it is better to use the following simple statement (which is valid, however, for all $t>0$).

\begin{proposition}
\label{Proposition small time}
Suppose that the assumptions of Theorem \textnormal{\ref{Theorem L2 exp}} are satisfied. Then for $t>0$ and $0<\varepsilon\leqslant 1$ we have
\begin{align}
\label{small t L2-H1}
\Vert f^\varepsilon e^{-\widetilde{B}_{D,\varepsilon}t}(f^\varepsilon)^*-f_0e^{-\widetilde{B}_D^0t}f_0\Vert _{L_2(\mathcal{O})\rightarrow H^1(\mathcal{O})}
&\leqslant C_{17}t^{-1/2}e^{-c_\flat t/2},
\\
\label{small t flux g^eps}
\Vert g^\varepsilon b(\mathbf{D})f^\varepsilon e^{-\widetilde{B}_{D,\varepsilon}t}(f^\varepsilon)^*\Vert _{L_2(\mathcal{O})\rightarrow L_2(\mathcal{O})}
&\leqslant \widetilde{C}_{17}t^{-1/2}e^{-c_\flat t/2},
\\
\label{small t flux g^0}
\Vert g^0b(\mathbf{D})f_0 e^{-\widetilde{B}_{D}^0t}f_0\Vert _{L_2(\mathcal{O})\rightarrow L_2(\mathcal{O})}
&\leqslant \widetilde{C}_{17}t^{-1/2}e^{-c_\flat t/2},
\end{align}
where the constants $C_{17}:=2c_3\Vert f\Vert _{L_\infty}$ and $\widetilde{C}_{17}:=\Vert g\Vert _{L_\infty}^{1/2}\Vert f\Vert _{L_\infty}$ depend only on the problem data \eqref{problem data}.
\end{proposition}

\begin{proof}
Inequality \eqref{small t L2-H1} follows from \eqref{f_0<=}, \eqref{exp tilde B_D,eps L2-H1}, and \eqref{exp tilde B_D^0 L2-H1}.

Next, by \eqref{tilde b D,eps=},
\begin{equation*}
\Vert
g^\varepsilon b(\mathbf{D})f^\varepsilon e^{-\widetilde{B}_{D,\varepsilon}t}(f^\varepsilon)^*\Vert _{L_2(\mathcal{O})\rightarrow L_2(\mathcal{O})}
\!\leqslant\! \Vert g\Vert _{L_\infty}^{1/2}\Vert f\Vert _{L_\infty}\Vert \widetilde{B}_{D,\varepsilon}^{1/2}e^{-\widetilde{B}_{D,\varepsilon}t}\Vert _{L_2(\mathcal{O})\rightarrow L_2(\mathcal{O})}.
\end{equation*}
Together with \eqref{tilde B_D,eps exp}, this yields \eqref{small t flux g^eps}.
By \eqref{f_0<=} and \eqref{tilde b_D^0 =}, estimate~\eqref{small t flux g^0} is checked similarly.
\end{proof}

\subsection{Removal of the smoothing operator $S_\varepsilon$ in the corrector}
It is possible to remove the smoothing operator in the corrector if the matrix-valued functions $\Lambda (\mathbf{x})$ and $\widetilde{\Lambda}(\mathbf{x})$
satisfy Conditions~\ref{Condition Lambda in L infty} and~\ref{Condition tilde Lambda in Lp}, respectively.
The following result is checked similarly to Theorem~\ref{Theorem H1 exp} by using Theorems~\ref{Theorem no S-eps}~and~\ref{Theorem Dr appr no S_eps}.

\begin{theorem}
\label{Theorem exp no S-eps}
Suppose that the assumptions of Theorem~\textnormal{\ref{Theorem H1 exp}} are satisfied.
Suppose that the matrix-valued functions $\Lambda (\mathbf{x})$ and $\widetilde{\Lambda}(\mathbf{x})$ satisfy Conditions~\textnormal{\ref{Condition Lambda in L infty}} and~\textnormal{\ref{Condition tilde Lambda in Lp}}, respectively. We put
\begin{align}
\label{K_D^0(t;eps)}
&\mathcal{K}_D^0(t;\varepsilon) := (\Lambda ^\varepsilon  b(\mathbf{D})+ \widetilde{\Lambda}^\varepsilon )f_0 e^{-\widetilde{B}_D^0 t}f_0,
\\
\label{G_3(t;eps)}
&\mathcal{G}_D^0(t;\varepsilon):=\widetilde{g}^\varepsilon  b(\mathbf{D})f_0e^{-\widetilde{B}_D^0 t}f_0
+g^\varepsilon \bigl(b(\mathbf{D})\widetilde{\Lambda}\bigr)^\varepsilon f_0e^{-\widetilde{B}_D^0 t}f_0.
\end{align}
Then for $t>0$ and $0<\varepsilon\leqslant \varepsilon _1$ we have
\begin{align*}
\begin{split}
&\Vert f^\varepsilon e^{-\widetilde{B}_{D,\varepsilon}t}(f^\varepsilon )^*-f_0 e^{-\widetilde{B}_D^0 t}f_0-\varepsilon\mathcal{K}_D^0(t;\varepsilon)\Vert _{L_2(\mathcal{O})\rightarrow H^1(\mathcal{O})}
\leqslant C_{18}\left(\varepsilon ^{1/2}t^{-3/4}+\varepsilon t^{-1}\right)e^{-c_\flat t/2},
\end{split}
\\
\begin{split}
&\Vert g^\varepsilon b(\mathbf{D})f^\varepsilon e^{-\widetilde{B}_{D,\varepsilon}t}(f^\varepsilon)^*-\mathcal{G}_D^0(t;\varepsilon)\Vert _{L_2(\mathcal{O})\rightarrow L_2(\mathcal{O})}
\leqslant \widetilde{C}_{18}\left(\varepsilon ^{1/2}t^{-3/4}+\varepsilon t^{-1}\right)e^{-c_\flat t/2}.
\end{split}
\end{align*}
The constants $C_{18}$ and $\widetilde{C}_{18}$ depend on the problem data~\eqref{problem data}, $p$, and the norms $\Vert \Lambda\Vert _{L_\infty}$ and $\Vert \widetilde{\Lambda}\Vert _{L_p(\Omega)}$.
\end{theorem}

By Remark \ref{Remark elliptic no S-eps}, we observe the following.

\begin{remark}
If only Condition \textnormal{\ref{Condition Lambda in L infty}} \textnormal{(}Condition~\textnormal{\ref{Condition tilde Lambda in Lp}}, respectively\textnormal{)}
is satisfied, then the smoothing operator $S_\varepsilon$ can be removed in the term of the corrector containing $\Lambda^\varepsilon$ \textnormal{(}$\widetilde{\Lambda}^\varepsilon$, respectively\textnormal{)}.
\end{remark}

\subsection{The case of smooth boundary} It is also possible to remove the smoothing operator $S_\varepsilon$ in the corrector by increasing smoothness of the boundary.
In this subsection, we consider the case where $d\geqslant 3$, because for $d\leqslant 2$ we can apply Theorem~\ref{Theorem exp no S-eps}
(see Propositions~\ref{Proposition Lambda in L infty <=} and~\ref{Proposition tilde Lambda in Lp if}).

\begin{lemma}
\label{Lemma exp tilde B_D^0 L2-Hq}
Let $k \geqslant 2$ be an integer. Let $\mathcal{O}\subset\mathbb{R}^d$ be a bounded domain with the boundary $\partial\mathcal{O}$ of class $C^{k-1,1}$.
Then for $t>0$ the operator $e^{-\widetilde{B}_D^0t}$ is a continuous mapping of $L_2(\mathcal{O};\mathbb{C}^n)$ to $H^q(\mathcal{O};\mathbb{C}^n)$, $0\leqslant q\leqslant k$,
and
\begin{equation}
\label{exp tildeB_D0 L2-Hs}
\Vert e^{-\widetilde{B}_D^0 t}\Vert _{L_2(\mathcal{O})\rightarrow H^q(\mathcal{O})}
\leqslant \widehat{\mathrm{C}}_q t^{-q/2}e^{-c_\flat t/2},\quad t>0.
\end{equation}
The constant $\widehat{\mathrm{C}}_q$ depends only on $q$ and the problem data \eqref{problem data}.
\end{lemma}

\begin{proof}
It suffices to check estimate \eqref{exp tildeB_D0 L2-Hs} for integer $q\in [0,k]$;
then the result for non-integer $q$ follows by interpolation. For $q=0,1,2$ estimate~\eqref{exp tildeB_D0 L2-Hs} has been already proved
(see Lemma~\ref{Lemma properties of operator exponential}).

So, let $q$ be an integer such that $2\leqslant q\leqslant k$. By theorems about regularity of solutions of strongly elliptic systems (see, e.~g., \cite[Chapter~4]{McL}),
   the operator $(\widetilde{B}_D^0)^{-1}$ is continuous from $H^\sigma (\mathcal{O};\mathbb{C}^n)$ to $H^{\sigma+2}(\mathcal{O};\mathbb{C}^n)$ under the assumption $\partial\mathcal{O}\in C^{\sigma +1,1}$, where $\sigma \in \mathbb{Z}_+$. We also take into account that the operator $(\widetilde{B}_D^0)^{-1/2}$ is continuous from $L_2(\mathcal{O};\mathbb{C}^n)$ to $H^1(\mathcal{O};\mathbb{C}^n)$. It follows that, under the assumptions of lemma, for integer $q\in [2,k]$ the operator $(\widetilde{B}_D^0)^{-q/2}$ is a continuous mapping of  $L_2(\mathcal{O};\mathbb{C}^n)$ to $H^q(\mathcal{O};\mathbb{C}^n)$. We have
\begin{equation}
\label{tilde B_D0 ^-s/2}
\Vert (\widetilde{B}_D^0)^{-q/2}\Vert _{L_2(\mathcal{O})\rightarrow H^q(\mathcal{O})}
\leqslant \check{\mathrm{C}}_q,
\end{equation}
where the constant $\check{\mathrm{C}}_q$ depends on $q$ and the problem data \eqref{problem data}. From \eqref{tilde B_D0 ^-s/2} it follows that
\begin{equation*}
\begin{split}
\Vert e^{-\widetilde{B}_D^0t}\Vert _{L_2(\mathcal{O})\rightarrow H^q(\mathcal{O})}
&\leqslant\check{\mathrm{C}}_q\Vert (\widetilde{B}_D^0)^{q/2}e^{-\widetilde{B}_D^0t}\Vert _{L_2(\mathcal{O})\rightarrow L_2(\mathcal{O})}
\leqslant
\check{\mathrm{C}}_q\sup _{x\geqslant c_\flat}x^{q/2}e^{-xt}
\\
&\leqslant\check{\mathrm{C}}_qt^{-q/2}e^{-c_\flat t/2}\sup _{x\geqslant 0}x^{q/2}e^{-x/2}
\leqslant\widehat{\mathrm{C}}_q t^{-q/2}e^{-c_\flat t/2},
\end{split}
\end{equation*}
where $\widehat{\mathrm{C}}_q:=\check{\mathrm{C}}_q(q/e)^{q/2}$.
\end{proof}

Using  Lemma~\ref{Lemma exp tilde B_D^0 L2-Hq}, the properties of the matrix-valued functions $\Lambda(\mathbf{x})$ and $\widetilde{\Lambda}(\mathbf{x})$, and the properties of the operator $S_\varepsilon$, we can estimate the difference of the correctors \eqref{K_D(t,e)} and \eqref{K_D^0(t;eps)}.

\begin{lemma}
\label{Lemma K-K^0}
Let $d\geqslant 3$. Let $\mathcal{O}\subset\mathbb{R}^d$ be a bounded domain of class $C^{d/2,1}$ if $d$ is even and of class
$C^{(d+1)/2,1}$ if $d$ is odd. Let $\mathcal{K}_D(t;\varepsilon)$ be the operator \eqref{K_D(t,e)}, and let $\mathcal{K}_D^0(t;\varepsilon)$ be the operator~\eqref{K_D^0(t;eps)}.
Then for $0<\varepsilon\leqslant 1$ and $t>0$ we have
\begin{equation}
\label{lemma K-K^0}
\Vert \mathcal{K}_D(t;\varepsilon)-\mathcal{K}_D^0(t;\varepsilon)\Vert _{L_2(\mathcal{O})\rightarrow H^1(\mathcal{O})}
\leqslant \widehat{\mathcal{C}}_d (t^{-1}+t^{-d/4-1/2})e^{-c_\flat t/2}.
\end{equation}
The constant $\widehat{\mathcal{C}}_d$ depends only on the problem data~\eqref{problem data}.
\end{lemma}

Lemma~\ref{Lemma K-K^0} and Theorem~\ref{Theorem H1 exp} imply the following result.

\begin{theorem}
\label{Theorem smooth boundary}
Suppose that the assumptions of Theorem~\textnormal{\ref{Theorem L2 exp}} are satisfied, and $d\geqslant 3$.
Suppose that the domain $\mathcal{O}$ satisfies the assumptions of Lemma~\textnormal{\ref{Lemma K-K^0}}.
Let~$\mathcal{K}_D^0(t;\varepsilon)$ be the corrector~\eqref{K_D^0(t;eps)}. Let $\mathcal{G}_D^0(t;\varepsilon)$ be the operator~\eqref{G_3(t;eps)}.
Then for $t>0$ and $0<\varepsilon\leqslant\varepsilon _1$ we have
\begin{align}
\label{Th smooth domain 1}
\Vert & f^\varepsilon e^{-\widetilde{B}_{D,\varepsilon}t}(f^\varepsilon)^*-f_0e^{-\widetilde{B}_D^0t}f_0-\varepsilon\mathcal{K}_D^0(t;\varepsilon)\Vert _{L_2(\mathcal{O})\rightarrow H^1(\mathcal{O})}
\leqslant \mathcal{C}_d(\varepsilon^{1/2}t^{-3/4}+\varepsilon t^{-d/4-1/2})e^{-c_\flat t/2},
\\
\label{Th smooth domain 2}
\Vert & g^\varepsilon b(\mathbf{D})f^\varepsilon e^{-\widetilde{B}_{D,\varepsilon}t}(f^\varepsilon)^*-\mathcal{G}_D^0(t;\varepsilon)\Vert _{L_2(\mathcal{O})\rightarrow L_2(\mathcal{O})}
\leqslant \widetilde{\mathcal{C}}_d(\varepsilon^{1/2}t^{-3/4}
+\varepsilon t^{-d/4-1/2})e^{-c_\flat t/2}.
\end{align}
The constants $\mathcal{C}_d$ and $\widetilde{\mathcal{C}}_d$ depend only on the problem data~\eqref{problem data}.
\end{theorem}

The proofs of Lemma \ref{Lemma K-K^0} and Theorem~\ref{Theorem smooth boundary} are presented in Appendix (see \S\ref{Section removing steklov operator}) in order
not to clutter the main presentation. Clearly, it is convenient to apply Theorem~\ref{Theorem smooth boundary} if $t$ is separated from zero.
For small $t$ the order of the factor $(\varepsilon^{1/2}t^{-3/4}+\varepsilon t^{-d/4-1/2})$ grows with dimension. This is a ``charge''\, for the removal of the smoothing
 operator.

\begin{remark}
Instead of the smoothness assumption on $\partial\mathcal{O}$ from Lemma~\textnormal{\ref{Lemma K-K^0}}, we could impose the following implicit condition\textnormal{:} a bounded domain $\mathcal{O}$ with Lipschitz boundary is such that estimate~\eqref{exp tildeB_D0 L2-Hs} holds for~$q=d/2+1$. In such domain the statements of Lemma \textnormal{\ref{Lemma K-K^0}} and Theorem~\textnormal{\ref{Theorem smooth boundary}} remain valid.
\end{remark}

\subsection{The case of zero corrector}
Suppose that $g^0=\overline{g}$, i.~e., relations~\eqref{overline-g} are satisfied.
Suppose also that condition~\eqref{sum Dj aj =0} is satisfied.
Then the $\Gamma$-periodic solutions of problems \eqref{Lambda problem} and~\eqref{tildeLambda_problem} are equal to zero: $\Lambda (\mathbf{x})=0$ and $\widetilde{\Lambda}(\mathbf{x})=0$. Using Proposition~\ref{Proposition K=0}, we obtain the following result.

\begin{proposition} Suppose that relations \eqref{overline-g} and \eqref{sum Dj aj =0} are satisfied.
Then, under the assumptions of Theorem~\textnormal{\ref{Theorem L2 exp}}, for~$0<\varepsilon\leqslant 1$ we have
\begin{equation}
\label{Th exp K=0}
\Vert f^\varepsilon e^{-\widetilde{B}_{D,\varepsilon}t}(f^\varepsilon )^*
-f_0e^{-\widetilde{B}_D^0 t}f_0
\Vert _{L_2(\mathcal{O})\rightarrow H^1(\mathcal{O})}
\leqslant C_{19}\varepsilon t^{-1}e^{-c_\flat t/2},\quad t>0,
\end{equation}
where the constant $C_{19}$ depends only on the problem data \eqref{problem data}.
\end{proposition}

\begin{proof}
We rely on  identity~\eqref{raznost exp=int}. For $\vert\zeta\vert \leqslant \check{c}$, where $\check{c}$ is the constant \eqref{check c}, we use~\eqref{Pf K=0 dr. appr.} and~\eqref{rho <= on gamma}. For $\vert\zeta\vert>\check{c}$ we apply~\eqref{Pr K=0} and~\eqref{sin >= on gamma}.
As a result, we see that for $0<\varepsilon\leqslant 1$
\begin{align*}
&\Vert (B_{D,\varepsilon}-\zeta Q_0^\varepsilon )^{-1}-(B_D^0-\zeta\overline{Q_0})^{-1}\Vert _{L_2(\mathcal{O})\rightarrow H^1(\mathcal{O})}
\leqslant \widehat{C}_{19}\varepsilon,\quad\zeta\in\gamma;\\
&\widehat{C}_{19}:=\max\bigl\lbrace \bigl(C_9+C_{10}(1+\check{c})^{1/2}\bigr)\mathfrak{C} ;25C_8\bigr\rbrace.
\end{align*}
Together with \eqref{raznost exp=int}, this yields \eqref{Th exp K=0} with the constant $C_{19}:=2\pi ^{-1}\widehat{C}_{19}$.
\end{proof}

\subsection{Special case} Now, we assume that $g^0=\underline{g}$, i.~e., relations \eqref{underline-g} are satisfied.
Then, by Proposition~\ref{Proposition Lambda in L infty <=}($3^\circ$), Condition \ref{Condition Lambda in L infty} is satisfied.
By \cite[Remark~3.5]{BSu05}, the matrix-valued function~\eqref{tilde g} is constant and coincides with $g^0$, i.~e., $\widetilde{g}(\mathbf{x})=g^0=\underline{g}$. Thus, $\widetilde{g}^\varepsilon b(\mathbf{D})f_0e^{-\widetilde{B}_D^0t}f_0=g^0b(\mathbf{D})f_0e^{-\widetilde{B}_D^0t}f_0$.

Suppose in addition that relation \eqref{sum Dj aj =0} is satisfied.
Then $\widetilde{\Lambda}(\mathbf{x})=0$. The following result can be deduced from Theorem~\ref{Theorem H1 exp} and Proposition~\ref{Proposition S__eps - I}.

\begin{proposition}
Suppose that the relations \eqref{underline-g} and \eqref{sum Dj aj =0} are satisfied.
Then, under the assumptions of Theorem~\textnormal{\ref{Theorem L2 exp}}, for~$0<\varepsilon\leqslant\varepsilon _1$ and $t>0$ we have
\begin{equation}
\label{Prop. special case exp S 2}
\begin{split}
\Vert
g^\varepsilon b(\mathbf{D})f^\varepsilon e^{-\widetilde{B}_{D,\varepsilon}t}(f^\varepsilon)^*
-g^0b(\mathbf{D})f_0e^{-\widetilde{B}_D^0t}f_0\Vert _{L_2(\mathcal{O})\rightarrow L_2(\mathcal{O})}
 \leqslant \widetilde{C}_{16}'\varepsilon ^{1/2}t^{-3/4}e^{-c_\flat t/2}.
\end{split}
\end{equation}
The constant $\widetilde{C}_{16}'$ depends only on the problem data \eqref{problem data}.
\end{proposition}

\begin{proof}
From Theorem~\ref{Theorem H1 exp} it follows that
\begin{equation}
\label{d-vo pr. 2.11 1}
\begin{split}
\Vert 
g^\varepsilon b(\mathbf{D})f^\varepsilon e^{-\widetilde{B}_{D,\varepsilon}t}(f^\varepsilon)^*
-g^0S_\varepsilon b(\mathbf{D})P_\mathcal{O}
f_0e^{-\widetilde{B}_D^0t}f_0\Vert _{L_2(\mathcal{O})\rightarrow L_2(\mathcal{O})}
\leqslant \widetilde{C}_{16}\varepsilon ^{1/2}t^{-3/4}e^{-c_\flat t/2}.
\end{split}
\end{equation}

On the one hand, Proposition~\ref{Proposition S__eps - I} and relations~\eqref{<b^*b<}, \eqref{|g^0|<=}, \eqref{f_0<=},  \eqref{PO}, \eqref{exp tilde B_D^0 L2-H2} imply that
\begin{equation}
\label{2.43a}
\begin{split}
\Vert  g^0 (S_\varepsilon -I)b(\mathbf{D})P_\mathcal{O}f_0 e^{-\widetilde{B}_D^0t}f_0\Vert _{L_2(\mathcal{O})\rightarrow L_2(\mathbb{R}^d)}
&\leqslant
\varepsilon \Vert g\Vert _{L_\infty} r_1\alpha _1^{1/2}\Vert P_\mathcal{O} f_0e^{-\widetilde{B}_D^0t}f_0\Vert _{L_2(\mathcal{O})\rightarrow H^2(\mathbb{R}^d)}
\\
&\leqslant
\varepsilon\Vert g\Vert _{L_\infty}\Vert f\Vert _{L_\infty}r_1\alpha _1^{1/2}C_\mathcal{O}^{(2)}\widetilde{c}t^{-1}e^{-c_\flat t/2}.
\end{split}
\end{equation}
On the other hand, from \eqref{S_eps <= 1}, \eqref{<b^*b<}, \eqref{|g^0|<=}, \eqref{f_0<=}, \eqref{PO}, and \eqref{exp tilde B_D^0 L2-H1} it follows that
\begin{equation}
\label{2.43b}
\begin{split}
\Vert  g^0 (S_\varepsilon -I)b(\mathbf{D})P_\mathcal{O}f_0 e^{-\widetilde{B}_D^0t}f_0\Vert _{L_2(\mathcal{O})\rightarrow L_2(\mathbb{R}^d)}
&\leqslant
2\Vert g\Vert _{L_\infty}\alpha _1^{1/2}\Vert  P_\mathcal{O} f_0e^{-\widetilde{B}_D^0t}f_0\Vert _{L_2(\mathcal{O})\rightarrow H^1(\mathbb{R}^d)}
\\
&\leqslant
2\Vert g\Vert _{L_\infty}\Vert f\Vert _{L_\infty}\alpha _1^{1/2}C_\mathcal{O}^{(1)}c_3 t^{-1/2}e^{-c_\flat t/2}.
\end{split}
\end{equation}
By \eqref{2.43a} and \eqref{2.43b},
\begin{equation*}
\Vert  g^0 (S_\varepsilon -I)b(\mathbf{D})P_\mathcal{O}f_0 e^{-\widetilde{B}_D^0t}f_0\Vert _{L_2(\mathcal{O})\rightarrow L_2(\mathbb{R}^d)}
\leqslant
\check{C}_{16}\varepsilon ^{1/2} t^{-3/4}e^{-c_\flat t/2},
\end{equation*}
where $\check{C}_{16}:=\Vert g\Vert _{L_\infty}\Vert f\Vert _{L_\infty}\alpha_1^{1/2}\bigl(2 r_1 C_\mathcal{O}^{(1)}C_\mathcal{O}^{(2)}\widetilde{c}c_3\bigr)^{1/2}$.
Combining this with \eqref{d-vo pr. 2.11 1}, we obtain estimate \eqref{Prop. special case exp S 2} with the constant $\widetilde{C}_{16}':=\widetilde{C}_{16}
+ \check{C}_{16}$.
\end{proof}

\subsection{Estimates in a strictly interior subdomain}
\label{Subsection strictly interior}

Using Theorem~\ref{Theorem O'}, we improve error estimates in a strictly interior subdomain.

\begin{theorem}
\label{Theorem O' exp}
Suppose that the assumptions of Theorem~\textnormal{\ref{Theorem H1 exp}} are satisfied.
Let $\mathcal{O}'$ be a strictly interior subdomain of the domain $\mathcal{O}$, and let $\delta$ be defined by~\eqref{delta= 1.62a}.
Then for $0<\varepsilon\leqslant\varepsilon _1$ and $t>0$ we have
\begin{align}
\label{Th_exp_korrector strictly interior}
\begin{split}
\Vert &  f^\varepsilon e^{-\widetilde{B}_{D,\varepsilon}t}(f^\varepsilon)^*-f_0 e^{-\widetilde{B}_D^0t}f_0-\varepsilon \mathcal{K}_D(t;\varepsilon)\Vert _{L_2(\mathcal{O})\rightarrow H^1(\mathcal{O}')}\leqslant
\varepsilon(C_{20}t^{-1/2}\delta ^{-1}+C_{21}t^{-1})
e^{-c_\flat t/2},
\end{split}
\\
\begin{split}
\Vert &g^\varepsilon b(\mathbf{D})f^\varepsilon e^{-\widetilde{B}_{D,\varepsilon}t}(f^\varepsilon)^*-\mathcal{G}_D(t;\varepsilon)\Vert _{L_2(\mathcal{O})\rightarrow L_2(\mathcal{O}')}
\leqslant \varepsilon(\widetilde{C}_{20}t^{-1/2}\delta ^{-1}+\widetilde{C}_{21}t^{-1})
e^{-c_\flat t/2}.
\end{split}
\nonumber
\end{align}
The constants $C_{20},C_{21},\widetilde{C}_{20}$, and $\widetilde{C}_{21}$ depend only on the problem data~\eqref{problem data}.
\end{theorem}

\begin{proof}
The proof is based on application of Theorem \ref{Theorem O'} and relations~\eqref{exp -exp-K=int}, \eqref{fluxes = int}.
Also, estimates \eqref{rho <= on gamma} and \eqref{sin >= on gamma} are used. We omit the details.
\end{proof}

The following result is checked similarly with the help of Theorems~\ref{Theorem O' no S_eps} and~\ref{Theorem Dr appr strictly interior subdomain no S_eps}.

\begin{theorem}
\label{Theorem O' exp no S_eps}
Suppose that the assumptions of Theorem~\textnormal{\ref{Theorem O' exp}} are satisfied.
Suppose that the matrix-valued functions $\Lambda(\mathbf{x})$ and $\widetilde{\Lambda}(\mathbf{x})$ satisfy Conditions~\textnormal{\ref{Condition Lambda in L infty}}
and~\textnormal{\ref{Condition tilde Lambda in Lp}}, respectively.
Let $\mathcal{K}_D^0(t;\varepsilon)$ be the corrector~\eqref{K_D^0(t;eps)}, and let $\mathcal{G}_D^0(t;\varepsilon)$ be the operator~\eqref{G_3(t;eps)}.
Then for $t>0$ and $0<\varepsilon\leqslant \varepsilon _1$ we have
\begin{align*}
\begin{split}
&\Vert f^\varepsilon e^{-\widetilde{B}_{D,\varepsilon}t}(f^\varepsilon )^*-f_0 e^{-\widetilde{B}_D^0 t}f_0-\varepsilon\mathcal{K}_D^0(t;\varepsilon)\Vert _{L_2(\mathcal{O})\rightarrow H^1(\mathcal{O}')}
\leqslant \varepsilon(C_{20}t^{-1/2}\delta ^{-1}+C_{22}t^{-1})
e^{-c_\flat t/2},
\end{split}
\\
\begin{split}
&\Vert g^\varepsilon b(\mathbf{D})f^\varepsilon e^{-\widetilde{B}_{D,\varepsilon}t}(f^\varepsilon)^*-\mathcal{G}_D^0(t;\varepsilon)\Vert _{L_2(\mathcal{O})\rightarrow L_2(\mathcal{O}')}
\leqslant \varepsilon(\widetilde{C}_{20}t^{-1/2}\delta ^{-1}+\widetilde{C}_{22}t^{-1})
e^{-c_\flat t/2}.
\end{split}
\end{align*}
The constants $C_{20}$ and $\widetilde{C}_{20}$ are the same as in Theorem~\emph{\ref{Theorem O' exp}}.
The constants $C_{22}$ and~$\widetilde{C}_{22}$ depend on the problem data~\eqref{problem data}, $p$, and the norms $\Vert \Lambda\Vert _{L_\infty}$, $\Vert \widetilde{\Lambda}\Vert _{L_p(\Omega)}$.
\end{theorem}

Note that it is possible to remove the smoothing operator $S_\varepsilon$ in the corrector in estimates of Theorem~\ref{Theorem O' exp}
 without any additional assumptions on the matrix-valued functions $\Lambda(\mathbf{x})$ and $\widetilde{\Lambda}(\mathbf{x})$.
 For this, the additional smoothness of the boundary is not required.
We consider the case where $d\geqslant 3$ (otherwise, by Propositions \ref{Proposition Lambda in L infty <=} and~\ref{Proposition tilde Lambda in Lp if}, we can apply Theorem~\ref{Theorem O' exp no S_eps}). We know that for $t>0$ the operator $e^{-\widetilde{B}_D^0t}$ is continuous from $L_2(\mathcal{O};\mathbb{C}^n)$ to $H^2(\mathcal{O};\mathbb{C}^n)$ and estimate~\eqref{exp tilde B_D^0 L2-H2} holds. Moreover, the following property of ``regularity improvement''\, inside the domain is valid: for $t>0$ the operator $e^{-\widetilde{B}_D^0t}$ is continuous from  $L_2(\mathcal{O};\mathbb{C}^n)$ to $H^\sigma (\mathcal{O}';\mathbb{C}^n)$ for any integer $\sigma\geqslant 3$. We have
\begin{equation}
\label{exp tilde B_D^0 L2-H-sigma strictly interior}
\begin{split}
\Vert e^{-\widetilde{B}_D^0t}\Vert_{L_2(\mathcal{O})\rightarrow H^\sigma (\mathcal{O}')}
\leqslant\mathrm{C}'_\sigma t^{-1/2}(\delta ^{-2}+t^{-1})^{(\sigma -1)/2}e^{-c_\flat t/2},\\
t>0,\quad\sigma\in\mathbb{N},\quad\sigma\geqslant 3.
\end{split}
\end{equation}
The constant $\mathrm{C}'_\sigma$ depends on $\sigma$ and the problem data~\eqref{problem data}.
For the scalar parabolic equations, the property of ``regularity improvement''\, inside the domain was obtained in \cite[Chapter~3, \S~12]{LaSoU}.
In a similar way, it can be checked for the operator $\widetilde{B}_D^0$.
It is easy to deduce the qualified estimates \eqref{exp tilde B_D^0 L2-H-sigma strictly interior}, noticing that the derivatives $\mathbf{D}^\alpha \mathbf{u}_0$ (where $\mathbf{u}_0$
 is the function \eqref{u_0=} with $\boldsymbol{\varphi}\in L_2(\mathcal{O};\mathbb{C}^n)$) are solutions of a parabolic equation $\overline{Q_0}\partial _t \mathbf{D}^\alpha \mathbf{u}_0 =-B^0 \mathbf{D}^\alpha \mathbf{u}_0$. We multiply this equation by $\chi ^2\mathbf{D}^\alpha \mathbf{u}_0$ and integrate over the cylinder $\mathcal{O}\times (0,t)$.
Here $\chi$ is a smooth cut-off function equal to zero near the lateral surface and the bottom of the cylinder. The standard analysis of the corresponding integral identity together with the already known inequalities of Lemma \ref{Lemma properties of operator exponential} leads to estimates~\eqref{exp tilde B_D^0 L2-H-sigma strictly interior}.

Using the properties of the matrix-valued functions $\Lambda(\mathbf{x})$ and $\widetilde{\Lambda}(\mathbf{x})$, and also the properties of the operator~$S_\varepsilon$,
we can deduce the following statement from relation \eqref{exp tilde B_D^0 L2-H-sigma strictly interior}.

\begin{lemma}
\label{Lemma K-K0 strictly interior}
Suppose that the assumptions of Theorem~\textnormal{\ref{Theorem O' exp}} are satisfied and that $d\geqslant 3$. Let $\mathcal{K}_D^0(t;\varepsilon)$ be the operator~\eqref{K_D^0(t;eps)}.
Denote
\begin{equation}
\label{h_d(delta;t)}
h_d(\delta ;t):=t^{-1}
+t^{-1/2}(\delta ^{-2}+t^{-1})^{d/4}.
\end{equation}
Let $2r_1=\mathrm{diam}\,\Omega$. Then for $0<\varepsilon\leqslant (4r_1)^{-1}\delta$ and $t>0$ we have
\begin{equation}
\label{lm K-K0 strictly interior}
\begin{split}
\Vert \mathcal{K}_D(t;\varepsilon)-\mathcal{K}_D^0(t;\varepsilon)\Vert _{L_2(\mathcal{O})\rightarrow H^1(\mathcal{O}')}
\leqslant \mathrm{C}_d''h_d(\delta ;t)e^{-c_\flat t/2}.
\end{split}
\end{equation}
The constant $\mathrm{C}_d''$ depends only on the problem data \eqref{problem data}.
\end{lemma}

From Lemma~\ref{Lemma K-K0 strictly interior} and Theorem~\ref{Theorem O' exp} we deduce the following result.

\begin{theorem}
\label{Theorem O' no S_eps no conditions on Lambda's}
Suppose that the assumptions of Theorem~\textnormal{\ref{Theorem O' exp}} are satisfied, and $d\geqslant 3$.
Let $\mathcal{K}_D^0(t;\varepsilon)$ be the corrector \eqref{K_D^0(t;eps)}, and let $\mathcal{G}_D^0(t;\varepsilon)$ be the operator~\eqref{G_3(t;eps)}. Let $2r_1=\mathrm{diam}\,\Omega$. Then for $0<\varepsilon\leqslant\min \lbrace \varepsilon _1;(4r_1)^{-1}\delta\rbrace$ and $t>0$ we have
\begin{align}
\label{Th O' no S_eps no conditions on Lambda's}
\Vert & f^\varepsilon e^{-\widetilde{B}_{D,\varepsilon}t}(f^\varepsilon)^*
-f_0 e^{-\widetilde{B}_D^0 t}f_0
-\varepsilon\mathcal{K}_D^0(t;\varepsilon)
\Vert _{L_2(\mathcal{O})\rightarrow H^1(\mathcal{O}')}
\leqslant
\varepsilon\mathrm{C}_d h_d (\delta ;t)e^{-c_\flat t/2},
\\
\label{Th O' no S_eps no conditions on Lambda's fluxes}
\Vert & g^\varepsilon b(\mathbf{D})f^\varepsilon e^{-\widetilde{B}_{D,\varepsilon}t}(f^\varepsilon)^*
-\mathcal{G}_D^0(t;\varepsilon)\Vert _{L_2(\mathcal{O})\rightarrow L_2(\mathcal{O}')}
\leqslant
\varepsilon\widetilde{\mathrm{C}}_d h_d (\delta ;t)e^{-c_\flat t/2}.
\end{align}
Here $h_d(\delta ;t)$ is given by \eqref{h_d(delta;t)}, the constants $\mathrm{C}_d$ and  $\widetilde{\mathrm{C}}_d$ depend only on the problem data~\eqref{problem data}.
\end{theorem}

The proofs of Lemma~\ref{Lemma K-K0 strictly interior} and Theorem~\ref{Theorem O' no S_eps no conditions on Lambda's} are presented in Appendix (see~\S~\ref{Section removing S-eps in strictly interior subdomain}) in order not to clutter the main presentation.
Clearly, it is convenient to apply Theorem~\ref{Theorem O' no S_eps no conditions on Lambda's} if $t$ is separated from zero.
For small~$t$ the order of the factor $h_d(\delta;t)$ grows with dimension.
This is a ``charge'' for removal of the smoothing operator.

\section{Homogenization of the first initial boundary value problem \\ for nonhomogeneous equation}

\label{Section 3}

\subsection{The principal term of approximation}

In this section, we study the behavior of the solution of the first initial boundary value problem for a nonhomogeneous parabolic equation:
\begin{equation}
\label{first initial-boundary value problem with F}
\begin{cases}
Q_0^\varepsilon (\mathbf{x}) \frac{\partial \mathbf{u}_\varepsilon}{\partial t}(\mathbf{x},t)&=-B_{\varepsilon} \mathbf{u}_\varepsilon (\mathbf{x},t)+\mathbf{F}(\mathbf{x},t),
\quad\mathbf{x}\in\mathcal{O},\quad t>0;
\\
\mathbf{u}_\varepsilon (\,\cdot\, ,t) \vert _{\partial \mathcal{O}}&=0,\quad t>0;\\
Q_0^\varepsilon (\mathbf{x}) \mathbf{u}_\varepsilon (\mathbf{x},0)&=\boldsymbol{\varphi}(\mathbf{x}),\quad \mathbf{x}\in \mathcal{O}.
\end{cases}
\end{equation}
Here $\mathbf{F}\in \mathfrak{H}_r(T):= L_r((0,T);L_2(\mathcal{O};\mathbb{C}^n))$, $0<T\leqslant\infty$, with some $1\leqslant r\leqslant \infty$. Then
\begin{equation}
\label{u_eps = F neq 0}
\mathbf{u}_\varepsilon (\,\cdot\, ,t)=f^\varepsilon e^{-\widetilde{B}_{D,\varepsilon} t}(f^\varepsilon)^*\boldsymbol{\varphi}(\,\cdot\,)
+\int\limits_0^{t} f^\varepsilon e^{-\widetilde{B}_{D,\varepsilon} (t-\widetilde{t})}(f^\varepsilon)^*\mathbf{F}(\,\cdot\, ,\widetilde{t})\,d\widetilde{t}.
\end{equation}

The corresponding effective problem takes the form
\begin{equation}
\label{effective first problem with F}
\begin{cases}
\overline{Q_0}\frac{\partial \mathbf{u}_0}{\partial t}(\mathbf{x},t)&=-B^0\mathbf{u}_0 (\mathbf{x},t)+\mathbf{F}(\mathbf{x},t),
\quad\mathbf{x}\in\mathcal{O},\quad t>0;
\\
\mathbf{u}_0 (\,\cdot\, ,t) \vert _{\partial \mathcal{O}}&=0,\quad t>0;\\
\overline{Q_0}\mathbf{u}_0 (\mathbf{x},0)&=\boldsymbol{\varphi}(\mathbf{x}),\quad \mathbf{x}\in \mathcal{O}.
\end{cases}
\end{equation}
The solution of this problem is given by
\begin{equation}
\label{u_0 = F neq 0}
\mathbf{u}_0 (\,\cdot\, ,t)=f_0 e^{-\widetilde{B}_{D}^0 t}f_0\boldsymbol{\varphi}(\,\cdot\, )
+\int\limits_0^t f_0 e^{-\widetilde{B}_{D}^0 (t-\widetilde{t})}f_0\mathbf{F}(\,\cdot\, ,\widetilde{t})\,d\widetilde{t}.
\end{equation}
Subtracting \eqref{u_0 = F neq 0} from \eqref{u_eps = F neq 0} and using Theorem~\ref{Theorem L2 exp} (see \eqref{Th_exp_L_2}), we conclude that for $0<\varepsilon \leqslant \varepsilon _1$ and $t>0$
\begin{equation*}
\Vert \mathbf{u}_\varepsilon (\,\cdot\, ,t)-\mathbf{u}_0(\,\cdot\, ,t)\Vert _{L_2(\mathcal{O})}
\leqslant C_{15}\varepsilon (t+\varepsilon ^2)^{-1/2} e^{-c_\flat t/2}\Vert \boldsymbol{\varphi}\Vert _{L_2(\mathcal{O})}
+C_{15} \varepsilon\mathcal{L}(\varepsilon ;t;\mathbf{F}),
\end{equation*}
where
$$
\mathcal{L}(\varepsilon ;t;\mathbf{F}):=\int\limits_0^t e^{-c_\flat (t-\widetilde{t})/2}(\varepsilon ^2 +t-\widetilde{t})^{-1/2}\Vert \mathbf{F}(\,\cdot\, ,\widetilde{t})\Vert _{L_2(\mathcal{O})}\,d\widetilde{t}.
$$
Estimating the term $\mathcal{L}(\varepsilon ;t;\mathbf{F})$, for the case $1<r\leqslant \infty $ we obtain the following result.
Its proof is completely analogous to the proof of Theorem~5.1 from \cite{MSu2}.

\begin{theorem}
\label{Theorem 3.1}
Suppose that $\mathcal{O}\subset \mathbb{R}^d$ is a bounded domain of class~$C^{1,1}$. Suppose that the assumptions of Subsections~\textnormal{\ref{Subsection operatoer A_D,eps}--\ref{Subsection Effective operator}} are satisfied.
Let $\mathbf{u}_\varepsilon$ be the solution of problem~\eqref{first initial-boundary value problem with F}, and let $\mathbf{u}_0$ be the solution of the effective problem~\eqref{effective first problem with F} with $\boldsymbol{\varphi}\in L_2(\mathcal{O};\mathbb{C}^n)$ and $\mathbf{F}\in\mathfrak{H}_r(T)$, $0<T\leqslant\infty$, with some $1<r\leqslant \infty$.
Then for $0<\varepsilon\leqslant\varepsilon _1$ and $0<t<T$ we have
\begin{equation*}
\Vert \mathbf{u}_\varepsilon (\,\cdot\, ,t)-\mathbf{u}_0(\,\cdot\, ,t)\Vert _{L_2(\mathcal{O})}
\leqslant C_{15}\varepsilon (t+\varepsilon ^2)^{-1/2} e^{-c_\flat t/2}\Vert \boldsymbol{\varphi}\Vert _{L_2(\mathcal{O})}
+
c_r\theta (\varepsilon ,r)\Vert \mathbf{F}\Vert _{\mathfrak{H}_r(T)}.
\end{equation*}
Here $\theta (\varepsilon ,r)$ is given by
\begin{equation}
\label{theta}
\theta (\varepsilon ,r)=\begin{cases}
\varepsilon ^{2-2/r},&1<r<2,\\
\varepsilon (\vert \ln \varepsilon \vert +1)^{1/2},&r=2,\\
\varepsilon ,&2<r\leqslant \infty .
\end{cases}
\end{equation}
The constant $c_r$ depends only on $r$ and the problem data \eqref{problem data}.
\end{theorem}

By analogy with the proof of Theorem~5.2 from \cite{MSu2},
we can deduce approximation of the solution of problem \eqref{first initial-boundary value problem with F} in $\mathfrak{H}_r(T)$
from Theorem~\ref{Theorem L2 exp}.

\begin{theorem}
Suppose that the assumptions of Theorem~\textnormal{\ref{Theorem 3.1}} are satisfied.
Let $\mathbf{u}_\varepsilon$ and $\mathbf{u}_0$ be the solutions of problems \eqref{first initial-boundary value problem with F} and \eqref{effective first problem with F},
respectively, with $\boldsymbol{\varphi}\in L_2(\mathcal{O};\mathbb{C}^n)$ and $\mathbf{F}\in\mathfrak{H}_r(T)$, $0<T\leqslant\infty$, for some $1\leqslant r<\infty$.
Then for $0<\varepsilon\leqslant\varepsilon _1$ we have
\begin{equation*}
\Vert \mathbf{u}_\varepsilon -\mathbf{u}_0\Vert _{\mathfrak{H}_r(T)}
\leqslant c_{r'}\theta (\varepsilon ,r')\Vert \boldsymbol{\varphi}\Vert _{L_2(\mathcal{O})}
+C_{23}\varepsilon \Vert \mathbf{F}\Vert_{\mathfrak{H}_r(T)}.
\end{equation*}
Here $\theta (\varepsilon ,\,\cdot\, )$ is given by \eqref{theta}, $r^{-1}+(r')^{-1}=1$.
The constant $C_{23}$ depends only on the problem data \eqref{problem data}, the constant $c_{r'}$ depends on the same parameters and~$r$.
\end{theorem}

\begin{remark}
For the case where $\boldsymbol{\varphi}=0$ and $\mathbf{F}\in \mathfrak{H}_\infty (T)$, Theorem~\textnormal{\ref{Theorem 3.1}} implies that
\begin{equation*}
\Vert \mathbf{u}_\varepsilon -\mathbf{u}_0\Vert _{\mathfrak{H}_\infty (T)}
\leqslant c_\infty \varepsilon \Vert \mathbf{F}\Vert _{\mathfrak{H}_\infty (T)},\quad 0<\varepsilon\leqslant\varepsilon _1.
\end{equation*}
\end{remark}

\subsection{Approximation of the solution in $H^1(\mathcal{O};\mathbb{C}^n)$}

Now, we obtain approximation of the solution of problem~\eqref{first initial-boundary value problem with F} in the $H^1(\mathcal{O};\mathbb{C}^n)$-norm with the help of  Theorem~\ref{Theorem H1 exp}. The difficulties arise in consideration of the integral term in~\eqref{u_eps = F neq 0}, because estimate~\eqref{Th_exp_korrector} ``deteriorates''\,for small~$t$.
\textit{Assuming that} $t\geqslant \varepsilon ^2$, we divide the integration interval in \eqref{u_eps = F neq 0} into two parts: $(0,t-\varepsilon ^2 )$ and $(t-\varepsilon ^2,t)$.
On the interval $(0,t-\varepsilon ^2 )$ we apply \eqref{Th_exp_korrector}, and on $(t-\varepsilon ^2,t)$ we use \eqref{small t L2-H1}.

Denote
\begin{equation}
\label{w_eps :=}
\mathbf{w}_\varepsilon (\,\cdot\, ,t):=f_0 e^{-\widetilde{B}_D^0\varepsilon ^2}f_0^{-1}\mathbf{u}_0(\,\cdot\, ,t-\varepsilon ^2),
\end{equation}
where $\mathbf{u}_0$ is the solution of problem~\eqref{effective first problem with F}. By~\eqref{u_0 = F neq 0},
\begin{equation*}
\mathbf{w}_\varepsilon (\,\cdot\, ,t)=f_0 e^{-\widetilde{B}_D^0t}f_0\boldsymbol{\varphi }(\,\cdot\, )
+\int\limits_0^{t-\varepsilon ^2} f_0 e^{-\widetilde{B}_D^0 (t-\widetilde{t})}f_0\mathbf{F}(\,\cdot\, ,\widetilde{t})\,d\widetilde{t}.
\end{equation*}

The following statement can be checked similarly to Theorem~5.4 from~\cite{MSu2}.

\begin{theorem}
\label{Theorem H1 solutions with F}
Suppose that the assumptions of Theorem~\textnormal{\ref{Theorem 3.1}} are satisfied. Suppose that $\mathbf{u}_\varepsilon$ and $\mathbf{u}_0$ are the solutions of
 problems~\eqref{first initial-boundary value problem with F} and~\eqref{effective first problem with F}, respectively, with $\boldsymbol{\varphi}\in L_2(\mathcal{O};\mathbb{C}^n)$ and $\mathbf{F}\in\mathfrak{H}_r(T)$, $0<T\leqslant\infty$, for some $2< r\leqslant\infty$.  Let $\mathbf{w}_\varepsilon (\,\cdot\, ,t)$ be given by~\eqref{w_eps :=}.
 Let $\Lambda (\mathbf{x})$ and $\widetilde{\Lambda}(\mathbf{x})$ be the $\Gamma$-periodic matrix solutions of problems~\eqref{Lambda problem} and~\eqref{tildeLambda_problem},
 respectively. Suppose that $P_\mathcal{O}$ is a linear continuous extension operator~\eqref{P_O H^1, H^2}. Let $S_\varepsilon$ be the Steklov smoothing operator~\eqref{S_eps}.
 We put $\widetilde{\mathbf{w}}_\varepsilon (\,\cdot\, ,t):=P_\mathcal{O}\mathbf{w}_\varepsilon (\,\cdot\, ,t)$ and denote
\begin{equation*}
\mathbf{v}_\varepsilon (\,\cdot\, ,t):=\mathbf{u}_0(\,\cdot\, ,t)+\varepsilon \Lambda ^\varepsilon S_\varepsilon b(\mathbf{D})\widetilde{\mathbf{w}}_\varepsilon (\,\cdot\, ,t)+\varepsilon \widetilde{\Lambda}^\varepsilon S_\varepsilon \widetilde{\mathbf{w}}_\varepsilon (\,\cdot\, ,t).
\end{equation*}
Let $\mathbf{p}_\varepsilon (\,\cdot\, ,t):=g^\varepsilon b(\mathbf{D})\mathbf{u}_\varepsilon (\,\cdot\, ,t)$, and let $\widetilde{g}(\mathbf{x})$ be the matrix-valued function~\eqref{tilde g}. We put
\begin{equation*}
\mathbf{q}_\varepsilon (\,\cdot\, ,t):=\widetilde{g}^\varepsilon S_\varepsilon b(\mathbf{D})\widetilde{\mathbf{w}}_\varepsilon (\,\cdot\, ,t)+g^\varepsilon \bigl(b(\mathbf{D})\widetilde{\Lambda}\bigr)^\varepsilon S_\varepsilon \widetilde{\mathbf{w}}_\varepsilon (\,\cdot\, ,t).
\end{equation*}
Then for $0<\varepsilon\leqslant \varepsilon _1$ and $\varepsilon ^2\leqslant t <T$ we have
\begin{align*}
\Vert &\mathbf{u}_\varepsilon (\,\cdot\, ,t)-\mathbf{v}_\varepsilon (\,\cdot\, ,t)\Vert _{H^1(\mathcal{O})}
\leqslant 2C_{16}\varepsilon ^{1/2}t^{-3/4}e^{-c_\flat t/2}\Vert \boldsymbol{\varphi}\Vert _{L_2(\mathcal{O})}
+\check{c}_r \omega(\varepsilon ,r)\Vert \mathbf{F}\Vert _{\mathfrak{H}_r(T)},
\\
&\Vert \mathbf{p}_\varepsilon (\,\cdot\, ,t)-\mathbf{q}_\varepsilon (\,\cdot\, ,t)\Vert _{L_2(\mathcal{O})}
\leqslant
\widetilde{C}_{16}\varepsilon ^{1/2}t^{-3/4}e^{-c_\flat t/2}\Vert \boldsymbol{\varphi}\Vert _{L_2(\mathcal{O})}
+\widetilde{c}_r \omega(\varepsilon ,r)\Vert \mathbf{F}\Vert _{\mathfrak{H}_r(T)}.
\end{align*}
Here constants $\check{c}_r$ and $\widetilde{c}_r$ depend only on the problem data~\eqref{problem data} and~$r$, and
\begin{equation}
\label{omega (eps,r)}
\omega (\varepsilon ,r):=\begin{cases}
\varepsilon ^{1-2/r}, &2<r<4,\\
\varepsilon ^{1/2}(\vert \ln \varepsilon \vert +1)^{3/4}, & r=4,\\
\varepsilon ^{1/2}, &4<r\leqslant \infty .
\end{cases}
\end{equation}
\end{theorem}

Since the right-hand side of estimate \eqref{exp apppr fluxes} grows slowly than the right-hand side in estimate \eqref{Th_exp_korrector}, as
$t\rightarrow 0$, for $r>4$ we can approximate the flux $\mathbf{p}_\varepsilon$ in terms of
\begin{equation}
\label{h eps}
\mathbf{h}_\varepsilon (\,\cdot\, ,t):=\widetilde{g}^\varepsilon S_\varepsilon b(\mathbf{D})\widetilde{\mathbf{u}}_0(\,\cdot\, ,t)+g^\varepsilon \bigl(b(\mathbf{D})\widetilde{\Lambda}\bigr)^\varepsilon S_\varepsilon \widetilde{\mathbf{u}}_0(\,\cdot\,,t).
\end{equation}

\begin{proposition}
Suppose that the assumptions of Theorem~\textnormal{\ref{Theorem 3.1}} are satisfied.
Suppose that $\mathbf{u}_\varepsilon$ and $\mathbf{u}_0$ are the solutions of problems~\eqref{first initial-boundary value problem with F} and \eqref{effective first problem with F},
respectively, with $\boldsymbol{\varphi}\in L_2(\mathcal{O};\mathbb{C}^n)$ and $\mathbf{F}\in\mathfrak{H}_r(T)$, $0<T\leqslant\infty$, for some $4< r\leqslant\infty$.
Let $\mathbf{p}_\varepsilon (\,\cdot\, ,t)=g^\varepsilon b(\mathbf{D})\mathbf{u}_\varepsilon (\,\cdot\, ,t)$ and let $\mathbf{h}_\varepsilon (\,\cdot\, ,t)$ be given by~\eqref{h eps}. Then for $0<t<T$ and $0<\varepsilon \leqslant\varepsilon _1$ we have
\begin{equation}
\label{Pr flux with F}
\begin{split}
\Vert \mathbf{p}_\varepsilon (\,\cdot\, ,t)-\mathbf{h}_\varepsilon (\,\cdot\, ,t)\Vert _{L_2(\mathcal{O})}
\leqslant \widetilde{C}_{16}\varepsilon ^{1/2}t^{-3/4}e^{-c_\flat t/2}\Vert \boldsymbol{\varphi}\Vert _{L_2(\mathcal{O})}
+C_{24}^{(r)}\varepsilon ^{1/2}\Vert \mathbf{F}\Vert _{\mathfrak{H}_p(t)}.
\end{split}
\end{equation}
The constant $C_{24}^{(r)}$ depends only on the problem data \eqref{problem data} and $r$.
\end{proposition}

\begin{proof}
To check estimate \eqref{Pr flux with F}, we use inequality~\eqref{exp apppr fluxes} and identities~\eqref{u_eps = F neq 0}, \eqref{u_0 = F neq 0}.
If~$r=\infty$, we deduce~\eqref{Pr flux with F} with~$C_{24}^{(\infty)}:=(2/c_\flat )^{1/4}\Gamma (1/4)\widetilde{C}_{16}$. If~$4<r<\infty$,
we apply the H\"older inequality:
\begin{equation*}
\begin{split}
\Vert \mathbf{p}_\varepsilon (\,\cdot\, ,t)-\mathbf{h}_\varepsilon (\,\cdot\, ,t)\Vert _{L_2(\mathcal{O})}
&\leqslant \widetilde{C}_{16}\varepsilon ^{1/2}t^{-3/4}e^{-c_\flat t/2}\Vert \boldsymbol{\varphi}\Vert _{L_2(\mathcal{O})}
\\
&+
\widetilde{C}_{16}\varepsilon ^{1/2}\Vert \mathbf{F}\Vert _{\mathfrak{H}_r(t)}\mathfrak{I}_r(\varepsilon ,t)^{1/r'},\quad r^{-1}+(r')^{-1}=1.
\end{split}
\end{equation*}
Here
$$
\mathfrak{I}_r(\varepsilon ,t):=\int\limits_0^t \tau ^{-3r'/4}e^{-c_\flat r'\tau /2}\,d\tau
\leqslant (c_\flat r'/2)^{3r'/4-1}\Gamma (1-3r'/4).
$$
This implies \eqref{Pr flux with F} with the constant $C_{24}^{(r)}:= (c_\flat r'/2)^{3/4-1/r'}\Gamma (1-3r'/4)^{1/r'}\widetilde{C}_{16}$.
\end{proof}

Combining Proposition~\ref{Proposition small time} and Theorem~\ref{Theorem exp no S-eps}, we deduce the following result.

\begin{theorem}
\label{Theorem 3.6}
Suppose that the assumptions of Theorem~\textnormal{\ref{Theorem H1 solutions with F}} are satisfied.
Suppose that the matrix-valued functions $\Lambda (\mathbf{x})$ and $\widetilde{\Lambda}(\mathbf{x})$ satisfy Conditions~\textnormal{\ref{Condition Lambda in L infty}}
and~\textnormal{\ref{Condition tilde Lambda in Lp}}, respectively. Denote
\begin{align}
\label{check v_eps}
&\check{\mathbf{v}}_\varepsilon (\,\cdot\, ,t):=\mathbf{u}_0(\,\cdot\, ,t)+\varepsilon \Lambda ^\varepsilon b(\mathbf{D})\mathbf{w}_\varepsilon (\,\cdot\, ,t)
+\varepsilon \widetilde{\Lambda}^\varepsilon \mathbf{w}_\varepsilon (\,\cdot\, ,t),\\
\label{check q_eps}
&\check{\mathbf{q}}_\varepsilon (\,\cdot\, ,t):=\widetilde{g}^\varepsilon b(\mathbf{D})\mathbf{w}_\varepsilon (\,\cdot\, ,t)+g^\varepsilon \bigl(b(\mathbf{D})\widetilde{\Lambda}\bigr)^\varepsilon \mathbf{w}_\varepsilon (\,\cdot\, ,t).
\end{align}
Then for $0<\varepsilon \leqslant\varepsilon _1$ and $\varepsilon ^2\leqslant t <T$ we have
\begin{align*}
\begin{split}
&\Vert \mathbf{u}_\varepsilon (\,\cdot\, ,t)-\check{\mathbf{v}}_\varepsilon (\,\cdot\, ,t)\Vert _{H^1(\mathcal{O})}
\leqslant 2C_{18}\varepsilon ^{1/2}t^{-3/4}e^{-c_\flat t/2}\Vert \boldsymbol{\varphi}\Vert _{L_2(\mathcal{O})}
+c_r'\omega (\varepsilon ,r)\Vert \mathbf{F}\Vert _{\mathfrak{H}_r(t)},
\end{split}
\\
\begin{split}
&\Vert \mathbf{p}_\varepsilon (\,\cdot\, ,t)-\check{\mathbf{q}}_\varepsilon (\,\cdot\, ,t)\Vert _{L_2(\mathcal{O})}
\leqslant
2\widetilde{C}_{18}\varepsilon ^{1/2}t^{-3/4}e^{-c_\flat t/2}\Vert \boldsymbol{\varphi}\Vert _{L_2(\mathcal{O})}
+c_r''\omega (\varepsilon ,r)\Vert \mathbf{F}\Vert _{\mathfrak{H}_r(t)}.
\end{split}
\end{align*}
The constants $c_r'$ and $c_r''$ depend only on the initial data \eqref{problem data}, $r$, $p$, and the norms $\Vert \Lambda\Vert _{L_\infty}$, $\Vert \widetilde{\Lambda}\Vert _{L_p(\Omega)}$.
\end{theorem}

For the case of sufficiently smooth boundary, we could apply Theorem~\ref{Theorem smooth boundary}. However,
because of the strong growth of the right-hand side in estimates \eqref{Th smooth domain 1}, \eqref{Th smooth domain 2} for small $t$,
we obtain a meaningful result only in the three-dimensional case and only for~$r>4$.

\begin{proposition}
Suppose that the assumptions of Theorem~\textnormal{\ref{Theorem H1 solutions with F}} are satisfied with $d=3$ and $r>4$.
Suppose that $\partial\mathcal{O}\in C^{2,1}$. Let $\check{\mathbf{v}}_\varepsilon$ and $\check{\mathbf{q}}_\varepsilon $ be given by \eqref{check v_eps} and \eqref{check q_eps}.
Then for $0<\varepsilon\leqslant\varepsilon_1$ and $\varepsilon ^2\leqslant t <T$ we have
\begin{align*}
\begin{split}
\Vert &\mathbf{u}_\varepsilon (\,\cdot\, ,t)-\check{\mathbf{v}}_\varepsilon (\,\cdot\, ,t)\Vert _{H^1(\mathcal{O})}
\leqslant
\mathcal{C}_3(\varepsilon ^{1/2}t^{-3/4}+\varepsilon t^{-5/4})e^{-c_\flat t/2}\Vert \boldsymbol{\varphi}\Vert _{L_2(\mathcal{O})}
+\widetilde{c}'_r\varepsilon ^{1/2-2/r}\Vert \mathbf{F}\Vert _{\mathfrak{H}_r(t)},
\end{split}
\\
\begin{split}
\Vert &\mathbf{p}_\varepsilon (\,\cdot\, ,t)-\check{\mathbf{q}}_\varepsilon (\,\cdot\, ,t)\Vert _{L_2(\mathcal{O})}
\leqslant
\widetilde{\mathcal{C}}_3(\varepsilon ^{1/2}t^{-3/4}+\varepsilon t^{-5/4})e^{-c_\flat t/2}\Vert \boldsymbol{\varphi}\Vert _{L_2(\mathcal{O})}
+\widetilde{c}''_r\varepsilon ^{1/2-2/r}\Vert \mathbf{F}\Vert _{\mathfrak{H}_r (t)}.
\end{split}
\end{align*}
The constants $\widetilde{c}'_r$ and $\widetilde{c}''_r$ depend only on the problem data \eqref{problem data} and $r$.
\end{proposition}

\subsection{Approximation of the solution in a strictly interior subdomain}

From Theorem~\ref{Theorem O' exp} and Proposition \ref{Proposition small time} we deduce the following result.

\begin{theorem}
\label{Theorem O' solutions with F}
Suppose that the assumptions of Theorem~\textnormal{\ref{Theorem H1 solutions with F}} are satisfied. Let~$\mathcal{O}'$ be a strictly interior subdomain of the domain~$\mathcal{O}$.
Let~$\delta $ be given by~\eqref{delta= 1.62a}. Then for $0<\varepsilon\leqslant\varepsilon _1$ and $\varepsilon ^2\leqslant t <T$ we have
\begin{align*}
\begin{split}
\Vert &\mathbf{u}_\varepsilon (\,\cdot\, ,t)-\mathbf{v}_\varepsilon(\,\cdot\, ,t)\Vert _{H^1(\mathcal{O}')}
\leqslant \varepsilon (C_{20}t^{-1/2}\delta ^{-1}+C_{21}t^{-1})e^{-c_\flat t/2}\Vert \boldsymbol{\varphi}\Vert _{L_2(\mathcal{O})}+k_r\vartheta (\varepsilon,\delta , r)\Vert \mathbf{F}\Vert _{\mathfrak{H}_r(t)},
\end{split}
\\
\begin{split}
\Vert &\mathbf{p}_\varepsilon (\,\cdot\, ,t)-\mathbf{q} _\varepsilon (\,\cdot\, ,t)\Vert _{L_2(\mathcal{O}')}
\leqslant \varepsilon(\widetilde{C}_{20}t^{-1/2}\delta ^{-1}+\widetilde{C}_{21}t^{-1}) e^{-c_\flat t/2}\Vert \boldsymbol{\varphi}\Vert _{L_2(\mathcal{O})}
+\widetilde{k}_r\vartheta (\varepsilon,\delta , r)\Vert \mathbf{F}\Vert _{\mathfrak{H}_r(t)}.
\end{split}
\end{align*}
Here the constants $k_r$ and $\widetilde{k}_r$ depend only on the problem data \eqref{problem data} and $r$, and
\begin{equation*}
\vartheta (\varepsilon ,\delta, r):=\begin{cases}
\varepsilon\delta^{-1}+\varepsilon ^{1-2/r},&2<r<\infty,\\
\varepsilon\delta^{-1}+\varepsilon (\vert \ln\varepsilon\vert +1), &r=\infty .
\end{cases}
\end{equation*}
\end{theorem}

Finally, under Conditions~\ref{Condition Lambda in L infty} and \ref{Condition tilde Lambda in Lp}, Theorem~\ref{Theorem O' exp no S_eps} implies the following result.

\begin{theorem}
Suppose that the assumtions of Theorem~\textnormal{\ref{Theorem O' solutions with F}} are satisfied.
Suppose that the matrix-valued functions~$\Lambda (\mathbf{x})$ and $\widetilde{\Lambda}(\mathbf{x})$ satisfy Conditions~\textnormal{\ref{Condition Lambda in L infty}}
and~\textnormal{\ref{Condition tilde Lambda in Lp}}, respectively. Suppose that $\check{\mathbf{v}}_\varepsilon$ and $\check{\mathbf{q}}_\varepsilon$ are given by~\eqref{check v_eps}
and~\eqref{check q_eps}. Then for $0<\varepsilon\leqslant\varepsilon _1$ and $\varepsilon ^2\leqslant t <T$ we have
\begin{align*}
\begin{split}
\Vert &\mathbf{u}_\varepsilon (\,\cdot\, ,t)-\check{\mathbf{v}}_\varepsilon (\,\cdot\, ,t)\Vert _{H^1(\mathcal{O}')}
\leqslant \varepsilon (C_{20}t^{-1/2}\delta^{-1}+C_{22}t^{-1})e^{-c_\flat t/2}\Vert \boldsymbol{\varphi}\Vert _{L_2(\mathcal{O})}
+\check{k}_r\vartheta (\varepsilon ,\delta , r)\Vert \mathbf{F}\Vert _{\mathfrak{H}_r(t)},
\end{split}
\\
\begin{split}
\Vert &\mathbf{p}_\varepsilon (\,\cdot\, ,t)-\check{\mathbf{q}}_\varepsilon (\,\cdot\, ,t)\Vert _{L_2(\mathcal{O}')}
\leqslant \varepsilon (\widetilde{C}_{20}t^{-1/2}\delta^{-1}+\widetilde{C}_{22}t^{-1})e^{-c_\flat t/2}\Vert \boldsymbol{\varphi}\Vert _{L_2(\mathcal{O})}
+\widehat{k}_r\vartheta (\varepsilon ,\delta ,r)\Vert \mathbf{F}\Vert _{\mathfrak{H}_r(t)}.
\end{split}
\end{align*}
The constants $\check{k}_r$ and $\widehat{k}_r$ depend only on the problem data~\eqref{problem data}, $r$, $p$, and the norms~$\Vert \Lambda\Vert _{L_\infty}$,
$\Vert \widetilde{\Lambda}\Vert _{L_p(\Omega)}$.
\end{theorem}

\section*{Applications}
For elliptic systems in the whole space $\mathbb{R}^d$, the examples considered below were studied in~\cite{SuAA,MSu15}.
For elliptic systems in a bounded domain, these examples were considered in \cite{MSuPOMI}.

\section{Scalar elliptic operator with a singular potential}

\label{Section 4}

\subsection{Description of the operator}
\label{Subsection 4/1}

We consider the case where $n=1$, $m=d$, $b(\mathbf{D})=\mathbf{D}$, and $g(\mathbf{x})$ is a $\Gamma$-periodic symmetric $(d\times d)$-matrix-valued function \textit{with real entries}
such that $g,g^{-1}\in L_\infty$ and $g(\mathbf{x})>0$. Obviously (see \eqref{<b^*b<}), $\alpha _0=\alpha _1 =1$ and $b(\mathbf{D})^*g^\varepsilon (\mathbf{x}) b(\mathbf{D})=-\mathrm{div}\,g^\varepsilon (\mathbf{x})\nabla$.

Next, let $\mathbf{A}(\mathbf{x})=\mathrm{col}\lbrace A_1(\mathbf{x}),\dots ,A_d(\mathbf{x})\rbrace$, where $A_j(\mathbf{x})$, $j=1,\dots,d$, are $\Gamma$-periodic
real-valued functions such that
\begin{equation}
\label{A_j in L rho}
A_j\in L_\rho (\Omega),\quad\rho=2\;\mbox{for}\;d=1,\quad \rho >d\;\mbox{for}\;d\geqslant 2;\quad j=1,\dots ,d.
\end{equation}
Let $v(\mathbf{x})$ and $\mathcal{V}(\mathbf{x})$ be real-valued $\Gamma$-periodic functions such that
\begin{equation}
\label{v,V condition}
v,\mathcal{V}\!\in\! L_s(\Omega),\quad s\!=\!1\;\mbox{ for } d\!=\!1,\quad s\!>\!d/2\;\mbox{ for }\;d\!\geqslant\! 2;\quad\int\limits_\Omega \!\!v(\mathbf{x})\,d\mathbf{x}\!=\!0.
\end{equation}

In $L_2(\mathcal{O})$, we consider the operator $\mathfrak{B}_{D,\varepsilon}$ given formally by the differential expression
\begin{equation}
\label{mathfrak B_D,eps}
\mathfrak{B}_{\varepsilon}=(\mathbf{D}-\mathbf{A}^\varepsilon(\mathbf{x}))^*g^\varepsilon (\mathbf{x})(\mathbf{D}-\mathbf{A}^\varepsilon (\mathbf{x}))+\varepsilon ^{-1}v^\varepsilon (\mathbf{x})+\mathcal{V}^\varepsilon (\mathbf{x})
\end{equation}
with the Dirichlet condition on~$\partial\mathcal{O}$. The precise definition of the operator~$\mathfrak{B}_{D,\varepsilon}$ is given in terms of the quadratic form
\begin{equation*}
\mathfrak{b}_{D,\varepsilon}[u,u]\!=\!\!\int\limits_\mathcal{O}
\left(\langle g^\varepsilon (\mathbf{D}\!-\!\mathbf{A}^\varepsilon )u,(\mathbf{D}\!-\!\mathbf{A}^\varepsilon )u\rangle\!+(\varepsilon^{-1}v^\varepsilon\!+\mathcal{V}^\varepsilon)\vert u\vert ^2
\right)
\,d\mathbf{x},
\quad u\!\in\! H^1_0(\mathcal{O}).
\end{equation*}

It is easily seen (cf.~\cite[Subsection~13.1]{SuAA}) that expression~\eqref{mathfrak B_D,eps} can be written~as
\begin{equation}
\label{mathfrak B_D,eps in other words}
\mathfrak{B}_{\varepsilon}=\mathbf{D}^*g^\varepsilon (\mathbf{x})\mathbf{D}+\sum _{j=1}^d\left(a_j^\varepsilon (\mathbf{x})D_j+D_j(a_j^\varepsilon (\mathbf{x}))^*\right) +Q^\varepsilon (\mathbf{x}).
\end{equation}
Here $Q(\mathbf{x})$ is a real-valued function defined by
\begin{equation}
\label{Q(x)=V+<gA,A>}
Q(\mathbf{x})=\mathcal{V}(\mathbf{x})+\langle g(\mathbf{x})\mathbf{A}(\mathbf{x}),\mathbf{A}(\mathbf{x})\rangle .
\end{equation}
The complex-valued functions $a_j(\mathbf{x})$ are given by
\begin{equation}
\label{a_j sec. 10.1}
a_j(\mathbf{x})=-\eta _j(\mathbf{x})+i\xi _j(\mathbf{x}),\quad j=1,\dots, d.
\end{equation}
Here $\eta _j(\mathbf{x})$ are the components of the vector-valued function $\boldsymbol{\eta}(\mathbf{x})=g(\mathbf{x})\mathbf{A}(\mathbf{x})$,
and the functions $\xi_j(\mathbf{x})$ are defined by $\xi _j (\mathbf{x})=-\partial _j \Phi (\mathbf{x})$, where
$\Phi (\mathbf{x})$ is the $\Gamma$-periodic solution of the problem $\Delta \Phi(\mathbf{x})=v(\mathbf{x})$, $\int_\Omega \Phi(\mathbf{x})\,d\mathbf{x}=0$.
We have
\begin{equation}
\label{v(x)=sum}
v(\mathbf{x})=-\sum _{j=1}^d\partial _j \xi _j(\mathbf{x}).
\end{equation}
It is easy to check that the functions~\eqref{a_j sec. 10.1} satisfy condition~\eqref{a_j cond} with a suitable $\rho '$ depending on $\rho$ and $s$, and the norms
$\Vert a_j\Vert _{L_{\rho '}(\Omega)}$ are controlled in terms of $\Vert g\Vert _{L_\infty}$, $\Vert \mathbf{A}\Vert _{L_\rho (\Omega)}$, $\Vert v\Vert _{L_s(\Omega)}$, and the parameters
of the lattice~$\Gamma$. (See~\cite[Subsection~13.1]{SuAA}.) The function~\eqref{Q(x)=V+<gA,A>} satisfies condition~\eqref{Q condition} with a suitable $s'=\min \lbrace s;\rho/2\rbrace$.

Let $Q_0(\mathbf{x})$ be a positive definite and bounded $\Gamma$-periodic function. According to~\eqref{B_D,eps}, we introduce a positive definite operator $\mathcal{B}_{D,\varepsilon}:=\mathfrak{B}_{D,\varepsilon}+\lambda Q_0^\varepsilon$.
Here the constant $\lambda$ is chosen according to condition~\eqref{lambda =} for the operator~$\mathcal{B}_{D,\varepsilon}$ with the coefficients $g$, $a_j$, $j=1,\dots,d$, $Q$, and $Q_0$ defined above. The operator $\mathcal{B}_{D,\varepsilon}$ is given by
\begin{equation}
\label{scal ell op}
\mathcal{B}_{\varepsilon}=(\mathbf{D}-\mathbf{A}^\varepsilon (\mathbf{x}))^* g^\varepsilon (\mathbf{x})(\mathbf{D}-\mathbf{A}^\varepsilon (\mathbf{x}))
+\varepsilon ^{-1} v^\varepsilon (\mathbf{x})+\mathcal{V}^\varepsilon (\mathbf{x})+\lambda Q_0^\varepsilon (\mathbf{x}).
\end{equation}
We are interested in the behavior of the exponential of the operator $\widetilde{\mathcal{B}}_{D,\varepsilon}:=f^\varepsilon\mathcal{B}_{D,\varepsilon}f^\varepsilon$, where  $f(\mathbf{x}):=Q_0(\mathbf{x})^{-1/2}$.

For the scalar elliptic operator \eqref{scal ell op}, the problem data \eqref{problem data} are reduced to the following set of parameters:
\begin{equation}
\label{problem data for the scalar operator}
\begin{split}
&d,\rho,s;\Vert g\Vert _{L_\infty}, \Vert g^{-1}\Vert _{L_\infty}, \Vert \mathbf{A}\Vert _{L_\rho (\Omega)}, \Vert v\Vert _{L_s(\Omega)},\Vert \mathcal{V}\Vert _{L_s(\Omega)},\\
&\Vert Q_0\Vert _{L_\infty}, \Vert Q_0^{-1}\Vert _{L_\infty};\; \mbox{the parameters of the lattice }\Gamma;\;\mbox{the domain }\mathcal{O}.
\end{split}
\end{equation}

\subsection{The effective operator}
\label{Subsection 4/2}

Let us write down the effective operator. In the case under consideration, the $\Gamma$-periodic solution of problem \eqref{Lambda problem} is a row:
$\Lambda (\mathbf{x})=i\Psi (\mathbf{x})$, $\Psi (\mathbf{x})=\left(\psi _1(\mathbf{x}),\dots ,\psi _d(\mathbf{x})\right)$,
where $\psi _j\in\widetilde{H}^1(\Omega)$ is the solution of the problem
\begin{equation*}
\mathrm{div}\,g(\mathbf{x})(\nabla \psi _j (\mathbf{x})+\mathbf{e}_j)=0,\quad \int\limits_\Omega \psi_j(\mathbf{x})\,d\mathbf{x}=0.
\end{equation*}
Here $\mathbf{e}_j$, $j=1,\dots,d$, is the standard orthonormal basis in $\mathbb{R}^d$.
Clearly, the functions $\psi _j(\mathbf{x})$ are real-valued, and the entries of $\Lambda(\mathbf{x})$ are purely imaginary.
By \eqref{tilde g}, the columns of the $(d\times d)$-matrix-valued function $\widetilde{g}(\mathbf{x})$ are the vector-valued functions
$g(\mathbf{x})(\nabla \psi _j (\mathbf{x})+\mathbf{e}_j)$, $j=1,\dots,d$. The effective matrix is defined according to \eqref{g^0}: $g^0=\vert \Omega\vert ^{-1}\int_\Omega\widetilde{g}(\mathbf{x})\,d\mathbf{x}$. Clearly, $\widetilde{g}(\mathbf{x})$ and $g^0$ have real entries.

According to \eqref{a_j sec. 10.1} and \eqref{v(x)=sum}, the periodic  solution of problem \eqref{tildeLambda_problem} is represented as  $\widetilde{\Lambda}(\mathbf{x})=\widetilde{\Lambda}_1(\mathbf{x})+i\widetilde{\Lambda}_2(\mathbf{x})$, where the real-valued $\Gamma$-periodic functions $\widetilde{\Lambda}_1(\mathbf{x})$ and $\widetilde{\Lambda}_2(\mathbf{x})$ are the solutions of the problems
\begin{align*}
&-\div g(\mathbf{x})\nabla \widetilde{\Lambda}_1(\mathbf{x})+v(\mathbf{x})=0,\quad\int\limits_\Omega\widetilde{\Lambda}_1(\mathbf{x})\,d\mathbf{x}=0;\\
&-\div g(\mathbf{x})\nabla \widetilde{\Lambda}_2(\mathbf{x})+\div g(\mathbf{x})\mathbf{A}(\mathbf{x})=0,\quad\int\limits_\Omega\widetilde{\Lambda}_2(\mathbf{x})\,d\mathbf{x}=0.
\end{align*}
The column $V$ (see \eqref{V}) has the form $V=V_1+iV_2$, where $V_1$, $V_2$ are the columns with real entries defined by
\begin{align*}
V_1\!\!=\!\vert\Omega\vert ^{-1}\!\!\int\limits_\Omega\!(\nabla\Psi (\mathbf{x}))^tg(\mathbf{x})\nabla\widetilde{\Lambda}_2(\mathbf{x})\,d\mathbf{x},
\quad V_2\!=\!-\vert \Omega\vert ^{-1}\!\!\int\limits_\Omega \!(\nabla \Psi(\mathbf{x}))^tg(\mathbf{x})\nabla \widetilde{\Lambda}_1(\mathbf{x})\,d\mathbf{x}.
\end{align*}
According to \eqref{W}, the constant $W$ is given by
\begin{equation*}
W=\vert \Omega\vert ^{-1}\int\limits_\Omega \left(\langle g(\mathbf{x})\nabla \widetilde{\Lambda}_1(\mathbf{x}),\nabla\widetilde{\Lambda}_1(\mathbf{x})\rangle+\langle g(\mathbf{x})\nabla\widetilde{\Lambda}_2(\mathbf{x}),\nabla\widetilde{\Lambda}_2(\mathbf{x})\rangle\right)\,d\mathbf{x}.
\end{equation*}
The effective operator for $\mathcal{B}_{D,\varepsilon}$ acts as follows
\begin{equation*}
\mathcal{B}_D^0u=-\div g^0\nabla u +2i\langle\nabla u, V_1+\overline{\boldsymbol{\eta}}\rangle +(-W+\overline{Q}+\lambda \overline{Q_0})u,\quad u\in H^2(\mathcal{O})\cap H^1_0(\mathcal{O}).
\end{equation*}
The corresponding differential expression can be written as
\begin{equation}
\label{B_D^0 in Section 4}
\mathcal{B}^0=(\mathbf{D}-\mathbf{A}^0)^*g^0(\mathbf{D}-\mathbf{A}^0)+\mathcal{V}^0+\lambda \overline{Q_0},
\end{equation}
where
\begin{equation*}
\mathbf{A}^0=(g^0)^{-1}(V_1+\overline{g\mathbf{A}}),\quad \mathcal{V}^0=\overline{\mathcal{V}}+\overline{\langle g\mathbf{A},\mathbf{A}\rangle}-\langle g^0\mathbf{A}^0,\mathbf{A}^0\rangle -W.
\end{equation*}

Let $f_0:=(\overline{Q_0})^{-1/2}$. Denote $\widetilde{\mathcal{B}}_D^0:=f_0\mathcal{B}_D^0f_0$.

\subsection{Approximation of the sandwiched operator exponential}

According to Remark~\ref{Remark scalar problem}, in the case under consideration, Conditions~\ref{Condition Lambda in L infty} and~\ref{Condition tilde Lambda in Lp} are satisfied,
and the norms $\Vert \Lambda\Vert _{L_\infty}$ and $\Vert \widetilde{\Lambda}\Vert  _{L_\infty}$ are estimated in terms of the problem data \eqref{problem data for the scalar operator}. Therefore, we can use the corrector which does not involve the smoothing operator:
\begin{equation}
\label{K_D for scalar case}
\mathcal{K}_D^0(t;\varepsilon ):=\Big([\Lambda ^\varepsilon ]\mathbf{D}+[\widetilde{\Lambda}^\varepsilon]\Big)f_0e^{-\widetilde{\mathcal{B}}_D^0t}f_0
=\Big([\Psi ^\varepsilon]\nabla +[\widetilde{\Lambda}^\varepsilon]\Big)f_0e^{-\widetilde{\mathcal{B}}_D^0t}f_0.
\end{equation}
The operator \eqref{G_3(t;eps)} takes the form~$\mathcal{G}_D^0(t;\varepsilon)=-i\mathfrak{G}_D^0(t;\varepsilon)$, where
\begin{equation}
\label{mathcal G_3}
\mathfrak{G}_D^0(t;\varepsilon  )=\widetilde{g}^\varepsilon \nabla f_0e^{-\widetilde{\mathcal{B}}_D^0t}f_0
+g^\varepsilon (\nabla\widetilde{\Lambda})^\varepsilon f_0e^{-\widetilde{\mathcal{B}}_D^0t}f_0.
\end{equation}

 Theorems~\ref{Theorem L2 exp} and~\ref{Theorem exp no S-eps} imply the following result.

\begin{proposition}
\label{Proposition example 1}
Suppose that the assumptions of Subsections~\textnormal{\ref{Subsection 4/1}} and~\textnormal{\ref{Subsection 4/2}} are satisfied.
Suppose that the operators~$\mathcal{K}_D^0(t;\varepsilon )$ and~$\mathfrak{G}_D^0(t;\varepsilon  )$ are given by~\eqref{K_D for scalar case} and~\eqref{mathcal G_3}, respectively.
Suppose that the number~$\varepsilon _1$ is subject to Condition~\textnormal{\ref{condition varepsilon}}.
Then for \hbox{$0<\varepsilon\leqslant\varepsilon _1$} we have
\begin{align*}
\begin{split}
\Vert &f^\varepsilon e^{-\widetilde{\mathcal{B}}_{D,\varepsilon}t}f^\varepsilon -f_0e^{-\widetilde{\mathcal{B}}_D^0t}f_0\Vert _{L_2(\mathcal{O})\rightarrow L_2(\mathcal{O})}
\leqslant C_{15}\varepsilon (t+\varepsilon ^2)^{-1/2}e^{-c_\flat t/2},\quad t\geqslant 0;
\end{split}
\\
\begin{split}
\Vert & f^\varepsilon e^{-\widetilde{\mathcal{B}}_{D,\varepsilon}t}f^\varepsilon -f_0e^{-\widetilde{\mathcal{B}}_D^0t}f_0 -\varepsilon \mathcal{K}_D^0(t;\varepsilon )\Vert _{L_2(\mathcal{O})\rightarrow H^1(\mathcal{O})}
\leqslant C_{18}(\varepsilon ^{1/2}t^{-3/4}+\varepsilon t^{-1})e^{-c_\flat t/2},\quad t>0;
\end{split}
\\
\begin{split}
\Vert & g^\varepsilon \nabla f^\varepsilon e^{-\widetilde{\mathcal{B}}_{D,\varepsilon}t}f^\varepsilon -\mathfrak{G}_D^0(t;\varepsilon  )\Vert  _{L_2(\mathcal{O})\rightarrow L_2(\mathcal{O})}
\leqslant \widetilde{C}_{18}(\varepsilon ^{1/2}t^{-3/4}+\varepsilon t^{-1})e^{-c_\flat t/2},\quad t>0.
\end{split}
\end{align*}
The constants $C_{15}$, $C_{18}$, and $\widetilde{C}_{18}$ depend only on the problem data~\eqref{problem data for the scalar operator}.
\end{proposition}

\subsection{Homogenization of the first initial boundary-value problem for parabolic equation with a singular potential}

\label{Subsec.4.4}

Consider the first initial boundary-value problem for nonhomogeneous parabolic equation with a singular potential:
\begin{equation*}
\begin{cases}
Q_0^\varepsilon (\mathbf{x}) \frac{\partial {u}_\varepsilon}{\partial t}(\mathbf{x},t)&=-
(\mathbf{D}-\mathbf{A}^\varepsilon(\mathbf{x}))^*g^\varepsilon (\mathbf{x})(\mathbf{D}-\mathbf{A}^\varepsilon (\mathbf{x})) {u}_\varepsilon (\mathbf{x},t)\\
&\quad-\left(\varepsilon ^{-1}v^\varepsilon (\mathbf{x})+\mathcal{V}^\varepsilon (\mathbf{x})
+\lambda Q_0^\varepsilon (\mathbf{x})
\right)
 {u}_\varepsilon (\mathbf{x},t)+{F}(\mathbf{x},t),\\
&\qquad\mathbf{x}\in\mathcal{O},\quad t>0;
\\
{u}_\varepsilon (\,\cdot\, ,t) \vert _{\partial \mathcal{O}}&=0,\quad t>0;\\
Q_0^\varepsilon (\mathbf{x}) {u}_\varepsilon (\mathbf{x},0)&={\varphi}(\mathbf{x}),\quad \mathbf{x}\in \mathcal{O}.
\end{cases}
\end{equation*}
Here $\varphi\in L_2(\mathcal{O})$ and ${F}\in \mathfrak{H}_r(T):= L_r((0,T);L_2(\mathcal{O}))$, $0<T\leqslant\infty$, for some~\hbox{$1\leqslant r\leqslant \infty$}.

According to \eqref{effective first problem with F} and \eqref{B_D^0 in Section 4}, the effective problem takes the form
\begin{equation*}
\begin{cases}
\overline{Q_0} \frac{\partial {u}_0}{\partial t}(\mathbf{x},t)&=-
(\mathbf{D}-\mathbf{A}^0)^*g^0(\mathbf{D}-\mathbf{A}^0) {u}_0(\mathbf{x},t)-
\left(\mathcal{V}^0+\lambda \overline{Q_0}\right) {u}_0 (\mathbf{x},t)\\
&\quad+F(\mathbf{x},t),
\quad
\mathbf{x}\in\mathcal{O},\quad t>0;
\\
{u}_0  (\,\cdot\, ,t)\vert _{\partial \mathcal{O}}&=0,\quad t>0;\\
\overline{Q_0} {u}_0 (\mathbf{x},0)&={\varphi}(\mathbf{x}),\quad \mathbf{x}\in \mathcal{O}.
\end{cases}
\end{equation*}

Applying Theorems~\ref{Theorem 3.1} and~\ref{Theorem 3.6}, we obtain the following result.

\begin{proposition}
Suppose that the number $\varepsilon _1$ is subject to Condition~\textnormal{\ref{condition varepsilon}}.
Suppose that the assumptions of Subsection~\textnormal{\ref{Subsec.4.4}} are satisfied, and $1<r\leqslant \infty$. Then for $0<\varepsilon\leqslant \varepsilon _1$ and $0<t<T$ we have
\begin{equation*}
\Vert {u}_\varepsilon (\,\cdot\, ,t)-{u}_0(\,\cdot\, ,t)\Vert _{L_2(\mathcal{O})}
\leqslant C_{15}\varepsilon (t+\varepsilon ^2)^{-1/2} e^{-c_\flat t/2}\Vert {\varphi}\Vert _{L_2(\mathcal{O})}
+
c_r\theta (\varepsilon ,r)\Vert{F}\Vert _{\mathfrak{H}_r(T)}.
\end{equation*}
Here $\theta (\varepsilon ,r)$ is given by \eqref{theta}.

Assuming that $t\geqslant \varepsilon ^2$, we put $w_\varepsilon (\,\cdot\, ,t):=f_0e^{-\widetilde{\mathcal{B}}_D^0 \varepsilon ^2}f_0^{-1} u_0(\,\cdot\, ,t-\varepsilon ^2)$.
Denote
$\check{v}_\varepsilon (\,\cdot\, ,t):= u_0(\,\cdot\, ,t)+\varepsilon\Psi^\varepsilon \nabla w_\varepsilon (\,\cdot\,,t)+\varepsilon \widetilde{\Lambda}^\varepsilon w_\varepsilon (\,\cdot\, ,t)$ and
$\check{q}_\varepsilon (\,\cdot\, ,t):=\widetilde{g}^\varepsilon \nabla w_\varepsilon (\,\cdot\, ,t)+g^\varepsilon \bigl(\nabla\widetilde{\Lambda}\bigr)^\varepsilon w_\varepsilon (\,\cdot\, ,t)$.
In addition, assume that $2<r\leqslant \infty$.
Then for $0<\varepsilon\leqslant \varepsilon _1$ and $\varepsilon^2\leqslant  t<T$ we have
\begin{align*}
\begin{split}
\Vert &{u}_\varepsilon (\,\cdot\, ,t)-\check{{v}}_\varepsilon (\,\cdot\, ,t)\Vert _{H^1(\mathcal{O})}
\leqslant 2C_{18}\varepsilon ^{1/2}t^{-3/4}e^{-c_\flat t/2}\Vert {\varphi}\Vert _{L_2(\mathcal{O})}
+c_r'\omega (\varepsilon ,r)\Vert {F}\Vert _{\mathfrak{H}_r(t)},
\end{split}
\\
\begin{split}
\Vert &g^\varepsilon\nabla u_\varepsilon (\,\cdot\, ,t)-\check{{q}}_\varepsilon (\,\cdot\, ,t)\Vert _{L_2(\mathcal{O})}
\leqslant
2\widetilde{C}_{18}\varepsilon ^{1/2}t^{-3/4}e^{-c_\flat t/2}\Vert {\varphi}\Vert _{L_2(\mathcal{O})}
+c_r''\omega (\varepsilon ,r)\Vert {F}\Vert _{\mathfrak{H}_r(t)}.
\end{split}
\end{align*}
Here $\omega(\varepsilon ,r)$ is given by \eqref{omega (eps,r)}.
The constants $C_{15}$, $C_{18}$, and $\widetilde{C}_{18}$ depend only on the problem data \eqref{problem data for the scalar operator}.
The constants $c_r$, $c_r'$, and $c_r''$ depend on the same parameters and on $r$.
\end{proposition}

\section{The scalar operator with a strongly singular \\ potential of order~$\varepsilon ^{-2}$}

\label{Section 5}

Homogenization of the first initial boundary-value problem for parabolic equation with a strongly singular potential was studied in \cite{Al}.
Some motivations can be found in \cite[\S 1]{Al}). However, the results of~\cite{Al} cannot be formulated in the uniform operator topology.

\subsection{Description of the operator}
\label{Subsection 5.1}

Let $\check{g}(\mathbf{x})$ be a $\Gamma$-periodic symmetric $(d\times d)$-matrix-valued function in $\mathbb{R}^d$ with real entries such that $\check{g},\check{g}^{-1}\in L_\infty $ and \hbox{$\check{g}(\mathbf{x})>0$}. Let~$\check{v}(\mathbf{x})$ be a real-valued $\Gamma$-periodic function such that
\begin{equation*}
\check{v}\in L_s(\Omega),\quad s=1\;\mbox{for}\;d=1,\quad s>d/2\;\mbox{for}\;d\geqslant 2.
\end{equation*}
By $\check{\mathcal{A}}$ we denote the operator in $L_2(\mathbb{R}^d)$ corresponding to the quadratic form
$$
\int\limits_{\mathbb{R}^d}\left(\langle\check{g}(\mathbf{x})\mathbf{D}u,\mathbf{D}u\rangle+\check{v}(\mathbf{x})\vert u\vert ^2\right)\,d\mathbf{x},\quad u\in H^1(\mathbb{R}^d).
$$
Adding a constant to the potential $\check{v}(\mathbf{x})$, we assume that the bottom of the spectrum of $\check{\mathcal{A}}$ is the point zero.
Then the operator $\check{\mathcal{A}}$ admits a factorization with the help of the eigenfunction of the operator $\mathbf{D}^*\check{g}(\mathbf{x})\mathbf{D}+\check{v}(\mathbf{x})$ on the cell $\Omega$ (with periodic boundary conditions) corresponding to the eigenvalue $\lambda=0$ (see \cite[Chapter~6, Subsection~1.1]{BSu}).
Apparently, such factorization trick was first used in homogenization problems in~\cite{Zh83,K}.

In $L_2(\mathcal{O})$, we consider the operator~$\check{\mathcal{A}}_D$ given by the expression~$\mathbf{D}^*\check{g}(\mathbf{x})\mathbf{D}+\check{v}(\mathbf{x})$
with the Dirichlet condition on~$\partial\mathcal{O}$.
The precise definition of the operator~$\check{\mathcal{A}}_D$ is given in terms of the quadratic form
\begin{equation}
\label{check a}
\check{\mathfrak{a}}_D[u,u]=\int\limits_\mathcal{O}\left(\langle\check{g}(\mathbf{x})\mathbf{D}u,\mathbf{D}u\rangle+\check{v}(\mathbf{x})\vert u\vert ^2\right)\,d\mathbf{x},\quad u\in H^1_0(\mathcal{O}).
\end{equation}
The operator~$\check{\mathcal{A}}_D$ inherits factorization of the operator $\check{\mathcal{A}}$. To describe this factorization, we consider the equation
\begin{equation}
\label{eq omega}
\mathbf{D}^*\check{g}(\mathbf{x})\mathbf{D}\omega (\mathbf{x})+\check{v}(\mathbf{x})\omega(\mathbf{x})=0.
\end{equation}
There exists a $\Gamma$-periodic solution $\omega\in \widetilde{H}^1(\Omega)$ of this equation defined up to a constant factor.
We can fix this factor so that $\omega (\mathbf{x})>0$ and
\begin{equation}
\label{mean value of omega-tilda}
\int\limits_\Omega \omega ^2(\mathbf{x})\,d\mathbf{x}=\vert \Omega \vert .
\end{equation}
Moreover, the solution is positive definite and bounded: \hbox{$0<\omega _0\leqslant \omega (\mathbf{x})\leqslant\omega_1<\infty$}.
The norms~$\Vert \omega\Vert _{L_\infty}$ and~$\Vert \omega ^{-1}\Vert _{L_\infty}$ are controlled in terms of~$\Vert \check{g}\Vert _{L_\infty}$, $\Vert \check{g}^{-1}\Vert _{L_\infty}$, and $\Vert \check{v}\Vert _{L_s(\Omega)}$. Note that $\omega$ and $\omega ^{-1}$ are multipliers in~$H^1_0(\mathcal{O})$.

Substituting $u=\omega z$ and taking \eqref{eq omega} into account, we represent the form~\eqref{check a}~as
\begin{equation*}
\check{\mathfrak{a}}_D[u,u]=\int\limits_\mathcal{O} \omega (\mathbf{x})^2\langle\check{g}(\mathbf{x})\mathbf{D}z,\mathbf{D}z\rangle\,d\mathbf{x},\quad u=\omega z,\quad z\in H^1_0(\mathcal{O}).
\end{equation*}
Hence, the differential expression for the operator $\check{\mathcal{A}}_D$ admits a factorization
\begin{equation}
\label{check A_D}
\check{\mathcal{A}}=\omega ^{-1}\mathbf{D}^* g\mathbf{D}\omega ^{-1},\quad g=\omega^2\check{g}.
\end{equation}

Now, we consider the operator~$\check{\mathcal{A}}_{D,\varepsilon}$ with rapidly oscillating coefficients acting in $L_2(\mathcal{O})$ and given by
\begin{equation}
\label{check A_D,eps}
\check{\mathcal{A}}_{\varepsilon}=(\omega ^\varepsilon) ^{-1}\mathbf{D}^* g^\varepsilon\mathbf{D}(\omega ^\varepsilon) ^{-1},\quad g=\omega^2\check{g},
\end{equation}
with the Dirichlet boundary condition. In the initial terms, expression \eqref{check A_D,eps} takes the form
\begin{equation}
\label{check A_D,eps in initial terms}
\check{\mathcal{A}}_{\varepsilon}=\mathbf{D}^*\check{g}^\varepsilon\mathbf{D}+\varepsilon ^{-2}\check{v}^\varepsilon.
\end{equation}

Next, let $\mathbf{A}(\mathbf{x}) =\mathrm{col}\,\lbrace A_1(\mathbf{x}),\dots ,A_d(\mathbf{x})\rbrace $, where $A_j(\mathbf{x})$ are $\Gamma$-periodic real-valued functions
satisfying~\eqref{A_j in L rho}. Let~$\widehat{v}(\mathbf{x})$ and~$\check{\mathcal{V}}(\mathbf{x})$ be $\Gamma$-periodic real-valued functions such that
\begin{equation}
\label{hat v, check V}
\widehat{v},\check{\mathcal{V}}\in L_s(\Omega), \quad s=1\;\mbox{for}\;d=1,\quad s>d/2\;\mbox{for}\;d\geqslant 2;\\
\int\limits_\Omega \widehat{v}(\mathbf{x})\omega ^2 (\mathbf{x})\,d\mathbf{x}=0.
\end{equation}
In $L_2(\mathcal{O})$, we consider the operator $\widetilde{\mathfrak{B}}_{D,\varepsilon}$ given formally by the differential expression
\begin{equation*}
\widetilde{\mathfrak{B}}_{\varepsilon}
=(\mathbf{D}-\mathbf{A}^\varepsilon )^* \check{g}^\varepsilon (\mathbf{D}-\mathbf{A}^\varepsilon)+\varepsilon ^{-2}\check{v}^\varepsilon +\varepsilon ^{-1}\widehat{v}^\varepsilon +\check{\mathcal{V}}^\varepsilon
\end{equation*}
with the Dirichlet condition on~$\partial\mathcal{O}$. The precise definition is given in terms of the quadratic form.

We put
\begin{equation}
\label{v=, V=}
v(\mathbf{x}):=\widehat{v}(\mathbf{x})\omega ^2(\mathbf{x}),\quad \mathcal{V}(\mathbf{x}):=\check{\mathcal{V}}(\mathbf{x})\omega ^2(\mathbf{x}).
\end{equation}
By \eqref{check A_D,eps} and \eqref{check A_D,eps in initial terms}, we have $\widetilde{\mathfrak{B}}_{D,\varepsilon}=(\omega ^\varepsilon )^{-1}\mathfrak{B}_{D,\varepsilon}(\omega ^\varepsilon )^{-1}$, where the operator~$\mathfrak{B}_{D,\varepsilon}$ is given by the expression~\eqref{mathfrak B_D,eps} with the Dirichlet condition on~$\partial\mathcal{O}$;
$g$ is defined by \eqref{check A_D}, and $v$, $\mathcal{V}$ are given by~\eqref{v=, V=}. By~\eqref{hat v, check V} and the properties of $\omega$, the coefficients~$v$ and~$\mathcal{V}$ satisfy~\eqref{v,V condition}. Then the operator~$\mathfrak{B}_{D,\varepsilon}$ can be represented in the form~\eqref{mathfrak B_D,eps in other words}, where~$a_j$, $j=1,\dots,d$, and~$Q$ are constructed in terms of $g$, $\mathbf{A}$, $v$, and~$\mathcal{V}$ according to~\eqref{Q(x)=V+<gA,A>}, \eqref{a_j sec. 10.1}.

The constant~$\lambda$ is chosen according to condition~\eqref{lambda =} for the operator with the same coefficients~$g$, $a_j$, $j=1,\dots ,d$, and $Q$,
as the coefficients of $\mathfrak{B}_{D,\varepsilon}$, and the coefficient $Q_0(\mathbf{x}):=\omega ^2(\mathbf{x})$. Then the operators  $\widetilde{\mathcal{B}}_{D,\varepsilon}:=\widetilde{\mathfrak{B}}_{D,\varepsilon}+\lambda I$ and $\mathcal{B}_{D,\varepsilon}:=\mathfrak{B}_{D,\varepsilon}+\lambda Q_0^\varepsilon$ are related by~$\widetilde{\mathcal{B}}_{D,\varepsilon}=(\omega ^\varepsilon )^{-1}\mathcal{B}_{D,\varepsilon}(\omega ^\varepsilon )^{-1}$.

The following set of parameters is called the~``problem data'':
\begin{equation}
\label{problem data for Schredinger}
\begin{split}
&d,\rho,s;\quad\Vert \check{g}\Vert _{L_\infty},\; \Vert \check{g}^{-1}\Vert _{L_\infty},\; \Vert \mathbf{A}\Vert _{L_\rho (\Omega)},\; \Vert \check{v}\Vert _{L_s(\Omega)}, \;\Vert \widehat{v}\Vert _{L_s(\Omega)},\;\Vert \check{\mathcal{V}}\Vert _{L_s(\Omega)};
\\
&\mbox{the parameters of the lattice }\Gamma ;\quad\mbox{the domain }\mathcal{O} .
\end{split}
\end{equation}

\subsection{Homogenization of the first initial boundary-value problem for the parabolic equation with strongly singular potential}

\label{Subsec. 5.3}

We apply Proposition~\ref{Proposition example 1} to the operator $\widetilde{\mathcal{B}}_{D,\varepsilon}$ described in Subsection~\ref{Subsection 5.1}.
We have $f(\mathbf{x})=\omega (\mathbf{x})^{-1}$, whence, by \eqref{mean value of omega-tilda}, $f_0=1$ and $\widetilde{\mathcal{B}}_D^0=\mathcal{B}_D^0$.
The coefficients $g^0$, $\mathbf{A}^0$, and $\mathcal{V}^0$ of the effective operator are constructed in terms of $g$, $\mathbf{A}$, $v$, and $\mathcal{V}$ (see~\eqref{check A_D,eps} and~\eqref{v=, V=}), as described in Subsection~\ref{Subsection 4/2}. We apply the results to homogenization of the solution of the first initial boundary-value problem
\begin{equation*}
\begin{cases}
\frac{\partial {u}_\varepsilon}{\partial t}(\mathbf{x},t)&=-
(\mathbf{D}-\mathbf{A}^\varepsilon(\mathbf{x}))^*\check{g}^\varepsilon (\mathbf{x})(\mathbf{D}-\mathbf{A}^\varepsilon (\mathbf{x})) {u}_\varepsilon (\mathbf{x},t)\\
&\quad\!-\left(\varepsilon ^{-2}\check{v}^\varepsilon +
\varepsilon ^{-1}\widehat{v}^\varepsilon (\mathbf{x})+\check{\mathcal{V}}^\varepsilon (\mathbf{x})
+\lambda I
\right)
 {u}_\varepsilon (\mathbf{x},t),\quad
\mathbf{x}\!\in\!\mathcal{O},\quad t\!>\!0;
\\
{u}_\varepsilon (\,\cdot\, ,t) \vert _{\partial \mathcal{O}}&=0,\quad t>0;\\
 {u}_\varepsilon (\mathbf{x},0)&=\omega ^\varepsilon (\mathbf{x})^{-1}\varphi(\mathbf{x}),\quad \mathbf{x}\in \mathcal{O}.
\end{cases}
\end{equation*}
Here $\varphi\in L_2(\mathcal{O})$. (For simplicity, we consider a homogeneous equation.) Then $u_\varepsilon (\,\cdot\, ,t)= e^{-\widetilde{\mathcal{B}}_{D,\varepsilon}t} (\omega ^\varepsilon )^{-1}\varphi$.

Let $u_0$ be the solution of the homogenized problem
\begin{equation*}
\begin{cases}
\frac{\partial u_0}{\partial t}(\mathbf{x},t)&=-
(\mathbf{D}-\mathbf{A}^0)^*g^0(\mathbf{D}-\mathbf{A}^0) {u}_0(\mathbf{x},t)-
\left(\mathcal{V}^0+\lambda \right) {u}_0 (\mathbf{x},t),\\
&\qquad\mathbf{x}\in\mathcal{O},\quad t>0;\\
u_0(\,\cdot\, ,t) \vert _{\partial \mathcal{O}}&=0,\quad t>0;\\
u_0(\mathbf{x},0)&=\varphi(\mathbf{x}),\quad \mathbf{x}\in\mathcal{O}.
\end{cases}
\end{equation*}

Proposition~\ref{Proposition example 1} implies the following result.

\begin{proposition}
Suppose that the assumptions of Subsection~\textnormal{\ref{Subsec. 5.3}} are satisfied. Denote
\begin{align*}
& \check{v}_\varepsilon (\,\cdot\, ,t):=u_0(\,\cdot\, ,t)+\varepsilon \Psi ^\varepsilon \nabla u_0(\,\cdot\, ,t)+\varepsilon \widetilde{\Lambda}^\varepsilon u_0(\,\cdot\, ,t),\\
& \check{q}_\varepsilon (\,\cdot\, ,t):=\widetilde{g}^\varepsilon \nabla u_0(\,\cdot\, ,t)+g^\varepsilon (\nabla \widetilde{\Lambda})^\varepsilon u_0(\,\cdot\, ,t).
\end{align*}
Then for $0<\varepsilon\leqslant\varepsilon _1$ we have
\begin{align*}
\Vert & (\omega ^\varepsilon )^{-1}u_\varepsilon (\,\cdot\, ,t)- u_0(\,\cdot\, ,t)\Vert _{L_2(\mathcal{O})}
\leqslant C_{15}\varepsilon (t+\varepsilon ^2)^{-1/2}e^{-c_\flat t/2}\Vert \varphi \Vert _{L_2(\mathcal{O})},\quad t\geqslant 0;
\\
\Vert & (\omega ^\varepsilon )^{-1}u_\varepsilon (\,\cdot\, ,t)- \check{v}_\varepsilon (\,\cdot\, ,t)\Vert _{H^1(\mathcal{O})}
\leqslant C_{18}(\varepsilon ^{1/2}t^{-3/4}+\varepsilon t^{-1})e^{-c_\flat t/2}\Vert \varphi \Vert _{L_2(\mathcal{O})},
\\
\Vert & g^\varepsilon \nabla (\omega ^\varepsilon )^{-1}u_\varepsilon (\,\cdot\, ,t)-\check{q}_\varepsilon (\,\cdot\, ,t)\Vert _{L_2(\mathcal{O})}
\leqslant \widetilde{C}_{18}(\varepsilon ^{1/2}t^{-3/4}+\varepsilon t^{-1})e^{-c_\flat t/2}\Vert \varphi \Vert _{L_2(\mathcal{O})},
\end{align*}
$t>0$.
The constants $C_{15}$, $C_{18}$, and $\widetilde{C}_{18}$ depend on the problem data~\eqref{problem data for Schredinger}.
\end{proposition}

Note that, in the presence of a strongly singular potential in the equation, not the solution $u_\varepsilon$, but the product
$(\omega ^\varepsilon )^{-1}u_\varepsilon$ admits a ``good approximation''. This shows that the nature of the results of \S\ref{Section 5} and~\S\ref{Section 4} is different.

\section*{Appendix}
\addcontentsline{toc}{section}{Appendix}
In Appendix, we consider the case where $d\geqslant 3$ and prove the statements about removal of the smoothing operator~$S_\varepsilon$
in the case of sufficiently smooth boundary (Lemma~\ref{Lemma K-K^0} and Theorem~\ref{Theorem smooth boundary})
and in  the case of a strictly interior subdomain (Lemma~\ref{Lemma K-K0 strictly interior} and Theorem~\ref{Theorem O' no S_eps no conditions on Lambda's}).

\section{The properties of the matrix-valued functions $\Lambda$ and $\widetilde{\Lambda}$}

\label{Section 6}

We need the following results; see \cite[Lemma~2.3]{PSu} and \cite[Lemma~3.4]{MSu15}.

\begin{lemma}
\label{Lemma Lambda-varepsilon}
Let $\Lambda $ be the $\Gamma$-periodic solution of problem~\eqref{Lambda problem}. Then for any function $u\in C_0^\infty (\mathbb{R}^d)$ and $\varepsilon >0$ we have
\begin{equation*}
\int\limits_{\mathbb{R}^d}\vert (\mathbf{D}\Lambda )^\varepsilon (\mathbf{x})\vert ^2\vert u(\mathbf{x})\vert ^2\,d\mathbf{x}
\leqslant\beta _1\Vert u\Vert ^2_{L_2(\mathbb{R}^d)}
+\beta _2 \varepsilon ^2\int\limits_{\mathbb{R}^d}\vert \Lambda ^\varepsilon (\mathbf{x})\vert ^2 \vert \mathbf{D}u(\mathbf{x})\vert ^2\,d\mathbf{x}.
\end{equation*}
The constants $\beta _1$ and $\beta _2$ depend on $m$, $d$, $\alpha _0$, $\alpha _1$, $\Vert g\Vert _{L_\infty}$, and $\Vert g^{-1}\Vert _{L_\infty}$.
\end{lemma}

\begin{lemma}
\label{Lemma on Lambda-tilda-varepsilon}
Let $\widetilde{\Lambda}$ be the $\Gamma$-periodic solution of problem \eqref{tildeLambda_problem}.
Then for any function $u\in C_0^\infty (\mathbb{R}^d)$ and $0< \varepsilon \leqslant 1$ we have
\begin{equation*}
\begin{split}
\int\limits_{\mathbb{R}^d}\vert (\mathbf{D}\widetilde{\Lambda})^\varepsilon(\mathbf{x})\vert ^2\vert u(\mathbf{x})\vert ^2\,d\mathbf{x}
\leqslant \widetilde{\beta} _1 \Vert u\Vert ^2_{H^1(\mathbb{R}^d)}
+\widetilde{\beta} _2 \varepsilon ^2\int\limits_{\mathbb{R}^d}\vert \widetilde{\Lambda}^\varepsilon (\mathbf{x})\vert ^2 \vert \mathbf{D}u(\mathbf{x})\vert ^2 \,d\mathbf{x}.
\end{split}
\end{equation*}
The constants $\widetilde{\beta} _1$ and $\widetilde{\beta} _2$ depend only on $n$, $d$, $\alpha _0$, $\alpha _1$, $\rho$, $\Vert g\Vert _{L_\infty}$, $\Vert g^{-1}\Vert _{L_\infty}$,
the norms $\Vert a_j\Vert _{L_\rho (\Omega)}$, $j=1,\dots ,d$, and the parameters of the lattice $\Gamma$.
\end{lemma}

Below in \S\ref{Section removing steklov operator} we will need the following multiplier properties of the matrix-valued functions $\Lambda (\mathbf{x})$ and $\widetilde{\Lambda}(\mathbf{x})$.

\begin{lemma}
\label{Lemma Lambda multiplicator properties}
Suppose that a matrix-valued function $\Lambda({\mathbf x})$ is the~$\Gamma$-periodic solution of problem~\textnormal{\eqref{Lambda problem}}. Let $d\geqslant 3$ and $l=d/2$.

\noindent $1^\circ$. For $0< \varepsilon \leqslant 1$ and ${\mathbf u} \in H^{l-1}({\mathbb R}^d; {\mathbb C}^m)$
we have $\Lambda^\varepsilon {\mathbf u} \in L_2({\mathbb R}^d;{\mathbb C}^n)$ and
\begin{equation}
\label{9.0}
\| \Lambda^\varepsilon {\mathbf u}\|_{L_2({\mathbb R}^d)}
 \leqslant C^{(0)} \| {\mathbf u}\|_{H^{l-1}({\mathbb R}^d)}.
\end{equation}

\noindent $2^\circ$. For $0< \varepsilon \leqslant 1$ and ${\mathbf u} \in H^{l}({\mathbb R}^d; {\mathbb C}^m)$ we have
 $\Lambda^\varepsilon {\mathbf u} \in H^1({\mathbb R}^d;{\mathbb C}^n)$ and
\begin{equation}
\label{9.1}
\| \Lambda^\varepsilon {\mathbf u}\|_{H^1({\mathbb R}^d)} \leqslant  C^{(1)} \varepsilon^{-1}
\| {\mathbf u}\|_{L_2({\mathbb R}^d)} + C^{(2)}  \|{\mathbf u}\|_{H^l({\mathbb R}^d)}.
\end{equation}
The constants $C^{(0)}$, $C^{(1)}$, and $C^{(2)}$ depend on $m$, $d$, $\alpha_0$, $\alpha_1$,
$\|g\|_{L_\infty}$, $\|g^{-1}\|_{L_\infty}$, and the parameters of the lattice~$\Gamma$.
\end{lemma}

\begin{proof}
It suffices to check \eqref{9.0} and~\eqref{9.1} for~${\mathbf u}\! \in\!\! C_0^\infty\!({\mathbb R}^d; {\mathbb C}^m\!)$.
Substituting ${\mathbf x}= \varepsilon {\mathbf y}$, $\varepsilon^{d/2}{\mathbf u}({\mathbf x})={\mathbf U}({\mathbf y})$,
we obtain
\begin{equation}
\label{9.2}
\begin{split}
\| \Lambda^\varepsilon {\mathbf u}\|^2_{L_2({\mathbb R}^d)} \leqslant
\int\limits_{{\mathbb R}^d} |\Lambda(\varepsilon^{-1} {\mathbf x})|^2
|{\mathbf u}({\mathbf x})|^2 \,d{\mathbf x}& =
\int\limits_{{\mathbb R}^d} |\Lambda({\mathbf y})|^2
|{\mathbf U}({\mathbf y})|^2 \,d{\mathbf y}
\\
\qquad=
\sum_{{\mathbf a}\in \Gamma} \int\limits_{\Omega + {\mathbf a}}|\Lambda({\mathbf y})|^2
|{\mathbf U}({\mathbf y})|^2 \,d{\mathbf y} &\leqslant
\sum_{{\mathbf a}\in \Gamma}   \|\Lambda\|^2_{L_{2 \nu}(\Omega)} \| {\mathbf U}\|_{L_{2 \nu'}(\Omega + {\mathbf a})}^2,
\end{split}
\end{equation}
where $\nu^{-1}+(\nu ')^{-1}=1$. We choose $\nu$ so that the embedding
$H^1(\Omega) \hookrightarrow  L_{2\nu}(\Omega)$ is continuous, i.~e., $\nu= d (d-2)^{-1}$. Then
\begin{equation}
\label{9.3}
\|\Lambda\|_{L_{2\nu}(\Omega)}^2 \leqslant c_\Omega \|\Lambda\|^2_{H^1(\Omega)},
\end{equation}
where the constant $c_\Omega$ depends only on the dimension $d$ and the lattice $\Gamma$.
We have $2\nu '=d$. Since the embedding $H^{l-1}(\Omega) \hookrightarrow L_{d}(\Omega)$ is continuous, we have
\begin{equation}
\label{9.4}
\| {\mathbf U}\|_{L_{d}(\Omega + {\mathbf a})}^2
 \leqslant c_\Omega'
\| {\mathbf U}\|_{H^{l-1}(\Omega + {\mathbf a})}^2,
\end{equation}
where the constant $c_\Omega'$ depends only on the dimension $d$ and the lattice~$\Gamma$.
Now, from \eqref{9.2}--\eqref{9.4} it follows that
\begin{equation}
\label{9.5}
\int\limits_{{\mathbb R}^d} |\Lambda ^\varepsilon({\mathbf x})|^2
|{\mathbf u}({\mathbf x})|^2 \,d{\mathbf x}
 \leqslant c_\Omega c_\Omega' \|\Lambda\|^2_{H^1(\Omega)}
\| {\mathbf U}\|_{H^{l-1}({\mathbb R}^d)}^2.
\end{equation}

Obviously, for~$0< \varepsilon \leqslant 1$ we have
$
\| {\mathbf U}\|_{H^{l-1}({\mathbb R}^d)} \leqslant \| {\mathbf u}\|_{H^{l-1}({\mathbb R}^d)}.
$
Combining this with \eqref{9.5} and \eqref{Lamba in H1<=}, we see that
\begin{equation}
\label{9.6}
\int\limits_{{\mathbb R}^d} |\Lambda ^\varepsilon({\mathbf x})|^2
|{\mathbf u}({\mathbf x})|^2 \,d{\mathbf x}
 \leqslant c_\Omega c_\Omega'  M^2 \| {\mathbf u}\|_{H^{l-1}({\mathbb R}^d)}^2, \quad
{\mathbf u} \in C_0^\infty({\mathbb R}^d; {\mathbb C}^m),
\end{equation}
which proves estimate \eqref{9.0} with the constant~$C^{(0)}:= (c_\Omega c_\Omega')^{1/2}  M$.

Next, by Lemma~\ref{Lemma Lambda-varepsilon},
\begin{equation}
\label{9.7}
\begin{split}
\| {\mathbf D}(\Lambda^\varepsilon {\mathbf u})\|^2_{L_2({\mathbb R}^d)}
&\leqslant 2 \varepsilon^{-2}
\int\limits_{{\mathbb R}^d} | ({\mathbf D}\Lambda)^\varepsilon ({\mathbf x}) {\mathbf u}({\mathbf x})  |^2
\,d{\mathbf x} +
2 \int\limits_{{\mathbb R}^d} | \Lambda^\varepsilon ({\mathbf x})|^2 | {\mathbf D}{\mathbf u}({\mathbf x})  |^2
\,d{\mathbf x}
\\
 &\qquad\leqslant 2 \beta_1 \varepsilon^{-2}
 \int\limits_{{\mathbb R}^d} |{\mathbf u}({\mathbf x})  |^2\,d{\mathbf x}
+ 2 (1+\beta_2)\int\limits_{{\mathbb R}^d} | \Lambda^\varepsilon ({\mathbf x})|^2 | {\mathbf D}{\mathbf u}({\mathbf x})  |^2
\,d{\mathbf x}.
\end{split}
\end{equation}
From \eqref{9.6} (with ${\mathbf u}$ replaced by the derivatives $\partial_j {\mathbf u}$) it follows that
\begin{equation}
\label{9.8}
\int\limits_{{\mathbb R}^d} | \Lambda^\varepsilon ({\mathbf x})|^2 | {\mathbf D}{\mathbf u}({\mathbf x})  |^2
\,d{\mathbf x}
\leqslant c_\Omega c_\Omega'  M^2 \| {\mathbf u}\|_{H^{l}({\mathbb R}^d)}^2, \quad
{\mathbf u} \in C_0^\infty({\mathbb R}^d; {\mathbb C}^m).
\end{equation}

As a result, relations \eqref{9.6}--\eqref{9.8} imply inequality \eqref{9.1} with the constants
$C^{(1)}:=(2\beta_1)^{1/2}$ and $C^{(2)}:=M (3+2 \beta_2)^{1/2} (c_\Omega c_\Omega')^{1/2}$.
\end{proof}

Using the extension operator $P_{\mathcal O}$ satisfying estimates \eqref{PO},
we deduce the following statement from Lemma~\ref{Lemma Lambda multiplicator properties}($1^\circ$).

\begin{corollary}
\label{Corollary Lambda eps multiplier}
Suppose that the assumptions of Lemma~\textnormal{\ref{Lemma Lambda multiplicator properties}} are satisfied.
Then the operator $[\Lambda^\eps]$ is continuous from $H^{l-1}({\mathcal O};{\mathbb C}^m)$ to $L_2({\mathcal O};{\mathbb C}^n)$ and
\begin{equation*}
\|[\Lambda^\eps] \|_{ H^{l-1}({\mathcal O}) \to L_2({\mathcal O})}
\leqslant C^{(0)} C^{(l-1)}_{\mathcal O}.
\end{equation*}
\end{corollary}

The following statement can be checked similarly to Lemma~\ref{Lemma Lambda multiplicator properties},
by using Lemma~\ref{Lemma on Lambda-tilda-varepsilon} and estimate~\eqref{tilde Lambda in H1 <=}.

\begin{lemma}
\label{Lemma tilde Lambda multiplicator properties}
Suppose that a matrix-valued function~$\widetilde{\Lambda}(\mathbf{x})$ is the~$\Gamma$-periodic solution of problem~\eqref{tildeLambda_problem}.
Let $d\geqslant 3$ and $l=d/2$.

\noindent
$1^\circ$. For $0<\varepsilon\leqslant 1$ and $\mathbf{u}\in H^{l-1}(\mathbb{R}^d;\mathbb{C}^n)$ we have
$\widetilde{\Lambda}^\varepsilon\mathbf{u}\in L_2(\mathbb{R}^d;\mathbb{C}^n)$ and
\begin{equation*}
\Vert \widetilde{\Lambda}^\varepsilon\mathbf{u}\Vert _{L_2(\mathbb{R}^d)}
\leqslant \widetilde{C}^{(0)}\Vert \mathbf{u}\Vert _{H^{l-1}(\mathbb{R}^d)}.
\end{equation*}

\noindent
$2^\circ$. For $0<\varepsilon\leqslant 1$ and $\mathbf{u}\in H^l(\mathbb{R}^d;\mathbb{C}^n)$ we have
$\widetilde{\Lambda}^\varepsilon\mathbf{u}\in H^1(\mathbb{R}^d;\mathbb{C}^n)$ and
\begin{equation*}
\Vert \widetilde{\Lambda}^\varepsilon\mathbf{u}\Vert _{H^1(\mathbb{R}^d)}
\leqslant \widetilde{C}^{(1)}\varepsilon ^{-1}\Vert \mathbf{u}\Vert _{H^1(\mathbb{R}^d)}
+\widetilde{C}^{(2)}\Vert \mathbf{u}\Vert _{H^l(\mathbb{R}^d)}.
\end{equation*}
The constants
$$
\widetilde{C}^{(0)}:=(c_\Omega c_\Omega ')^{1/2}\widetilde{M},\quad \widetilde{C}^{(1)}:=(2\widetilde{\beta}_1)^{1/2},\quad \widetilde{C}^{(2)}:=\sqrt{2}(\widetilde{\beta}_2+1)^{1/2}(c_\Omega c_\Omega ')^{1/2}\widetilde{M}
$$
depend only on the problem data \eqref{problem data}.
\end{lemma}

Using the extension operator~$P_\mathcal{O}$, we deduce the following corollary from Lemma~\ref{Lemma tilde Lambda multiplicator properties}($1^\circ$).

\begin{corollary}
\label{Corollary tilde Lambda eps multiplier}
Under the assumptions of Lemma~\textnormal{\ref{Lemma tilde Lambda multiplicator properties}}, the operator~$[\widetilde{\Lambda}^\varepsilon]$
is continuous from $H^{l-1}(\mathcal{O};\mathbb{C}^n)$ to $L_2(\mathcal{O};\mathbb{C}^n)$ and
\begin{equation*}
\Vert [\widetilde{\Lambda}^\varepsilon ]\Vert _{H^{l-1}(\mathcal{O})\rightarrow L_2(\mathcal{O})}
\leqslant \widetilde{C}^{(0)}C_\mathcal{O}^{(l-1)}.
\end{equation*}
\end{corollary}

\section{Removal of the smoothing operator in the corrector \\ in the case of sufficiently smooth boundary}
\label{Section removing steklov operator}

\subsection{Proof of Lemma~\ref{Lemma K-K^0}}

Suppose that the assumptions of Lemma~\ref{Lemma K-K^0} are satisfied.
Let $\mathbf{u}_0$ be given by \eqref{u_0=}, where $\boldsymbol{\varphi}\in L_2(\mathcal{O};\mathbb{C}^n)$. We put
$$
\widetilde{\mathbf{u}}_0(\,\cdot\, ,t)=P_\mathcal{O}\mathbf{u}_0(\,\cdot\, ,t).
$$
According to \eqref{K_D(t,e)} and~\eqref{K_D^0(t;eps)}, we have
\begin{align}
\label{K_D phi =}
&\mathcal{K}_D(t;\varepsilon)\boldsymbol{\varphi}=\bigl(\Lambda ^\varepsilon S_\varepsilon b(\mathbf{D})+\widetilde{\Lambda}^\varepsilon S_\varepsilon \bigr)\widetilde{\mathbf{u}}_0(\,\cdot\, ,t),
\\
\label{K_D^0 phi =}
&\mathcal{K}_D^0(t;\varepsilon)\boldsymbol{\varphi}=\bigl(\Lambda ^\varepsilon b(\mathbf{D})+\widetilde{\Lambda}^\varepsilon  \bigr)\mathbf{u}_0(\,\cdot\, ,t).
\end{align}
We need to estimate the following value
\begin{equation}
\label{rem. S_eps 1}
\begin{split}
\Vert \mathcal{K}_D(t;\varepsilon)\boldsymbol{\varphi}-\mathcal{K}_D^0(t;\varepsilon)\boldsymbol{\varphi}\Vert _{H^1(\mathcal{O})}
&\leqslant
\Vert \Lambda ^\varepsilon \bigl( (S_\varepsilon -I)b(\mathbf{D})\widetilde{\mathbf{u}}_0\bigr)(\,\cdot\, ,t)\Vert _{H^1(\mathbb{R}^d)}
\\
&+
\Vert \widetilde{\Lambda}^\varepsilon \bigl( (S_\varepsilon -I)\widetilde{\mathbf{u}}_0\bigr)(\,\cdot\,,t)\Vert _{H^1(\mathbb{R}^d)}.
\end{split}
\end{equation}
Under the above assumptions, by Lemma~\ref{Lemma exp tilde B_D^0 L2-Hq}, we have $\mathbf{u}_0\in H^{l+1}(\mathcal{O};\mathbb{C}^n)$, whence $\widetilde{\mathbf{u}}_0\in H^{l+1}(\mathbb{R}^d;\mathbb{C}^n)$. This gives us possibility to apply Lemma~\ref{Lemma Lambda multiplicator properties}($2^\circ$) to estimate the first summand in the right-hand side
of~\eqref{rem. S_eps 1}:
\begin{equation}
\label{rem. S_eps 2}
\begin{split}
\Vert \Lambda ^\varepsilon \bigl( (S_\varepsilon -I)b(\mathbf{D})\widetilde{\mathbf{u}}_0\bigr)(\,\cdot\,,t)\Vert _{H^1(\mathbb{R}^d)}
&\leqslant
C^{(1)}\varepsilon ^{-1}\Vert \bigl((S_\varepsilon -I)b(\mathbf{D})\widetilde{\mathbf{u}}_0\bigr)(\,\cdot\, ,t)\Vert _{L_2(\mathbb{R}^d)}
\\
&+
C^{(2)}\Vert \bigl( (S_\varepsilon -I)b(\mathbf{D})\widetilde{\mathbf{u}}_0\bigr)(\,\cdot\, ,t)\Vert _{H^l(\mathbb{R}^d)},
\end{split}
\end{equation}
where $l=d/2$. The first term in the right-hand side of \eqref{rem. S_eps 2} is estimated with the help of Proposition~\ref{Proposition S__eps - I} and relations~\eqref{<b^*b<}, \eqref{f_0<=}, \eqref{PO}, \eqref{u_0=}, and \eqref{exp tilde B_D^0 L2-H2}:
\begin{equation}
\label{rem. S_eps 3}
\begin{split}
\varepsilon ^{-1}\Vert &\bigl( (S_\varepsilon -I)b(\mathbf{D})\widetilde{\mathbf{u}}_0\bigr)(\,\cdot\,,t)\Vert _{L_2(\mathbb{R}^d)}
\leqslant r_1\Vert \mathbf{D}b(\mathbf{D})\widetilde{\mathbf{u}}_0(\,\cdot\, ,t)\Vert _{L_2(\mathbb{R}^d)}
\\
&\leqslant r_1\alpha _1^{1/2}C_\mathcal{O}^{(2)}\Vert \mathbf{u}_0(\,\cdot\, ,t)\Vert _{H^2(\mathcal{O})}
\leqslant
C^{(3)}t^{-1}e^{-c_\flat t/2}\Vert \boldsymbol{\varphi}\Vert _{L_2(\mathcal{O})},
\end{split}
\end{equation}
where $C^{(3)}:=r_1\alpha _1^{1/2}C_\mathcal{O}^{(2)}\widetilde{c}\Vert f\Vert _{L_\infty}$.
To estimate the second term in the right-hand side of~\eqref{rem. S_eps 2}, we apply \eqref{S_eps <= 1} and \eqref{<b^*b<}:
\begin{equation}
\label{rem. S_eps 4}
\begin{split}
\Vert \bigl((S_\varepsilon -I)b(\mathbf{D})\widetilde{\mathbf{u}}_0\bigr)(\,\cdot\,,t)\Vert _{H^l(\mathbb{R}^d)}
&\leqslant 2\alpha _1^{1/2}\Vert \widetilde{\mathbf{u}}_0(\,\cdot\, ,t)\Vert _{H^{l+1}(\mathbb{R}^d)}.
\end{split}
\end{equation}
By \eqref{f_0<=}, \eqref{PO}, \eqref{u_0=}, and Lemma~\ref{Lemma exp tilde B_D^0 L2-Hq},
\begin{equation}
\label{rem. S_eps 5}
\begin{split}
\Vert \widetilde{\mathbf{u}}_0(\,\cdot\, ,t)\Vert _{H^{l+1}(\mathbb{R}^d)}
&\leqslant C_\mathcal{O}^{(l+1)}\widehat{\mathrm{C}}_{l+1}\Vert f\Vert ^2_{L_\infty}t^{-(l+1)/2}e^{-c_\flat t/2}\Vert \boldsymbol{\varphi}\Vert _{L_2(\mathcal{O})}.
\end{split}
\end{equation}
From \eqref{rem. S_eps 4} and \eqref{rem. S_eps 5} it follows that
\begin{equation}
\label{rem. S_eps 6}
\Vert \bigl( (S_\varepsilon -I)b(\mathbf{D})\widetilde{\mathbf{u}}_0\bigr) (\,\cdot\,,t)\Vert _{H^l(\mathbb{R}^d)}
\leqslant C^{(4)}t^{-(l+1)/2}e^{-c_\flat t/2}\Vert \boldsymbol{\varphi}\Vert _{L_2(\mathcal{O})},
\end{equation}
where $C^{(4)}:=2\alpha _1^{1/2}C_\mathcal{O}^{(l+1)}\widehat{\mathrm{C}}_{l+1}\Vert f\Vert ^2_{L_\infty}$.

Now we estimate the second term in the right-hand side of~\eqref{rem. S_eps 1}. By Lemma~\ref{Lemma tilde Lambda multiplicator properties}($2^\circ$),
\begin{equation}
\label{rem. S_eps 7}
\begin{split}
\Vert \widetilde{\Lambda}^\varepsilon \bigl((S_\varepsilon -I)\widetilde{\mathbf{u}}_0\bigr)(\,\cdot\, ,t)\Vert _{H^1(\mathbb{R}^d)}
&\leqslant
\widetilde{C}^{(1)}\varepsilon ^{-1}\Vert (S_\varepsilon -I)\widetilde{\mathbf{u}}_0(\,\cdot\,,t)\Vert _{H^1(\mathbb{R}^d)}
\\
&+\widetilde{C}^{(2)}\Vert (S_\varepsilon -I)\widetilde{\mathbf{u}}_0(\,\cdot\,,t)\Vert _{H^l(\mathbb{R}^d)},\quad l=d/2.
\end{split}
\end{equation}
The first summand in the right-hand side of~\eqref{rem. S_eps 7} is estimated by using Proposition~\ref{Proposition S__eps - I} and relations~\eqref{f_0<=}, \eqref{PO}, \eqref{u_0=}, \eqref{exp tilde B_D^0 L2-H2}:
\begin{equation}
\label{rem. S_eps 7a}
\begin{split}
\varepsilon ^{-1}\Vert (S_\varepsilon -I)\widetilde{\mathbf{u}}_0(\,\cdot\,,t)\Vert _{H^1(\mathbb{R}^d)}
\leqslant r_1C_\mathcal{O}^{(2)}\Vert \mathbf{u}_0(\,\cdot\, ,t)\Vert _{H^2(\mathcal{O})}
\leqslant C ^{(5)}t^{-1}e^{-c_\flat t/2}\Vert \boldsymbol{\varphi}\Vert _{L_2(\mathcal{O})};
\\
C^{(5)}:=r_1C_\mathcal{O}^{(2)}\widetilde{c}\Vert f\Vert _{L_\infty}.
\end{split}
\end{equation}
The second summand in \eqref{rem. S_eps 7} is estimated with the help of \eqref{S_eps <= 1} and \eqref{rem. S_eps 5}:
\begin{equation}
\label{rem. S_eps 7b}
\begin{split}
\Vert (S_\varepsilon -I)\widetilde{\mathbf{u}}_0(\,\cdot\,,t)\Vert _{H^l(\mathbb{R}^d)}
&\leqslant 2\Vert \widetilde{\mathbf{u}}_0(\,\cdot\, ,t)\Vert _{H^l(\mathbb{R}^d)}
\leqslant  2\Vert \widetilde{\mathbf{u}}_0(\,\cdot\, ,t)\Vert _{H^{l+1}(\mathbb{R}^d)}
\\
&\leqslant C^{(6)}t^{-(l+1)/2}e^{-c_\flat t/2}\Vert \boldsymbol{\varphi}\Vert _{L_2(\mathcal{O})};
\quad
C^{(6)}:=2C_\mathcal{O}^{(l+1)}\widehat{\mathrm{C}}_{l+1}\Vert f\Vert ^2_{L_\infty}.
\end{split}
\end{equation}

As a result, relations~\eqref{rem. S_eps 1}--\eqref{rem. S_eps 3} and \eqref{rem. S_eps 6}--\eqref{rem. S_eps 7b} imply that
\begin{equation*}
\begin{split}
\Vert \mathcal{K}_D(t;\varepsilon)\boldsymbol{\varphi}-\mathcal{K}_D^0(t;\varepsilon)\boldsymbol{\varphi}\Vert _{H^1(\mathcal{O})}
\leqslant (C^{(7)}t^{-1}+C^{(8)}t^{-(l+1)/2})e^{-c_\flat t/2}\Vert \boldsymbol{\varphi}\Vert _{L_2(\mathcal{O})},
\end{split}
\end{equation*}
where $l=d/2$,
$C^{(7)}:=C^{(1)}C^{(3)}+\widetilde{C}^{(1)}C^{(5)}$, and $C^{(8)}:=C^{(2)}C^{(4)}+\widetilde{C}^{(2)}C^{(6)}$.
This proves estimate \eqref{lemma K-K^0} with the constant $\widehat{\mathcal{C}}_d:=\max \lbrace C^{(7)};C^{(8)}\rbrace$. \qed

\subsection{Proof of Theorem~\ref{Theorem smooth boundary}}
Inequality \eqref{Th smooth domain 1} directly follows from \eqref{Th_exp_korrector} and \eqref{lemma K-K^0}.
Herewith, $\mathcal{C}_d:=2(\widehat{\mathcal{C}}_d+C_{16})$. Above, we took  into account that for $t>1$
the term $\varepsilon t^{-1}$ does not exceed $\varepsilon ^{1/2}t^{-3/4}$, and for $t\leqslant 1$ it does not
exceed $\varepsilon t^{-d/4-1/2}$ since $d\geqslant 3$.

Let us check \eqref{Th smooth domain 2}. By \eqref{Th smooth domain 1} and \eqref{b_l <=},
\begin{equation}
\label{rem. S_eps 9}
\begin{split}
\Bigl\Vert & g^\varepsilon b(\mathbf{D})\Bigl(
f^\varepsilon e^{-\widetilde{B}_{D,\varepsilon}t}(f^\varepsilon)^*-f_0e^{-\widetilde{B}_D^0t}f_0
-\varepsilon \bigl(\Lambda ^\varepsilon b(\mathbf{D})+\widetilde{\Lambda}^\varepsilon\bigr)f_0e^{-\widetilde{B}_D^0t}f_0
\Bigr)\Bigr\Vert _{L_2\rightarrow L_2}
\\
&\leqslant \Vert g\Vert _{L_\infty}(d\alpha _1)^{1/2}\mathcal{C}_d(\varepsilon ^{1/2}t^{-3/4}+\varepsilon t^{-d/4-1/2})e^{-c_\flat t/2}.
\end{split}
\end{equation}

We have
\begin{equation}
\label{rem. S_eps 10}
\begin{split}
&\varepsilon g^\varepsilon b(\mathbf{D})\bigl(\Lambda^\varepsilon b(\mathbf{D})+\widetilde{\Lambda}^\varepsilon\bigr)f_0e^{-\widetilde{B}_D^0t}f_0
=g^\varepsilon \Bigl( (b(\mathbf{D})\Lambda)^\varepsilon +\bigl(b(\mathbf{D})\widetilde{\Lambda}\bigr)^\varepsilon \Bigr)f_0e^{-\widetilde{B}_D^0t}f_0
\\
&+
\varepsilon\sum _{k,j=1}^d g^\varepsilon b_k \Lambda ^\varepsilon b_j D_kD_jf_0e^{-\widetilde{B}_D^0t}f_0
+\varepsilon \sum _{j=1}^d g^\varepsilon b_j \widetilde{\Lambda}^\varepsilon D_j f_0 e^{-\widetilde{B}_D^0 t}f_0.
\end{split}
\end{equation}
The norm of the second summand in the right-hand side of \eqref{rem. S_eps 10} is estimated with the help of \eqref{b_l <=}, \eqref{f_0<=}, Lemma~\ref{Lemma exp tilde B_D^0 L2-Hq}, and Corollary~\ref{Corollary Lambda eps multiplier}:
\begin{equation}
\label{rem. S_eps 11}
\begin{split}
\varepsilon\Bigl\Vert & \sum _{k,j=1}^d g^\varepsilon b_k \Lambda ^\varepsilon b_j D_kD_jf_0e^{-\widetilde{B}_D^0t}f_0\Bigr\Vert _{L_2(\mathcal{O})\rightarrow L_2(\mathcal{O})}
\leqslant
C^{(9)}\varepsilon t^{-(l+1)/2}e^{-c_\flat t/2},
\end{split}
\end{equation}
$l=d/2$; $C^{(9)}:=\alpha_1dC^{(0)}C_\mathcal{O}^{(l-1)}\widehat{\mathrm{C}}_{l+1}\Vert g\Vert _{L_\infty}\Vert f\Vert ^2_{L_\infty}$.
The third summand in the right-hand side of~\eqref{rem. S_eps 10} is estimated by using \eqref{b_l <=}, \eqref{f_0<=}, Lemma~\ref{Lemma exp tilde B_D^0 L2-Hq}, and Corollary~\ref{Corollary tilde Lambda eps multiplier}:
\begin{equation}
\label{rem. S_eps 12}
\begin{split}
\varepsilon \Bigl\Vert    \sum _{j=1}^d g^\varepsilon b_j \widetilde{\Lambda}^\varepsilon D_j f_0 e^{-\widetilde{B}_D^0 t}f_0\Bigr\Vert _{L_2(\mathcal{O})\rightarrow L_2(\mathcal{O})}
\leqslant C^{(10)}\varepsilon t^{-(l+1)/2}e^{-c_\flat t/2},
\end{split}
\end{equation}
where $l=d/2$ and
$$
C^{(10)}:=(d\alpha_1)^{1/2}\widetilde{C}^{(0)}C_\mathcal{O}^{(l-1)}\widehat{\mathrm{C}}_{l+1}\Vert g\Vert _{L_\infty}\Vert f\Vert ^2_{L_\infty}.
$$
As a result, relations \eqref{rem. S_eps 9}--\eqref{rem. S_eps 12} imply inequality \eqref{Th smooth domain 2} with the constant
 \begin{equation*}
 \widetilde{\mathcal{C}}_d:=\Vert g\Vert _{L_\infty}(d\alpha _1)^{1/2}\mathcal{C}_d +C^{(9)}+C^{(10)}.\phantom{ccccccccc}\qed
\end{equation*}

\section{Removal of the smoothing operator in the corrector\\ in a strictly interior subdomain}
\label{Section removing S-eps in strictly interior subdomain}
\subsection{One property of the operator $S_\eps$}
Now we proceed to estimates in a strictly interior subdomain.
We start with one simple property of the operator~$S_\eps$.

Let ${\mathcal O}'$ be a strictly interior subdomain of the domain ${\mathcal O}$, and let $\delta$ be given by \eqref{delta= 1.62a}.
Denote
$$
{\mathcal O}'':=\{ {\mathbf x}\in {\mathcal O}:\ {\rm dist}\,\{{\mathbf x};
\partial {\mathcal O}\} > \delta/2\},\quad
{\mathcal O}''':=\{ {\mathbf x}\in {\mathcal O}:\ {\rm dist}\,\{{\mathbf x};
\partial {\mathcal O}\} > \delta/4\}.
$$

\begin{lemma}
\label{Lemma one property of S-eps}
Let $S_\eps$ be the operator \textnormal{\eqref{S_eps}}. Let $2r_1=\mathrm{diam}\,\Omega$.
Suppose that $\mathbf{v}\in L_2(\mathbb{R}^d;\mathbb{C}^m)$ and ${\mathbf v} \in H^{\sigma}({\mathcal O}'''; {\mathbb C}^m)$ with some
$\sigma \in {\mathbb Z}_+$. Then for $0< \eps \leqslant (4r_1)^{-1}\delta$ we have $S_\varepsilon\mathbf{v}\in H^\sigma (\mathcal{O}'';\mathbb{C}^m)$ and
\begin{equation*}
\| S_\eps  {\mathbf v}\|_{H^\sigma({\mathcal O}'')} \leqslant
\|  {\mathbf v}\|_{H^\sigma({\mathcal O}''')}.
\end{equation*}
\end{lemma}

\begin{proof}
According to \eqref{S_eps},
\begin{equation}
\label{10.2}
\begin{split}
\| S_\eps  {\mathbf v}\|^2_{H^\sigma({\mathcal O}'')}
=
|\Omega|^{-2} \sum_{|\alpha| \leqslant \sigma} \int\limits_{{\mathcal O}''} d{\mathbf x}\bigg| \int\limits_{\Omega} {\mathbf D}^\alpha
{\mathbf v}({\mathbf x} - \eps {\mathbf z}) \, d{\mathbf z}\bigg|^2
\leqslant
|\Omega|^{-1} \sum_{|\alpha| \leqslant \sigma} \int\limits_{{\mathcal O}''} d{\mathbf x} \int\limits_{\Omega} \left|{\mathbf D}^\alpha
{\mathbf v}({\mathbf x} - \eps {\mathbf z})\right|^2 d{\mathbf z}.
\end{split}
\end{equation}
Since $0< \eps r_1 \leqslant \delta/4$, for ${\mathbf x} \in {\mathcal O}''$ and ${\mathbf z}\in \Omega$
we have ${\mathbf x} - \eps {\mathbf z}\in {\mathcal O}'''$.
Hence, changing the order of integration in~\eqref{10.2}, we obtain the required estimate.
\end{proof}

\subsection{A cut-off function $\chi({\mathbf x})$}
We fix a smooth cut-off function $\chi({\mathbf x})$ such that
\begin{equation}
\label{chi}
\begin{split}
&\chi \in C^\infty_0({\mathbb R}^d),\quad 0 \leqslant \chi({\mathbf x}) \leqslant 1; \quad \chi({\mathbf x})=1,\ {\mathbf x}\in {\mathcal O}' ;
\\
&{\rm supp}\,\chi \subset   {\mathcal O}'';\quad |{\mathbf D}^\alpha \chi({\mathbf x})| \leqslant \kappa_\sigma \delta^{-\sigma}, \quad |\alpha|=\sigma,\quad
\sigma\in {\mathbb N}.
\end{split}
\end{equation}
The constants $\kappa_\sigma$ depend only on $d$, $\sigma$, and the domain ${\mathcal O}$.

\begin{lemma}
\label{Lemma chi}
Suppose that $\chi({\mathbf x})$ is a cut-off function satisfying conditions \textnormal{\eqref{chi}}.
Let $k\in {\mathbb Z}_+$.

\noindent
$1^\circ$. For any function ${\mathbf v} \in H^{k}({\mathbb R}^d; {\mathbb C}^m)$ we have
\begin{equation}
\label{10.4}
\| \chi {\mathbf v}\|_{H^k({\mathbb R}^d)}
\leqslant
C_k^{(11)} \sum_{j=0}^k \delta^{-(k-j)} \|{\mathbf v}\|_{H^{j}({\mathcal O}'')}.
\end{equation}

\noindent
$2^\circ$. For any function ${\mathbf v} \in H^{k+1}({\mathbb R}^d; {\mathbb C}^m)$ we have
\begin{equation}
\label{10.5}
\begin{split}
\| \chi {\mathbf v}\|_{H^{k+1/2}({\mathbb R}^d)}
\leqslant
C_{k+1/2}^{(11)} \biggl( \sum_{j=0}^{k+1} \delta^{-(k+1 -j)} \|{\mathbf v}\|_{H^{j}({\mathcal O}'')} \biggr)^{1/2}
 \biggl( \sum_{i =0}^{k} \delta^{-(k - i)} \|{\mathbf v}\|_{H^{i}({\mathcal O}'')} \biggr)^{1/2}.
\end{split}
\end{equation}

The constants $C_k^{(11)}$ and $C_{k+1/2}^{(11)}$ depend on $d$, $k$, and the domain~${\mathcal O}$.
\end{lemma}

\begin{proof}
Inequality \eqref{10.4} follows from the Leibniz formula for the derivatives of the product
$\chi {\mathbf v}$ and from the estimates for the derivatives of $\chi$ (see \eqref{chi}).
To check~\eqref{10.5}, we should also take into account that
$$
\| {\mathbf w} \|_{H^{k+1/2}({\mathbb R}^d)}^2 \leqslant
\| {\mathbf w} \|_{H^{k+1}({\mathbb R}^d)} \|  {\mathbf w} \|_{H^{k}({\mathbb R}^d)},\quad
{\mathbf w} \in H^{k+1}({\mathbb R}^d; {\mathbb C}^m).
$$
\end{proof}

\subsection{Proof of Lemma~\ref{Lemma K-K0 strictly interior}}
  Suppose that the assumptions of Lemma~\ref{Lemma K-K0 strictly interior} are satisfied.
  Let $\mathbf{u}_0$ be given by \eqref{u_0=} with $\boldsymbol{\varphi}\in L_2(\mathcal{O};\mathbb{C}^n)$.
  According to \eqref{f_0<=} and  \eqref{exp tilde B_D^0 L2-H1}, \eqref{exp tilde B_D^0 L2-H2}, we have
  \begin{align}
\label{rem. S_eps 14}
&\Vert \mathbf{D}\mathbf{u}_0(\,\cdot\,,t)\Vert _{L_2(\mathcal{O})}
\leqslant \Vert \mathbf{u}_0(\,\cdot\, ,t)\Vert _{H^1(\mathcal{O})}
\leqslant c_3\Vert f\Vert _{L_\infty}t^{-1/2}e^{-c_\flat t/2}\Vert \boldsymbol{\varphi}\Vert _{L_2(\mathcal{O})},
\\
\label{rem. S_eps 14a}
&\Vert \mathbf{D}\mathbf{u}_0(\,\cdot\,,t)\Vert _{H^1(\mathcal{O})}
\leqslant \Vert \mathbf{u}_0(\,\cdot\, ,t)\Vert _{H^2(\mathcal{O})}
\leqslant \widetilde{c}\Vert f\Vert _{L_\infty}t^{-1}e^{-c_\flat t/2}\Vert \boldsymbol{\varphi}\Vert _{L_2(\mathcal{O})}.
\end{align}
Let $\widetilde{\mathbf{u}}_0=P_\mathcal{O}\mathbf{u}_0$. Relations \eqref{K_D phi =} and \eqref{K_D^0 phi =} remain valid.
We need to estimate the following value:
\begin{equation}
\label{rem. S_eps 15}
\begin{split}
\Vert \mathcal{K}_D(t;\varepsilon)\boldsymbol{\varphi}-\mathcal{K}_D^0(t;\varepsilon)\boldsymbol{\varphi}\Vert _{H^1(\mathcal{O}')}
&\leqslant
\Vert \Lambda ^\varepsilon \chi \bigl( (S_\varepsilon -I)b(\mathbf{D})\widetilde{\mathbf{u}}_0\bigr)(\,\cdot\,,t)\Vert _{H^1(\mathbb{R}^d)}
\\
&+
\Vert \widetilde{\Lambda}^\varepsilon \chi \bigl( (S_\varepsilon -I)\widetilde{\mathbf{u}}_0\bigr)(\,\cdot\, ,t)\Vert _{H^1(\mathbb{R}^d)}.
\end{split}
\end{equation}
Recall (cf. discussion in Subsection~\ref{Subsection strictly interior}) that $\mathbf{u}_0(\,\cdot\, ,t)\in H^\sigma (\mathcal{O}''';\mathbb{C}^n)$ for any $\sigma\in\mathbb{Z}_+$.
Then the function $\widetilde{\mathbf{u}}_0(\,\cdot\,,t)$ satisfies the assumptions of Lemma~\ref{Lemma one property of S-eps} for any $\sigma\in\mathbb{Z}_+$. Hence, $(S_\varepsilon\widetilde{\mathbf{u}}_0)(\,\cdot\, ,t)\in H^\sigma (\mathcal{O}'';\mathbb{C}^n)$ for ${0<\varepsilon\leqslant (4r_1)^{-1}\delta}$.
Then we can apply Lemma~\ref{Lemma Lambda multiplicator properties}($2^\circ$) to estimate the first summand in the right-hand side of~\eqref{rem. S_eps 15}:
\begin{equation}
\label{rem. S_eps 16}
\begin{split}
\Vert \Lambda ^\varepsilon \chi \bigl((S_\varepsilon -I)b(\mathbf{D})\widetilde{\mathbf{u}}_0\bigr)(\,\cdot\,,t)\Vert _{H^1(\mathbb{R}^d)}
&\leqslant
C^{(1)}\varepsilon ^{-1}\Vert \chi \bigl( (S_\varepsilon -I)b(\mathbf{D})\widetilde{\mathbf{u}}_0\bigr)(\,\cdot\,,t)\Vert _{L_2(\mathbb{R}^d)}
\\&
+
C^{(2)}\Vert \chi \bigl( (S_\varepsilon -I)b(\mathbf{D})\widetilde{\mathbf{u}}_0\bigr)(\,\cdot\,,t)\Vert _{H^l(\mathbb{R}^d)},
\end{split}
\end{equation}
$l=d/2$.
The first term in the right-hand side of \eqref{rem. S_eps 16}
is estimated by using inequality~\eqref{rem. S_eps 3} (which holds without additional smoothness assumption on~$\partial\mathcal{O}$):
\begin{equation}
\label{rem. S_eps 17}
\begin{split}
\varepsilon ^{-1}\Vert \chi \bigl((S_\varepsilon -I)b(\mathbf{D})\widetilde{\mathbf{u}}_0\bigr) (\,\cdot\, ,t)\Vert _{L_2(\mathbb{R}^d)}
&\leqslant
C^{(3)}t^{-1}e^{-c_\flat t/2}\Vert \boldsymbol{\varphi}\Vert _{L_2(\mathcal{O})}.
\end{split}
\end{equation}
Now, we consider the second summand in the right-hand side of~\eqref{rem. S_eps 16}. Obviously,
\begin{equation}
\label{rem. S_eps 18}
\begin{split}
\Vert \chi \bigl( (S_\varepsilon -I)b(\mathbf{D})\widetilde{\mathbf{u}}_0\bigr)(\,\cdot\, ,t)\Vert _{H^l(\mathbb{R}^d)}
&\leqslant
\Vert \chi  (S_\varepsilon b(\mathbf{D})\widetilde{\mathbf{u}}_0)(\,\cdot\, ,t)\Vert _{H^l(\mathbb{R}^d)}
\\
&+
\Vert \chi b(\mathbf{D})\widetilde{\mathbf{u}}_0(\,\cdot\, ,t)\Vert _{H^l(\mathbb{R}^d)}.
\end{split}
\end{equation}
To estimate the second term in the right-hand side of~\eqref{rem. S_eps 18}, we apply Lemma~\ref{Lemma chi} and \eqref{b_l <=}.
If $l=d/2$ is integer (i.~e., the dimension $d$ is even), we have
\begin{equation}
\label{rem. S_eps 19}
\Vert \chi b(\mathbf{D})\widetilde{\mathbf{u}}_0(\,\cdot\, ,t)\Vert _{H^l(\mathbb{R}^d)}
\leqslant C_l^{(11)}(d\alpha _1)^{1/2}
\sum _{j=0}^l\delta^{-(l-j)}\Vert \mathbf{D}\mathbf{u}_0(\,\cdot\, ,t)\Vert _{H^j(\mathcal{O}'')}.
\end{equation}
If $l=d/2=k+1/2$, then
\begin{equation}
\label{rem. S_eps 20}
\begin{split}
\Vert \chi b(\mathbf{D})\widetilde{\mathbf{u}}_0(\,\cdot\, ,t)\Vert _{H^l(\mathbb{R}^d)}
&\leqslant
C_l^{(11)}\!(d\alpha _1)^{1/2}
\Bigl(
\sum _{j=0}^{k+1}\delta ^{-(k+1-j)}\Vert \mathbf{D}\mathbf{u}_0(\,\cdot\, ,t)\Vert _{H^j(\mathcal{O}'')}
\Bigr)^{1/2}
\\
&\qquad\qquad\qquad\times
\Bigl(
\sum _{\sigma=0}^k \delta ^{-(k-\sigma)}\Vert\mathbf{D}\mathbf{u}_0(\,\cdot\, ,t)\Vert _{H^\sigma (\mathcal{O}'')}\Bigr)^{1/2}.
\end{split}
\end{equation}
The norms of $\mathbf{D}\mathbf{u}_0(\,\cdot\, ,t)$ in $L_2(\mathcal{O};\mathbb{C}^n)$ and in $H^1(\mathcal{O};\mathbb{C}^n)$ are estimated in \eqref{rem. S_eps 14} and \eqref{rem. S_eps 14a}. By \eqref{f_0<=}, \eqref{u_0=}, and \eqref{exp tilde B_D^0 L2-H-sigma strictly interior} (with $\mathcal{O}'$ replaced by $\mathcal{O}''$),
\begin{equation}
\label{rem. S_eps 21}
\begin{split}
\Vert \mathbf{D}\mathbf{u}_0 (\,\cdot\, ,t)\Vert _{H^\sigma (\mathcal{O}'')}
\!\leqslant\! \mathrm{C}'_{\sigma +1}\Vert f\Vert ^2_{L_\infty} 2^\sigma t^{-1/2}(\delta ^{-2}\!\!+t^{-1}\!)^{\sigma /2}e^{-c_\flat t/2}\Vert \boldsymbol{\varphi}\Vert _{L_2(\mathcal{O})},
\end{split}
\end{equation}
$\sigma\geqslant 2$.
Using \eqref{rem. S_eps 14}, \eqref{rem. S_eps 14a}, and \eqref{rem. S_eps 19}--\eqref{rem. S_eps 21}, we arrive at the inequality
\begin{equation}
\label{rem. S_eps 22}
\Vert \chi b(\mathbf{D})\widetilde{\mathbf{u}}_0(\,\cdot\, ,t)\Vert _{H^l(\mathbb{R}^d)}
\leqslant
C^{(12)}t^{-1/2}(\delta ^{-2}+t^{-1})^{d/4}e^{-c_\flat t/2}\Vert \boldsymbol{\varphi}\Vert _{L_2(\mathcal{O})}.
\end{equation}
The constant $C^{(12)}$ depends only on the problem data \eqref{problem data}.

To estimate the first term in the right-hand side of \eqref{rem. S_eps 18}, we apply Lemmas \ref{Lemma one property of S-eps} and~\ref{Lemma chi}.
Assume that $0<\varepsilon\leqslant (4 r_1)^{-1}\delta$. By \eqref{b_l <=}, in the case of integer $l$, we have
\begin{equation}
\label{rem. S_eps 23}
\begin{split}
\Vert \chi (S_\varepsilon b(\mathbf{D})\widetilde{\mathbf{u}}_0)(\,\cdot\,,t)\Vert _{H^l(\mathbb{R}^d)}
\leqslant C_l^{(11)}(d\alpha _1)^{1/2}\sum _{\sigma =0}^l\delta ^{-(l-\sigma)}\Vert \mathbf{D}\mathbf{u}_0(\,\cdot\, ,t)\Vert _{H^\sigma(\mathcal{O}''')}.
\end{split}
\end{equation}
The norms of $\mathbf{D}\mathbf{u}_0(\,\cdot\, ,t)$ in $L_2(\mathcal{O};\mathbb{C}^n)$ and in $H^1(\mathcal{O};\mathbb{C}^n)$ are estimated in \eqref{rem. S_eps 14} and \eqref{rem. S_eps 14a}. By \eqref{f_0<=}, \eqref{u_0=} and \eqref{exp tilde B_D^0 L2-H-sigma strictly interior} (with $\mathcal{O}'$ replaced by $\mathcal{O}'''$)
\begin{equation}
\label{rem. S_eps 24}
\Vert \mathbf{D}\mathbf{u}_0(\,\cdot\, ,t)\Vert _{H^\sigma (\mathcal{O}''')}
\!\leqslant\!
\mathrm{C}'_{\sigma+1}\Vert f\Vert^2_{L_\infty}4^\sigma t^{-1/2}(\delta ^{-2}\!\!+t^{-1})^{\sigma/2}e^{-c_\flat t/2}\Vert \boldsymbol{\varphi}\Vert _{L_2(\mathcal{O})},
\end{equation}
$\sigma\geqslant 2$.
From \eqref{rem. S_eps 14}, \eqref{rem. S_eps 14a}, \eqref{rem. S_eps 23}, and \eqref{rem. S_eps 24} it follows that
\begin{equation}
\label{rem. S_eps 25}
\Vert \chi (S_\varepsilon b(\mathbf{D})\widetilde{\mathbf{u}}_0)(\,\cdot\, ,t)\Vert _{H^l(\mathbb{R}^d)}
\leqslant
C^{(13)}t^{-1/2}(\delta ^{-2}+t^{-1})^{d/4}e^{-c_\flat t/2}\Vert \boldsymbol{\varphi}\Vert _{L_2(\mathcal{O})}.
\end{equation}
The constant $C^{(13)}$ depends only on the problem data \eqref{problem data}.
Estimate \eqref{rem. S_eps 25} in the case of half-integer $l$ is checked similarly.
Combining \eqref{rem. S_eps 16}--\eqref{rem. S_eps 18}, \eqref{rem. S_eps 22}, and \eqref{rem. S_eps 25}, we estimate the first summand in the right-hand side of \eqref{rem. S_eps 15}:
\begin{equation}
\label{first summand in K-K0 str. int.}
\begin{split}
\Vert \Lambda ^\varepsilon \chi \bigl( (S_\varepsilon -I)b(\mathbf{D})\widetilde{\mathbf{u}}_0\bigr) (\,\cdot\, ,t)\Vert _{H^1(\mathbb{R}^d)}
\leqslant C^{(14)} \bigl( t^{-1}+t^{-1/2}(\delta ^{-2}+t^{-1})^{d/4}\bigr)e^{-c_\flat t/2}\Vert \boldsymbol{\varphi}\Vert _{L_2(\mathcal{O})}.
\end{split}
\end{equation}
Here $C^{(14)}:=\max\lbrace C^{(1)}C^{(3)};C^{(2)}(C^{(12)}+C^{(13)})\rbrace$.

The second summand in the right-hand side of \eqref{rem. S_eps 15} is estimated by Lemma~\ref{Lemma tilde Lambda multiplicator properties}($2^\circ$):
\begin{equation}
\label{rem. S_eps 26}
\begin{split}
\Vert \widetilde{\Lambda}^\varepsilon \chi \bigl((S_\varepsilon -I)\widetilde{\mathbf{u}}_0\bigr)(\,\cdot\, ,t)\Vert _{H^1(\mathbb{R}^d)}
&\leqslant
\widetilde{C}^{(1)}\varepsilon ^{-1}\Vert \chi \bigl((S_\varepsilon -I)\widetilde{\mathbf{u}}_0\bigr)(\,\cdot\,,t)\Vert _{H^1(\mathbb{R}^d)}
\\
&+\widetilde{C}^{(2)}\Vert \chi \bigl((S_\varepsilon -I)\widetilde{\mathbf{u}}_0\bigr)(\,\cdot\, ,t)\Vert _{H^l(\mathbb{R}^d)},
\end{split}
\end{equation}
where $l=d/2$.
 To estimate the first summand in the right-hand side of \eqref{rem. S_eps 26}, we use \eqref{chi} and inequality \eqref{rem. S_eps 7a} (which holds without extra smoothness assumption
 on the boundary):
\begin{equation*}
\begin{split}
\varepsilon^{-1}\Vert & \chi \bigl((S_\varepsilon -I)\widetilde{\mathbf{u}}_0\bigr)(\,\cdot\,,t)\Vert _{H^1(\mathbb{R}^d)}\\
&\leqslant
\varepsilon ^{-1}\Vert \bigl( (S_\varepsilon -I)\widetilde{\mathbf{u}}_0\bigr)(\,\cdot\,,t)\Vert _{H^1(\mathbb{R}^d)}
+\varepsilon ^{-1}\Vert (\mathbf{D}\chi)\bigl( (S_\varepsilon -I)\widetilde{\mathbf{u}}_0\bigr)(\,\cdot\,,t)\Vert _{L_2(\mathbb{R}^d)}
\\
&\leqslant C^{(5)} t^{-1}e^{-c_\flat t/2}\Vert \boldsymbol{\varphi}\Vert _{L_2(\mathcal{O})}
+\varepsilon ^{-1}\kappa _1 \delta ^{-1}\Vert (S_\varepsilon -I)\widetilde{\mathbf{u}}_0 (\,\cdot\, ,t)\Vert _{L_2(\mathbb{R}^d)}
.
\end{split}
\end{equation*}
Combining this with Proposition~\ref{Proposition S__eps - I} and relations \eqref{f_0<=}, \eqref{PO}, \eqref{u_0=}, and \eqref{exp tilde B_D^0 L2-H1}, we obtain
\begin{equation}
\label{rem. S_eps 27}
\begin{split}
\varepsilon^{-1}\Vert \chi \bigl((S_\varepsilon -I)\widetilde{\mathbf{u}}_0\bigr)(\,\cdot\,,t)\Vert _{H^1(\mathbb{R}^d)}
\leqslant
C^{(15)}(\delta ^{-1}t^{-1/2}+t^{-1})e^{-c_\flat t/2}\Vert \boldsymbol{\varphi}\Vert _{L_2(\mathcal{O})},
\end{split}
\end{equation}
where
$C^{(15)}:=\max\lbrace C^{(5)};\kappa _1 r_1 C_\mathcal{O}^{(1)}c_3\Vert f\Vert _{L_\infty}\rbrace$.

If $l=d/2$ is integer, the second summand in the right-hand side of \eqref{rem. S_eps 26} is estimated by analogy with \eqref{rem. S_eps 23}:
\begin{equation}
\label{rem. S_eps 28}
\begin{split}
\Vert \chi \bigl((S_\varepsilon -I)\widetilde{\mathbf{u}}_0\bigr)(\,\cdot\, ,t)\Vert _{H^l(\mathbb{R}^d)}
\leqslant
2C_l^{(11)}\sum  _{\sigma =0}^l \delta ^{-(l-\sigma)}\Vert \mathbf{u}_0(\,\cdot\, ,t)\Vert _{H^\sigma (\mathcal{O}''')},
\end{split}
\end{equation}
$ 0<\varepsilon\leqslant (4 r_1)^{-1}\delta $.
The norms of $\mathbf{u}_0$ in $L_2(\mathcal{O};\mathbb{C}^n)$, $H^1(\mathcal{O};\mathbb{C}^n)$, and $H^2(\mathcal{O};\mathbb{C}^n)$ are estimated by Lemma~\ref{Lemma properties of operator exponential} and relations \eqref{f_0<=}, \eqref{u_0=}. For $\sigma\geqslant 3$ the norm $\Vert \mathbf{u}_0 (\,\cdot\, ,t)\Vert _{H^\sigma (\mathcal{O}''')}$
is estimated by using \eqref{exp tilde B_D^0 L2-H-sigma strictly interior} (with $\mathcal{O}'$ replaced by $\mathcal{O}'''$):
\begin{equation*}
\begin{split}
\Vert \mathbf{u}_0(\,\cdot\, ,t)\Vert _{H^\sigma (\mathcal{O}''')}
&\leqslant \mathrm{C}_{\sigma +1} '\Vert f\Vert^2_{L_\infty} 4^{\sigma }t^{-1/2}(\delta ^{-2}+t^{-1})^{\sigma /2}e^{-c_\flat t/2}\Vert \boldsymbol{\varphi}\Vert _{L_2(\mathcal{O})}.
\end{split}
\end{equation*}
Combining these arguments with \eqref{rem. S_eps 28}, we deduce that
\begin{equation}
\label{rem. S_eps 29}
\Vert \chi \bigl((S_\varepsilon -I)\widetilde{\mathbf{u}}_0\bigr)(\,\cdot\, ,t)\Vert _{H^l(\mathbb{R}^d)}
\leqslant
C^{(16)}t^{-1/2}(\delta ^{-2}+t^{-1})^{d/4}e^{-c_\flat t/2}\Vert \boldsymbol{\varphi}\Vert _{L_2(\mathcal{O})},
\end{equation}
with the constant $C^{(16)}$ depending only on the problem data \eqref{problem data}.
For the case of half-integer $l$, estimate \eqref{rem. S_eps 29} is checked similarly.
As a result, relations \eqref{rem. S_eps 26}, \eqref{rem. S_eps 27}, and \eqref{rem. S_eps 29} imply the following estimate for the second summand in the right-hand side
 of \eqref{rem. S_eps 15}:
\begin{equation*}
\begin{split}
\Vert \widetilde{\Lambda}^\varepsilon\chi \bigl((S_\varepsilon -I)\widetilde{\mathbf{u}}_0\bigr)(\,\cdot\,,t)\Vert _{H^1(\mathbb{R}^d)}
&\leqslant
\widetilde{C}^{(1)}C^{(15)}(\delta ^{-1}t^{-1/2}+t^{-1})e^{-c_\flat t/2}\Vert \boldsymbol{\varphi}\Vert _{L_2(\mathcal{O})}
\\
&+
\widetilde{C}^{(2)}C^{(16)}t^{-1/2}(\delta ^{-2}+t^{-1})^{d/4}e^{-c_\flat t/2}\Vert \boldsymbol{\varphi}\Vert _{L_2(\mathcal{O})}.
\end{split}
\end{equation*}
Together with \eqref{rem. S_eps 15} and \eqref{first summand in K-K0 str. int.}, this implies inequality \eqref{lm K-K0 strictly interior} with the constant
$
\mathrm{C}_d'':=
C^{(14)}+\widetilde{C}^{(1)}C^{(15)}+ \widetilde{C}^{(2)}C^{(16)}
$.
We have taken into account that the term $\delta ^{-1}t^{-1/2}$ does not exceed $t^{-1/2}(\delta ^{-2}+t^{-1})^{d/4}$.
\qed

\subsection{Proof of Theorem \ref{Theorem O' no S_eps no conditions on Lambda's}}
Inequality \eqref{Th O' no S_eps no conditions on Lambda's} follows directly from \eqref{Th_exp_korrector strictly interior} and \eqref{lm K-K0 strictly interior}. Herewith,
$\mathrm{C}_d :=\max\lbrace C_{20};C_{21}\rbrace +\mathrm{C}_d''$.

Let us check \eqref{Th O' no S_eps no conditions on Lambda's fluxes}. From \eqref{Th O' no S_eps no conditions on Lambda's}, \eqref{b_l <=}, and \eqref{K_D^0(t;eps)} it follows that
\begin{equation}
\label{6.24}
\begin{split}
\Vert &g^\varepsilon b(\mathbf{D})
\bigl(
f^\varepsilon e^{-\widetilde{B}_{D,\varepsilon}t}(f^\varepsilon)^*
-
(I+\varepsilon \Lambda ^\varepsilon  b(\mathbf{D})+\varepsilon\widetilde{\Lambda}^\varepsilon)
f_0 e^{-\widetilde{B}_D^0 t}f_0
\bigr)
\Vert _{L_2(\mathcal{O})\rightarrow L_2(\mathcal{O}')}
\\
&\leqslant \Vert g\Vert _{L_\infty}(d\alpha _1)^{1/2}\mathrm{C}_d\varepsilon h_d(\delta ;t) e^{-c_\flat t/2}.
\end{split}
\end{equation}

Let us apply identity \eqref{rem. S_eps 10}. The norm of the second summand in the right-hand side of \eqref{rem. S_eps 10} is estimated with the help of
\eqref{b_l <=}, \eqref{chi}, and Lemma~\ref{Lemma Lambda multiplicator properties}($1^\circ$):
\begin{equation}
\label{6.25}
\begin{split}
\varepsilon &\Bigl\Vert \sum _{k,j=1}^d  g^\varepsilon b_k \Lambda ^\varepsilon b_j D_kD_j f_0 e^{-\widetilde{B}_D^0t}f_0\Bigr\Vert _{L_2(\mathcal{O})\rightarrow L_2(\mathcal{O}')}
\\
&\leqslant
\varepsilon \alpha _1 \Vert g\Vert _{L_\infty}
C^{(0)}\sum _{k,j=1}^d \Vert \chi D_k D_j f_0 e^{-\widetilde{B}_D^0t}f_0\Vert _{L_2(\mathcal{O})\rightarrow  H^{l-1}(\mathbb{R}^d)},
\quad l=d/2.
\end{split}
\end{equation}
Next, we apply Lemma~\ref{Lemma chi}. If $l$ is integer, \eqref{f_0<=} yields
\begin{equation}
\label{6.26}
\begin{split}
&\sum _{k,j=1}^d \Vert \chi D_k D_j f_0 e^{-\widetilde{B}_D^0t}f_0\Vert _{L_2(\mathcal{O})\rightarrow  H^{l-1}(\mathbb{R}^d)}\\
&\qquad\leqslant
d C_{l-1}^{(11)}\Vert f\Vert _{L_\infty}\sum _{i=0}^{l-1} \delta ^{-(l-1-i)}\Vert f_0 e^{-\widetilde{B}_D^0 t}\Vert _{L_2(\mathcal{O})\rightarrow H^{i+2}(\mathcal{O}'')}.
\end{split}
\end{equation}
The norm $\Vert f_0e^{-\widetilde{B}_D^0 t}\Vert _{L_2(\mathcal{O})\rightarrow H^{2}(\mathcal{O})}$ satisfies \eqref{exp tilde B_D^0 L2-H2}.
If $i\geqslant 1$, relations \eqref{f_0<=} and \eqref{exp tilde B_D^0 L2-H-sigma strictly interior} (with $\mathcal{O}$ replaced by $\mathcal{O}''$) imply that
\begin{equation*}
\Vert f_0 e^{-\widetilde{B}_D^0 t}\Vert _{L_2(\mathcal{O})\rightarrow H^{i+2}(\mathcal{O}'')}
\leqslant \mathrm{C}_{i+2}'\Vert f\Vert _{L_\infty}2^{i+1}t^{-1/2}(\delta ^{-2}+t^{-1})^{(i+1)/2}e^{-c_\flat t/2}.
\end{equation*}
Combining this with \eqref{exp tilde B_D^0 L2-H2}, \eqref{6.25}, and \eqref{6.26}, we obtain
\begin{equation}
\label{6.28}
\begin{split}
\varepsilon \Bigl\Vert \sum _{k,j=1}^d g^\varepsilon b_k \Lambda ^\varepsilon b_j D_kD_j f_0 e^{-\widetilde{B}_D^0t}f_0\Bigr\Vert _{L_2(\mathcal{O})\rightarrow L_2(\mathcal{O}')}
\leqslant
C^{(17)}\varepsilon t^{-1/2}(\delta ^{-2}+t^{-1})^{d/4}e^{-c_\flat t/2},
\end{split}
\end{equation}
where the constant $C^{(17)}$ depends only on the problem data \eqref{problem data}.
If $l$ is half-integer, inequality \eqref{6.28} is checked by using Lemma~\ref{Lemma chi}($2^\circ$).

The third summand in the right-hand side of \eqref{rem. S_eps 10} is estimated similarly by using \eqref{b_l <=}, \eqref{chi}, Lemma~\ref{Lemma tilde Lambda multiplicator properties}($1^\circ$), and Lemma~\ref{Lemma chi}. As a result, we obtain
\begin{equation}
\label{6.29}
\begin{split}
\varepsilon \sum _{j=1}^d \Vert g^\varepsilon b_j \widetilde{\Lambda}^\varepsilon D_j f_0 e^{-\widetilde{B}_D^0 t}f_0\Vert _{L_2(\mathcal{O})\rightarrow L_2(\mathcal{O}')}
\leqslant C^{(18)}
\varepsilon  t^{-1/2}(\delta^{-2}+t^{-1})^{d/4}e^{-c_\flat t/2}.
\end{split}
\end{equation}
Here the constant $C^{(18)}$ depends only on the problem data \eqref{problem data}.

Finally, relations \eqref{tilde g}, \eqref{rem. S_eps 10}, \eqref{6.24}, \eqref{6.28}, and \eqref{6.29} imply
inequality  \eqref{Th O' no S_eps no conditions on Lambda's fluxes} with the constant
$\widetilde{\mathrm{C}}_d:=\Vert g\Vert _{L_\infty} (d\alpha _1)^{1/2}\mathrm{C}_d+C^{(17)}+C^{(18)}$.
\qed

\end{document}